\documentclass[11pt]{amsart}
\usepackage[
paper=a4paper ,
  hmarginratio={1:1},     
  vmarginratio={1:1},     
]{geometry}

\usepackage{graphicx}

\usepackage[hidelinks]{hyperref}
\usepackage{amsrefs}
\usepackage{amsmath,amssymb,mathtools,mathrsfs,amsthm}
\usepackage{tikz}
\usetikzlibrary{cd}

\usepackage{enumerate}
\usepackage{dsfont}

\usepackage{thmtools}
\usepackage{thm-restate}

\usepackage[margin=10px]{caption}

\usepackage{spectralsequences}

\newtheorem{theorem}{Theorem}[section]
\newtheorem*{theorem*}{Theorem}

\usepackage[section]{placeins}
\newtheorem{proposition}[theorem]{Proposition}
\newtheorem{corollary}[theorem]{Corollary}
\newtheorem{lemma}[theorem]{Lemma}
\newtheorem{conjecture}[theorem]{Conjecture}

\newtheorem{mainthm}{Theorem}
\newtheorem{maincorr}[mainthm]{Corollary}

\theoremstyle{definition}
\newtheorem{definition}[theorem]{Definition}

\newtheorem*{convention}{Convention}

\newtheorem{remark}[theorem]{Remark}

\numberwithin{equation}{section}

\definecolor{myo}{RGB}{255,196,0}
\definecolor{mpurp}{RGB}{255,0,255}

\definecolor{orange}{RGB}{255,128,0}
\definecolor{light-blue}{RGB}{0,128,255}
\definecolor{dark-green}{RGB}{0,170,0}

\DeclareMathOperator{\im}{im}
\DeclareMathOperator{\Wh}{Wh}
\DeclareMathOperator{\GL}{GL}
\DeclareMathOperator{\St}{St}
\DeclareMathOperator{\tr}{tr}
\DeclareMathOperator{\Diff}{Diff}
\newcommand{\id}{\mathds{1}}
\DeclareMathOperator{\coker}{coker}
\DeclareMathOperator{\E}{E}
\DeclareMathOperator{\Conj}{Conj}
\DeclareMathOperator{\Cl}{Cl}
\newcommand{\hTheta}{\widehat{\Theta}}

\makeatletter
\DeclareRobustCommand\widecheck[1]{{\mathpalette\@widecheck{#1}}}
\def\@widecheck#1#2{%
    \setbox\z@\hbox{\m@th$#1#2$}%
    \setbox\tw@\hbox{\m@th$#1%
       \widehat{%
          \vrule\@width\z@\@height\ht\z@
          \vrule\@height\z@\@width\wd\z@}$}%
    \dp\tw@-\ht\z@
    \@tempdima\ht\z@ \advance\@tempdima2\ht\tw@ \divide\@tempdima\thr@@
    \setbox\tw@\hbox{%
       \raise\@tempdima\hbox{\scalebox{1}[-1]{\lower\@tempdima\box
\tw@}}}%
    {\ooalign{\box\tw@ \cr \box\z@}}}
\makeatother

\let\oldtocsection=\tocsection
\let\oldtocsubsection=\tocsubsection
\let\oldtocsubsubsection=\tocsubsubsection

\renewcommand{\tocsection}[2]{\hspace{0em}\oldtocsection{#1}{#2}}
\renewcommand{\tocsubsection}[2]{\hspace{1em}\oldtocsubsection{#1}{#2}}
\renewcommand{\tocsubsubsection}[2]{\hspace{2em}\oldtocsubsubsection{#1}{#2}}

\title{Pseudo-Isotopies and diffeomorphisms of 4-manifolds}
\author{Oliver Singh}
\address{Department of Mathematical Sciences, Durham University, United Kingdom.}
\email{\href{mailto:oliver.singh@durham.ac.uk} {oliver.singh@durham.ac.uk}}

\begin{document}
\vspace*{-1.4cm}
\begin{abstract}

A diffeomorphism $f$ of a compact manifold $X$ is pseudo-isotopic to the identity if there is a diffeomorphism $F$ of $X\times I$ which restricts to $f$ on $X\times 1$, and which restricts to the identity on $X\times 0$ and $\partial X\times I$. We construct examples of diffeomorphisms of 4-manifolds which are pseudo-isotopic but not isotopic to the identity. To do so, we further understanding of which elements of the ``second pseudo-isotopy obstruction'', defined by Hatcher and Wagoner, can be realised by pseudo-isotopies of 4-manifolds. We also prove that all elements of the first and second pseudo-isotopy obstructions can be realised after connected sums with copies of $S^2\times S^2$.
\end{abstract}
\maketitle

\tableofcontents

\pagenumbering{arabic}

\section{Introduction}\label{pseudointro}
Let $X$ be a smooth compact manifold of dimension $n$. A \emph{pseudo-isotopy} of $X$ is a diffeomorphism $F\colon X\times [0,1]\rightarrow X\times[0,1]$ such that $F|_{X\times 0}$ and $F|_{\partial X\times [0,1]}$ are the identity. We let $\mathcal{P}=\mathcal{P}(X,\partial X)$ be the \emph{space of pseudo-isotopies of $X$}, which is a topological space when equipped with the $C^\infty$ topology. When $X$ is closed we write $\mathcal{P}(X) = \mathcal{P}(X,\emptyset)$.

Given diffeomorphisms $f,g\colon X\rightarrow X$ we say $f$ and $g$ are \emph{pseudo-isotopic} if there exists a pseudo-isotopy of $X$ such that $F|_{X\times 1}$ is $g^{-1}\circ f$. Thinking of an isotopy of $X$ as a map~$F\colon X\times [0,1]\rightarrow X\times[0,1]$ which is level preserving, that is $F(X\times t) = X\times t$ for every~$t\in[0,1]$, it is clear that if diffeomorphisms $f$ and $g$ of $X$ are isotopic, then they are pseudo-isotopic. 

The first aim of this paper is to extend the list of 4-manifolds for which the converse is known to be false. Denoting the subgroup of diffeomorphisms fixing the boundary which are pseudo-isotopic to the identity by~$\Diff_{PI}(X,\partial X)$, we construct non-trivial elements of $\pi_0\Diff_{PI}(X,\partial X)$ for certain 4-manifolds $X$.

\begin{restatable}{mainthm}{nontrivdiffeos}\label{nontrivdiffeos}
Let $X$ be either the 4-manifold $S^1\times S^2\times I$ or $(M_1\# M_2)\times I$, for~$M_1$,~$M_2$ closed, orientable, aspherical 3-manifolds. Then there is a subgroup~$K\leqslant \pi_0\Diff_{PI}(X, \partial X)$ and a surjective map
\[\Theta'\colon K\longrightarrow \bigoplus_{i\in \mathbb{N}}\mathbb{Z}.\]
Hence there are infinitely many distinct isotopy classes of diffeomorphisms of $X$ fixing the boundary, which are pseudo-isotopic to the identity.
\end{restatable}

Igusa points out in \cite{Igusanewnew}, that the map \[\pi_0\Diff_{PI}(M^{n-1}\times I,\partial (M^{n-1}\times I))\rightarrow \pi_0\Diff_{PI}(M^{n-1}\times S^1)\] induced by gluing the top and bottom of $M\times I$ together is injective \cite[Lemma~5.1]{Igusanewnew}. This allows us to state the below corollary.

\begin{maincorr}
Let $X$ be either the 4-manifold $S^1\times S^2\times S^1$ or $(M_1\# M_2)\times S^1$, for $M_1$, $M_2$ closed, orientable, aspherical 3-manifolds. Then there is a subgroup $K\leqslant \pi_0\Diff_{PI}(X)$ and a surjective map
$\Theta'\colon K\longrightarrow \bigoplus_{i\in \mathbb{N}}\mathbb{Z}.$
Hence there are infinitely many distinct isotopy classes of diffeomorphisms of $X$ which are pseudo-isotopic to the identity.
\end{maincorr}
%

To prove Theorem \ref{nontrivdiffeos}, we utilise the so called ``second pseudo-isotopy obstruction''~$\Theta$, taking values in $\Wh_1(\pi_1 X;\mathbb{Z}_2\times\pi_2 X)$, which was defined by Hatcher and Wagoner in \cite{HatcherWagoner} and refined by Igusa \cite{Igusa}. In dimensions $n\geq 5$, Hatcher \cite{Hatcher} uses~$\Theta$ to construct non-trivial elements of $\pi_0\Diff_{PI}(M^{n-1}\times I,\partial (M^{n-1}\times I))$. To extend this result to 4 dimensions we further understanding of which elements of~$\Wh_1(\pi_1 X;\mathbb{Z}_2\times\pi_2 X)$ are realised by $\Theta$ for 4-manifolds. This is the second aim of this paper.

Our third aim is to show that both $\Theta$ and the ``first pseudo-isotopy obstruction''
\[\Sigma\colon\pi_0\mathcal{P}(X,\partial X)\rightarrow\Wh_2(\pi_1 X)\]
are (in some sense) surjective in dimension 4 after taking connect sums of $X$ with copies of $S^2\times S^2$. Both invariants are surjective without any such stabilisation in dimension~$\geq 5$.

Before we state the rest of our main theorems, we recall some high-dimensional background. In high dimensions, pseudo-isotopies are classified up to isotopy by~$\Theta$ and $\Sigma$. Building on work of Hatcher and Wagoner in \cite{HatcherWagoner} and \cite{Hatcher}, Igusa shows in \cite{Igusa} that for $X$ a smooth manifold of dimension at least 6 there is a natural exact sequence
\[K_3\mathbb{Z}[\pi_1 X]\xrightarrow{\chi} \Wh_1(\pi_1 X; \mathbb{Z}_2\times \pi_2 X)\rightarrow \pi_0 \mathcal{P}(X,\partial X)\xrightarrow{\Sigma}\Wh_2(\pi_1 X)\rightarrow 0.\]
Igusa shows that if the first Postnikov invariant $k_1 X$ vanishes, then $\chi$ is 0 and the sequence splits, with splitting 
\[\Theta_\sigma\colon\pi_0 \mathcal{P}\rightarrow \Wh_1(\pi_1 X, \mathbb{Z}_2\times \pi_2 X)\]
dependent on a choice of section $\sigma\colon X_{(1)}\rightarrow X_{(2)}$, where $X_{(i)}$ is the $i$th stage in a Postnikov tower for $X$. The restriction of this map to $\ker\Sigma$
\[\Theta\colon\pi_0 \ker\Sigma\rightarrow \Wh_1(\pi_1 X, \mathbb{Z}_2\times \pi_2 X),\]
originally defined by Hatcher and Wagoner in \cite{HatcherWagoner}, does not depend on a choice of section. When $k_1 X\neq 0$ it follows from the constructions in \cite{Igusa} that there is a map 
\[\Theta\colon\ker\Sigma\rightarrow\Wh_1(\pi_1 X;\mathbb{Z}_2\times \pi_2 X)/\chi(K_3\mathbb{Z}[\pi_1 X]).\]

The maps $\Sigma$ and $\Theta$ (restricted to $\ker\Sigma$) are also defined in 4 and 5 dimensions. In dimension $n\geq 5$ Hatcher and Wagoner show that $\Sigma$ and $\Theta$ are surjective. However, they are not able to prove this in dimension 4, where the situation is less clear.

\begin{convention}
Throughout, 4-manifolds will be smooth, compact, and connected.
\end{convention}

We prove that in dimension 4, the following elements of the second obstruction group are realised.

\begin{restatable}{mainthm}{Whoneimage}
\label{Whoneimage}
For $X$ a compact 4-manifold, let 
\begin{align*} \Xi& =\left\langle (s+\sigma)\gamma\;|\;w_2^X(\sigma)\neq 0\text{ or }s=0,\ s\in \mathbb{Z}_2,\ \sigma\in\pi_2 X,\ \gamma\in\pi_1 X  \right\rangle \\
&\subset (\mathbb{Z}_2\times\pi_2 X )[\pi_1 X]/ \langle \alpha\gamma -\alpha^\tau \tau\gamma\tau^{-1}, \beta\cdot 1\; |\; \alpha,\beta\in\mathbb{Z}_2\times\pi_2 X,\; \tau,\gamma\in \pi_1 X\rangle \\
&= \Wh_1(\pi_1 X ;\mathbb{Z}_2\times\pi_2 X  ).
\end{align*}
If $k_1 X=0$ then $\Xi\subset\Theta(\ker\Sigma)$. Otherwise the same is true passing to the quotient $\Wh_1(\pi_1 X ;\mathbb{Z}_2 \times \pi_2 X )/\chi (K_3 \mathbb{Z}[\pi_1 X])$.
\end{restatable}

For the identification of $\Wh_1(\pi_1 X ;\mathbb{Z}_2\times \pi_2 X)$ with 
\[(\mathbb{Z}_2\times\pi_2 X )[\pi_1 X]/ \langle \alpha\sigma -\alpha^\tau \tau\sigma\tau^{-1}, \beta\cdot 1\;|\;\alpha,\beta\in\mathbb{Z}_2\times\pi_2 X,\; \tau,\sigma\in \pi_1 X\rangle \]
see Corollary \ref{whtrace}. The proof of Theorem \ref{Whoneimage} involves a detailed analysis of the $\Theta$ obstruction in terms of Whitney discs and framings of these discs in the 4-dimensional ``middle-middle level'' of 1-parameter families of handlebody structures. We believe this to be of independent interest; see Propositions \ref{claimone} and \ref{claimtwo}.

Jahren \cite{Jahren}, in an unpublished work, proves a similar theorem by different methods. Specifically he proves that all elements of $\Wh_1(\pi_1 X, \pi_2 X)$ are realised, and that when~$X$ contains an odd sphere, all elements of $\Wh_1(\pi_1 X, \mathbb{Z}_2)$ are realised. We obtain both of these results as a corollary of Theorem \ref{Whoneimage}.

\begin{corollary}\label{surjcorr}\leavevmode
\begin{enumerate}
\item For any 4-manifold $X$, $\Theta$ surjects onto $\Wh_1(\pi_1 X, \pi_2 X)$.
\item If $X$ is a 4-manifold with an odd sphere, that is $S\in\pi_2 X$ with $S\cdot S$ odd $($note that we do not require $S$ to be embedded$)$, then $\Theta$ surjects onto $\Wh_1(\pi_1 X;\mathbb{Z}_2)$.
\end{enumerate}
\end{corollary}
Corollary \ref{surjcorr} extends a recent result by Igusa in concurrent work \cite[Theorem~A, Theorem~B]{Igusanew}. Using different methods, Igusa proves Corollary \ref{surjcorr}.(2) with the additional requirement that $S$ is embedded. Igusa also proves that certain elements of $\Wh_1(\pi_1 X, \pi_2 X)$ are realised, namely ones where the element of $\pi_2 X$ can be represented by an embedded sphere.

%

We also prove that after stabilisation of $X$ with a single $S^2\times S^2$ we may realise any element of $\Wh_1(\pi_1 X;\mathbb{Z}_2\times \pi_2 X)$.

\begin{restatable}{mainthm}{whonestablethm}
\label{theorem:stablewh1realisation}
Let $X$ be a compact 4-manifold. Note that $\Wh_1(\pi_1 X ;\mathbb{Z}_2 \times \pi_2 X)$ includes in 
\[\Wh_1(\pi_1(X\# S^2\times S^2) ;\mathbb{Z}_2 \times \pi_2(X\# S^2\times S^2) ),\]
and identify $x \in \Wh_1(\pi_1 X ;\mathbb{Z}_2 \times \pi_2 X)$ with its image under this inclusion. There is a pseudo-isotopy $F$ of $X\# S^2\times S^2$, which is in~$\ker \Sigma$ such that 
\[\Theta (F)=x\in \Wh_1(\pi_1(X\# S^2\times S^2) ; \mathbb{Z}_2 \times\pi_2(X\# S^2\times S^2))/\chi(K_3\mathbb{Z}[\pi_1(X\# S^2\times S^2)]).\]
\end{restatable}

We also prove a stable version of Hatcher and Wagoner's surjectivity result for $\Sigma$.
\begin{restatable}{mainthm}{stablesurjectivity}\label{stablesurjectivity}

Let $X$ be a compact 4-manifold and $x\in\Wh_2(\pi_1 X)$. There exists $N$, and a pseudo-isotopy $F$ of $X\#^N S^2\times S^2$ such that
\[\Sigma(F)=x \in \Wh_2(\pi_1 (X\#^N S^2\times S^2))=\Wh_2(\pi_1 X).\]
\end{restatable}

As a consequence of Theorem \ref{Whoneimage} we are also able to construct diffeomorphisms of certain 5-manifolds.
\begin{restatable}{mainthm}{diffoffivemanifolds}\label{diffoffivemanifolds}
Suppose $X$ is a 4-manifold which contains an element $\sigma\in\pi_2(X)$ with $w_2^X(\sigma)\neq 0$, and an element $\gamma\in\pi_1 X$ such that $\gamma$ and $\gamma^{-1}$ are not conjugate, and suppose also that either $k_1 X=0$ or $K_3\mathbb{Z}[\pi_1X]=0$. Then in $\Diff(X\times I,\partial(X\times I))$ there exist diffeomorphisms pseudo-isotopic to the identity but not isotopic to it.
\end{restatable} 
This result is analogous to a result by Hatcher \cite[Corollary 4.5]{Hatcher}, which gives diffeomorphisms of manifolds of dimension greater than 6.

The other examples of diffeomorphisms of 4-manifolds which are pseudo-isotopic but not isotopic to the identity that we know of come from Budney-Gabai \cite{BG}, Watanabe \cite{Watanabe}, Igusa \cite{Igusanewnew}, and the following examples from gauge theory: Ruberman \cite{rubermanone}, \cite{rubermantwo}, Baraglia-Konno \cite{konno}, Kronheimer-Mrowka \cite{kronheimer}, and Lin \cite{lin}. Budney and Gabai construct diffeomorphisms of $S^1\times B^3$ and $S^1\times S^3$, and by related methods Watanabe constructs diffeomorphisms of $\Sigma(2,3,5)\times S^1$. Budney and Gabai's examples are pseudo-isotopic to the identity by Sato \cite{sato} and Lashof-Shaneson \cite{lashof}, while Watanabe's examples are pseudo-isotopic to the identity by \cite[Theorem 1.8 and Theorem 9.3]{Watanabe}, which Watanabe attributes to Teichner. In a concurrent work, posted subsequently to the initial posting of this paper, Igusa \cite{Igusanewnew} constructed diffeomorphisms of $((S^1\times S^2)\# M)\times S^1$ for $M$ a non-simply connected 3-manifold, also using the second obstruction to pseudo-isotopy.
The gauge theoretic examples on the other hand give diffeomorphisms of simply connected 4-manifolds which are not isotopic to the identity but induce the identity on homology; by Kreck \cite[Theorem 1]{kreck}, these diffeomorphisms are pseudo-isotopic to the identity.

Other work that has been done on pseudo-isotopies of 4-manifolds includes the work of Quinn \cite{Quinn} who proves that topological pseudo-isotopies of simply connected 4-manifolds are topologically isotopic rel boundary to an isotopy. Hence in the topological setting isotopy and pseudo-isotopy are the same for simply connected 4-manifolds.

Kwasik \cite{Kwasik} shows that the same is not true for non simply connected 4-manifolds, and asserts that if topological pseudo-isotopies are allowed then $\Theta$ is surjective. However, it is unclear to us how to define $\Theta$ and $\Sigma$ in the topological setting, and unclear how to show it is well defined, particularly as current definitions are heavily reliant on Cerf theory. It appears to us that a topological definition of both $\Theta$ and $\Sigma$ would be of value, and could facilitate further results in the topological setting.

\subsection*{Acknowledgements}
Thanks to Mark Powell and Andrew Lobb for extensive discussions and ideas. Thanks to Arunima Ray and Dirk Schuetz for useful comments and corrections. Thanks to Wolfgang L\"uck for useful correspondence, Bj{\o}rn Jahren for his interest and for sharing his notes, and to Slava Krushkal for interesting conversations. Thanks also to the Max Planck Institute for Mathematics for hosting me for three months while I worked on this paper. This research was supported by EPSRC Grant no.\ EP/N509462/1 project no.\ 1918079, awarded through Durham University.


\section{Background}

We first recall some general topology background. We will review some standard tools of Morse theory, we will also review Whitney discs and operations on Whitney discs, as well as Postnikov towers and $k$-invariants.

As is standard in Morse theory and Cerf theory, we fix a Riemannian metric $\mu$ on $X$ and take the product metric on $X\times I$. 

\begin{remark} The choice of metric is not important, and in fact it is possible to avoid fixing a metric entirely. Indeed Hatcher and Wagoner consider a space $\widehat{\mathcal{F}}(X,\partial X)$ which includes all possible choices of metric, and show that it is homotopy equivalent to the space~$\mathcal{F}(X,\partial X)$ which we define in Section \ref{functional-approach}; see \cite[Chapter 1, \S 3]{HatcherWagoner}.
\end{remark}

\subsection{Morse theory and the Morse chain complex}\label{morsetheorysection}
Let $f$ be a Morse function
\[f\colon X\times I \rightarrow I\]
such that $f|_{\partial X\times I}$ is the projection to $I$ and that $f(X\times i)=i$ for $i\in\{0,1\}$. Suppose also that $f$ has no critical points near $X\times\partial I$ or $\partial X\times I$. By choice of Riemannian metric we have a uniquely defined gradient $\nabla f$. Given $s\in \mathbb{R}$ we denoted the \emph{flow} of $f$ by $\phi_{f,s}\colon X\times I \rightarrow X\times I$, where $\phi_{f,s}(p)$ is giving by pushing $p$ along $\nabla f$ for time $s$. 

Recall that the \emph{trajectory} of a point $p$ is given by $\{\phi_{f,s}(p) | s\in\mathbb{R}\}$, and that trajectories are embedded copies of $(0,1)$, $(0,1]$, $[0,1)$ or $[0,1]$, with any boundary points lying in~$X\times \partial I$. Recall also that the limit point on approaching the open end of a trajectory is always a critical point. Given a critical point $p$ we denote the \emph{stable} set by 
\[W_f(p) = \{q\; |\; \lim_{s\rightarrow +\infty} \phi_{f,s}(q) = p\}\]
and the \emph{unstable} set by 
\[W^\star_f(p) = \{q\ |\ \lim_{s\rightarrow -\infty} \phi_{f,s}(q) = p\}.\]

Given two critical points $p,q\in X\times I$ of index $i$ and $j$ respectively, let 
\[T^q_p =\left\{\left\{\phi_{f,s}(a)\right\}_{s\in\mathbb{R}} \;| \; a\in W(p)\cap W^\star(q)\right\}\]
be the set of trajectories from $q$ to $p$.  We refer to these trajectories as \emph{$j/i$ trajectories} (the $j$ being on top of the fraction indicates that is also ``on top'' in the manifold, since a~$j/i$ trajectory goes down from a critical point of index $j$ to one of index $i$). If $p, q$ are critical points of index $i$ and $i+1$ respectively and $f$ is a Morse function in general position, a dimension count shows that $W(p)\cap W^\star(q)$ is a collection of isolated arcs and so there are finitely many trajectories from $q$ to $p$.

\subsubsection{The Morse chain complex}
We can capture the data of the $\frac{i}{i+1}$ intersections in a chain complex, the \emph{Morse chain complex}, defined as follows. Let $\psi:\widetilde{X}\rightarrow X$ be the universal cover of $X$, and let $f$ be a self indexing Morse function on X, i.e.\ there exists $r_0=0<r_1<\cdots <r_{n+1}=1\in [0,1]$ with all critical points of index $i$ having critical values between $r_i$ and $r_{i+1}$. 
Let $V_i=f^{-1}([r_i,r_{i+1}])$, and let $\widetilde{V_i}=\psi^{-1}(V_i)$. Define the chain complex by 
\[C_i(X)=H_i(\widetilde{V_i},\widetilde{V}_{i-1})\]
which is a finitely generated free module over $\mathbb{Z}[\pi_1(X)]$ with a generator for each critical point of index $i$.

Fix a base point $p\in X$ and a lift of the base point $\tilde{p}\in\widetilde{X}$, and for each critical point $c$ of $f$ pick a path $\gamma_c$ from $p$ to $c$. This gives a basis for $C_i$ as a finitely generated free $\mathbb{Z}[\pi_1(X)]$ module. We can describe the differential as follows. Let $S_i$ be the collection of index $i$ critical points, recall that $T^a_b$ is the set of trajectories from $a$ to $b$. Using the basis above, this determines a matrix 
$\left(\partial_i^{f}\colon C_{i}\rightarrow C_{i-1}\right)$
given by
\[\partial_i^{f}(c)=\sum_{b\in S_{i-1}}\sum_{\psi\in T^c_b} [\gamma_c\cdot\psi\cdot\gamma_b^{-1}]b.\]
\subsection{Homotopies of surfaces in 4-manifolds and Whitney discs}\label{subsec:Homotopy}

We will wish to study deformations of immersed surfaces in 4-manifolds, and so recall some standard techniques for doing so; for full details see \cite{Freedman1990}.

Let $X$ be a smooth 4-manifold. Given a compact immersed surface $S\subset X$ (with possibly many components), recall that a \emph{regular homotopy} of this surface in $X$ is a homotopy through immersions; i.e.\ a smooth map $H\colon \Sigma \times I\rightarrow X$ such that $H(-,0)$ maps $\Sigma$ to the surface $S$, and $H(-,t)$ is an immersion for all $t$. We frequently omit the map and discuss only the immersed surface, which we denote $S_t = H(\Sigma,t)$.

Given an immersed surface, by standard general position arguments we can perturb it so that it has only finitely many isolated self-intersections, and that these self-intersections are transverse double points; that is the immersion is 2 to 1 at these points (as opposed to $n$ to 1 for $n>2$).

By general position we may perturb a regular homotopy so that for each $t$, $S_t$ has only transverse double-points, except at finitely many values of $t$. At these values of $t$ we see a non-transverse point; near these values, the homotopy can always be described as one of the following local moves; see \cite{Freedman1990} for further details.

\begin{definition}\label{fingermovedef}
Let $S\subset X$ be an immersed oriented surface, and let $\gamma$ be an arc such that the endpoints are in $S$, but the interior is disjoint from $S$. Then we may perform a \emph{finger move}, pushing one sheet of the surface along this arc into the other to introduce two points of transverse intersection; see the first image in Figure \ref{fig:finger-whitney-def}.
\end{definition}
\begin{convention}
It will be useful to us to insist that the finger move arc $\gamma$ in Definition \ref{fingermovedef} be an oriented arc (one may freely pick the orientation, we just require that it is kept track of).
\end{convention}

\begin{figure}[!ht]
   \centering
\begin{tikzpicture}
\node[anchor=south west,inner sep=0] (image) at (0,0) {\includegraphics[width=0.55\textwidth]{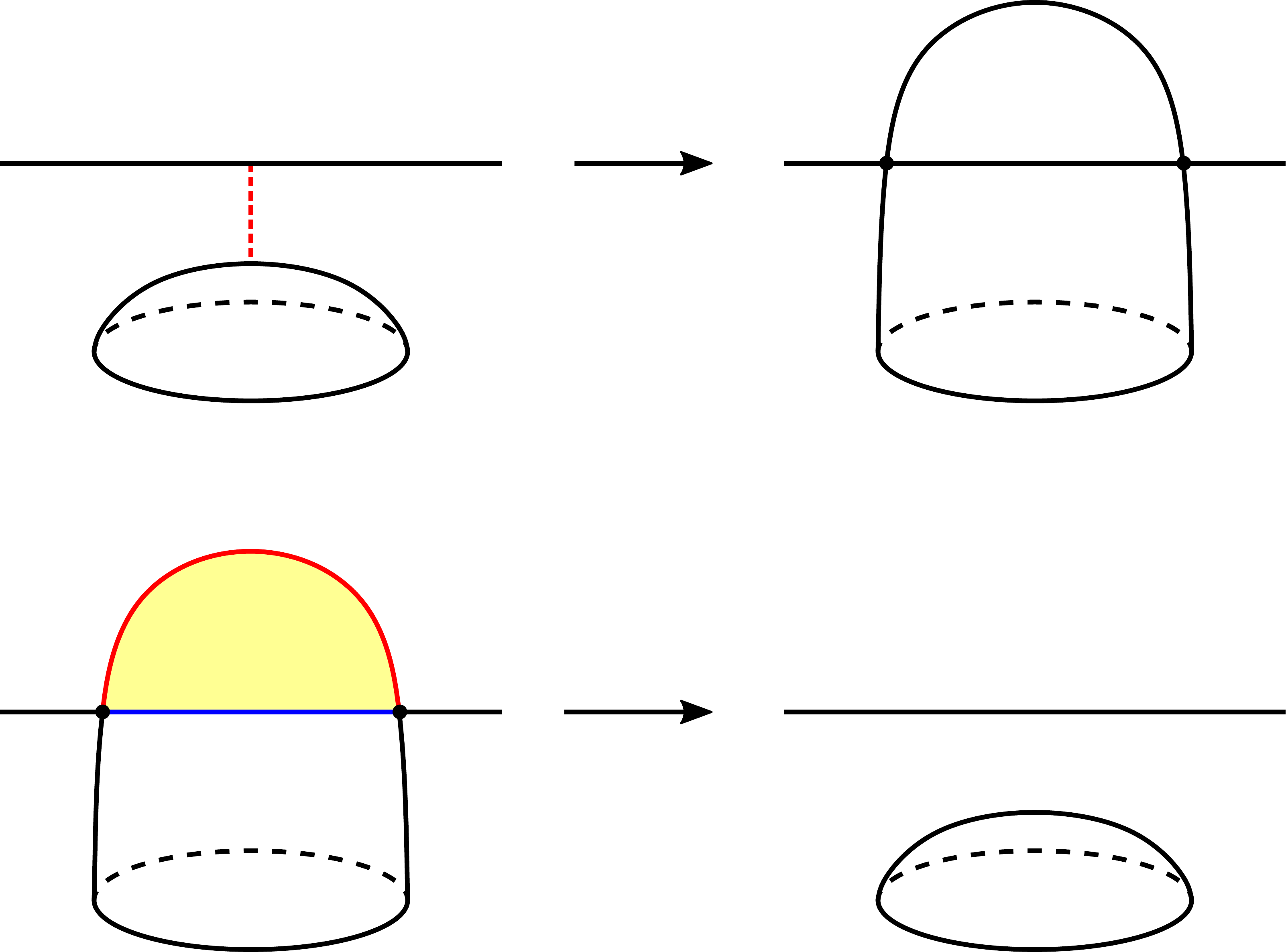}};
  \begin{scope}[x={(image.south east)},y={(image.north west)}]
    \node[text=red] at (0.195,0.445) {$\alpha$};
    \node[text=blue] at (0.195,0.215) {$\beta$};
    \node at (0.048,0.29) {$z^+$};
    \node at (0.352,0.29) {$z^-$};
    \node at (0.195,0.34) {$W$};
    \node[red] at (0.22,0.77) {$\gamma$};
    \node at (0.5,0.785) {Finger};
    \node at (0.5,0.785-0.046) {move};
    \node at (0.5,0.335) {Whitney};
    \node at (0.5,0.335-0.046) {move};
    \end{scope}
\end{tikzpicture}
\caption{\label{fig:finger-whitney-def} The time slice where the finger moves and Whitney moves occur. Note that the horizontal line continues into the past and future, and is unchanged by the homotopies.}
\end{figure}
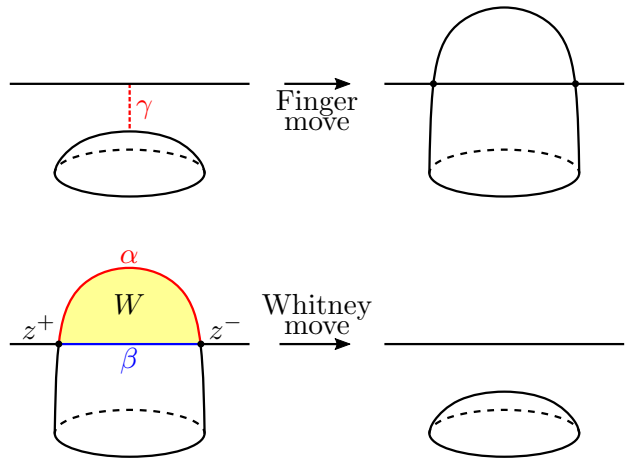

\begin{remark}\label{pi1finger}
Let $S$ be a collection of spheres embedded in $X$, and suppose for each sphere we have an arc to a basepoint~$\ast\in X$. Then we can associate an element of $\pi_1(X)$ to any finger move arc  by going along a basepoint arc, along some arc in $S$ to the base of the finger move arc, along the finger move arc, along another arc in $S$ to the basepoint arc, then backwards along another basepoint arc (possibly the same basepoint arc if the endpoints of the finger move arc lie on the same component of $S$). We call this the element of $\pi_1(X)$ \emph{associated} to the finger move. If additionally $S$ is $\pi_1$-negligible (that is $\pi_1(X) = \pi_1(X\setminus S)$) then the finger move is uniquely determined by its associated element of $\pi_1(X)$; indeed for any $\gamma\in\pi_1(X)$ there always exists some choice of finger move arc with associated element $\gamma$ and any two choices can easily be shown to be homotopic in $X$, and hence in $X\setminus S$ since $S$ is $\pi_1$-negligible.
\end{remark}

\begin{remark}\label{reversingorientation}
Note that reversing the orientation of $\gamma$ results in an isotopic finger move, and inverts the associated element of $\pi_1 X$. Later, the surfaces between which we perform finger moves will have a natural order, so it will be useful to consider a finger move ``from'' one surface ``to'' another.
\end{remark}

\begin{remark}\label{localorientations}
Given an oriented immersed surface $S\subset X$ we wish to assign a sign to the double points of $S$. When $X$ is oriented this is easy, however when $X$ is non-orientable we must work a little harder. In this case, we restrict to the case that $S$ is $\pi_1$-trivial; that is $\pi_1 S$ includes as $0$ in $\pi_1 X$. Let $S=\cup_i S_i$, where $S_i$ are the connected components. We pick a basepoint $\ast\in X$ and arcs from $\ast$ to each component $S_i$. We pick a local orientation at $\ast$ then transport the orientation at $\ast$ along the arcs to give an orientation of the total space of $\nu(S_i,X)$ for each $i$. This allows us to define signs of intersection as for an oriented manifold. Note that the signs of the intersections depend on our choices of arcs, but that for any two points $p,q\in S_i\cap S_j$ the dichotomy of these having the same sign or a different sign does not depend on our choices.
\end{remark}

We now recall the definition of a Whitney disc, and a Whitney move.

\begin{definition}
Let $S\subset X$ be an oriented immersed surface (and that $S$ is $\pi_1$-trivial if $X$ is non-orientable). Suppose $z^+, z^-$ are double points of $S$ with sign $+$ and $-$ respectively. Suppose $\alpha,\beta\subset S$ are embedded arcs, both of which have endpoints $z^+$ and $z^-$, and which are disjoint except at these endpoints. Suppose also that $\alpha$ and $\beta$ do not intersect any intersection points of $S$ except at their endpoints. We also require that the endpoints of $\alpha$ live in different sheets of $S$ to the endpoints of $\beta$ of $S$; that is, letting $S$ be the image of some immersion $f\colon \Sigma\rightarrow X$ we require that $f^{-1}(\alpha)$ and $f^{-1}(\beta)$ are disjoint including at the endpoints.
	
If $W\subset X$ is an immersed disc with boundary $\alpha\cup\beta$ we call $W$ a \emph{Whitney disc}. We require that $W$ meets $S$ transversely on $\partial W$. We call $\alpha$ and $\beta$ the \emph{Whitney arcs}.
\end{definition}

\begin{definition}\label{def:cfwd}
Given a Whitney disc $W\subset X$, we pick an orientation for $W$. The bundle $\nu(W, X)$ is a 2-dimensional bundle over a disc and so is trivial and has a unique trivialisation determined by the orientation. This trivialisation induces a trivialisation, or \emph{framing}, of  $\nu(W,X)|_{\partial W}\cong S^1\times D^2$ which we call the \emph{disc framing}.
	
	There is another framing of $\nu(W,X)|_{\partial W}$ defined as follows. We define a section of~$\nu(W,X)|_{\partial W}$, denoted~$s\colon \partial W\rightarrow \nu(W,X)_{\partial W}$, by requiring that $s$ be parallel to $S$ along $\beta$, and normal to $S$ along $\alpha$; note that $W$ is transverse to $S$ on $\partial W$ so this uniquely determines $s$ up to homotopy. We call $s$ the \emph{Whitney section}. Now $s$ is a section of a 2-dimensional bundle over $S^1$, and so uniquely determines a trivialisation (up to orientation, we use the orientation on $\nu(W,X)|_{\partial W}$ induced by the orientation of~$W$ and the orientation of $X$ when $X$ is oriented; when $X$ is non orientable we use a local orientation as in Remark \ref{localorientations}). We call this trivialisation the \emph{Whitney framing}.
	
	If the Whitney framing and the disc framing agree (up to isotopy), we say that $W$ is a \emph{correctly framed Whitney disc}, or simply a \emph{framed Whitney disc}.
\end{definition}

\begin{remark}\label{Zframings}

If $\nu(W,X)|_{\partial W}$ is oriented, then if we fix a framing $\nu(W,X)|_{\partial W}\cong S^1\times D^2$, then any other framing with the same orientation gives an element of $\pi_1\GL_2(\mathbb{R})=\mathbb{Z}$, and this element of $\mathbb{Z}$ is canonical. Hence given two framings we have a well defined difference between them in $\mathbb{Z}$, using the first framing to identify $\nu(W,X)|_{\partial W}\cong S^1\times D^2$ and the second to give the element of $\pi_1\GL_2(\mathbb{Z})=\mathbb{Z}$.
\end{remark}

\begin{remark}
Equivalently and more succinctly one may say a Whitney disc is framed if the Whitney section extends to a section of $\nu(W,X)$. However, as we will perform moves to change the disc framing and the Whitney framing separately, it will be useful for us to consider both framings independently.
\end{remark}

We now describe the second local move.

\begin{definition} Given a correctly framed, embedded Whitney disc, as in Definition \ref{def:cfwd}, whose interior is disjoint from $S$, we may perform a \emph{Whitney move}, a regular homotopy removing the two double points $z^+$ and $z^-$; see Figure \ref{fig:finger-whitney-def}. 
\end{definition}

\subsubsection{Interior twists and boundary twists of discs}

Suppose that we have an immersed Whitney disc $W$ which pairs two intersections of some surface $S$ ($S$ is possibly disconnected, and $W$ may intersect $S$ away from the Whitney arcs). As in Remark \ref{Zframings} we denote the difference between the Whitney framing and the disc framing by $n_W\in\mathbb{Z}$. There are two operations to create a new Whitney disc described in \cite[Section 1.3]{Freedman1990}: the interior twist and the boundary twist. We recall briefly these operations and their effects on $n_W$ below. 

A \emph{positive interior twist} is an operation which alters $W$ in a small neighbourhood in the interior of $W$. In this neighbourhood we add an additional positive self-intersection to $W$. We call the resulting disc $W'$. Note that $\nu(W,X)|_{\partial W}=\nu(W',X)|_{\partial W'}$, and that the Whitney framing does not change. By perturbing $W$ we can make $W'\cup -W$ immersed, one can easily see that $e(\nu(W'\cup -W,X)=2$, and so the disc framings of $W$ and $W'$ must differ by 2 (considering the clutching construction of bundles over spheres). Hence $n_{W'}=n_W + 2$.

Similarly a \emph{negative interior twist} introduces an additional negative self-intersection and $n_{W'}=n_W - 2$.


A \emph{positive boundary twist} alters $W$ only in a neighbourhood of one Whitney arc; note that we may choose which arc. It introduces an additional single positive intersection between $W$ and $S$ by twisting $W$ positively around $S$ along this arc; see \cite[Section 1.3]{Freedman1990}. We call the result $W'$. In this case $n_{W'}=n_W + 1$.

Similarly a \emph{negative boundary twist} introduces a additional single negative intersection between $W'$ and $S$, and $n_{W'}=n_W - 1$.

We will frequently use these moves to obtain a correctly framed Whitney disc from one which is not correctly framed.

\subsubsection{Pushing down}\label{section:push-down}
Let $W$ be a Whitney disc in a 4-manifold $V$ which pairs two intersections of some surface $A\subset X$. If $W$ intersects some surface $U$ in some point $p$, we may remove the intersection between $W$ and $U$ by \emph{pushing down} the intersection into $A$. To do this we perform a finger move between $U$ and $A$, resulting in two intersections between $U$ and $A$; see Figure \ref{fig:pushingdown}.

\begin{figure}[ht]
	\begin{tikzpicture}
		\node[anchor=south west,inner sep=0] (image) at (0,0) {\includegraphics[width=0.9\textwidth]{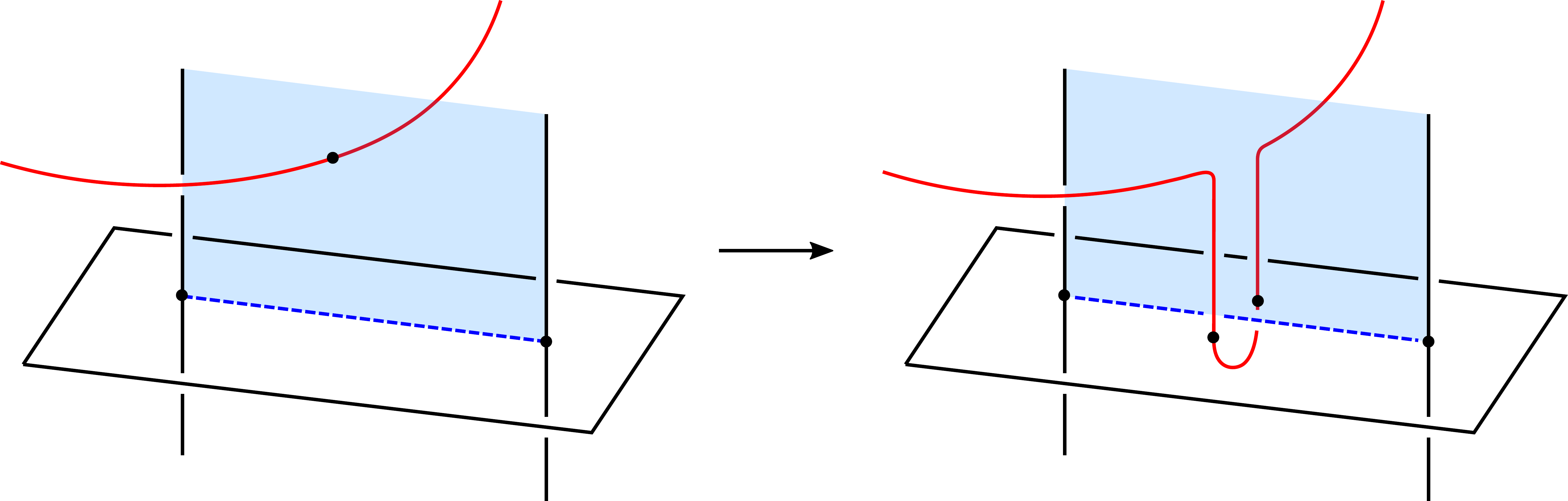}};
		\begin{scope}[x={(image.south east)},y={(image.north west)}]
			\node[text=light-blue]  at (0.37,0.7)  {$W$};
			\node[text=blue]  at (0.23,0.32)  {$\alpha$};
			\node at (0.42,0.25)  {$A$};
			\node[text=red] at (0.05,0.73)  {$U$};
			
		\end{scope}
	\end{tikzpicture}
	\caption{\label{fig:pushingdown}A depiction of the pushing down operation, which turns an intersection between $U$ and $W$ into two intersections between $U$ and $A$.}
\end{figure}
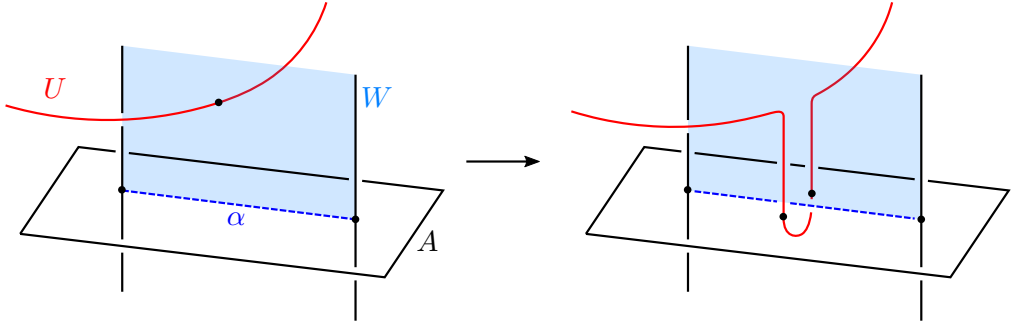

We can also push down to trade self intersections of $W$ for intersections between $W$ and $A$.

\subsubsection{Transverse spheres and the Norman trick}

Let $A$ be a surface in a 4-manifold $V$, and let $A^\ast\subset V$ be a (possibly non-embedded) sphere with trivial normal bundle which intersects $A$ transversely in a single point. We say $A^\ast$ is a \emph{transverse sphere} for $A$.

Given such a surface with a transverse sphere, if $A$ intersects some other surface $U$ in some point $p$, we may find a surface $U'$ with one fewer intersection with $A$, by taking an embedded arc $\gamma\subset A$ from $p$ to $A\cap A^\ast$, and tubing $U$ to a parallel copy of $A^\ast$ along the arc $\gamma$; see Figure \ref{fig:normantrick}. This operation is called the \emph{Norman trick}.

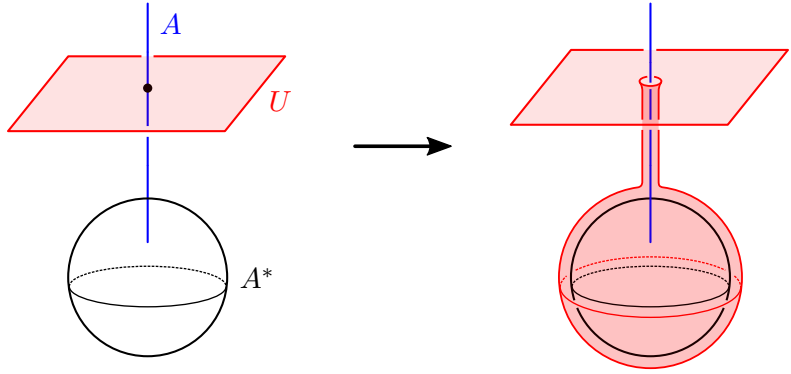
\begin{figure}[ht]
	\begin{tikzpicture}
		\node[anchor=south west,inner sep=0] (image) at (0,0) {\includegraphics[width=0.7\textwidth]{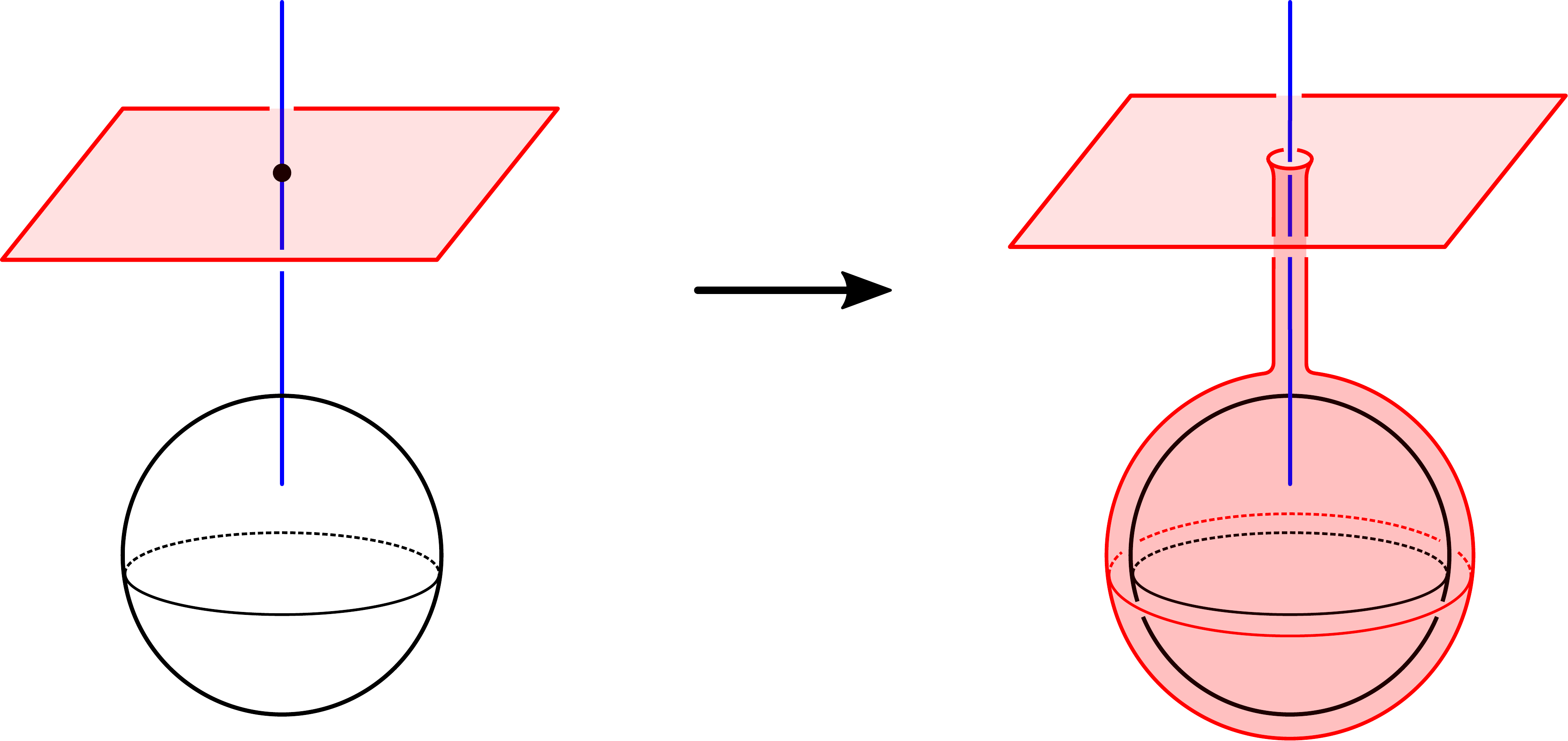}};
		\begin{scope}[x={(image.south east)},y={(image.north west)}]
			\node[text=red]  at (0.35,0.73)  {$U$};
			\node[text=blue]  at (0.21,0.94)  {$A$};
			\node at (0.32,0.25)  {$A^*$};
		\end{scope}
	\end{tikzpicture}
	\caption{\label{fig:normantrick}Using the transverse sphere $A^*$ to perform the Norman trick, removing the intersection between $A$ and $U$.}
\end{figure}

\subsection{The \texorpdfstring{$\mathbb{Z}[\pi_1 X]$}{Z[pi1X]} intersection number of spheres}\label{algebraic-intersection}
Let $A$ and $B$ be immersed spheres in a 4-manifold which intersect transversely (or more generally any two connected $pi_1$-trivial surfaces). Suppose also that we have chosen a basepoint $\ast\in X$, and that we have chosen arcs $\alpha,\beta\subset X$ from $\ast$ to $A$ and $B$ respectively; denote the endpoints at which these arcs meet $A$ and $B$ by~$\ast_A\in A$ and $\ast_B\in B$ respectively. Then given $p\in A\cap B$ we can assign a value in~$\pm\pi_1 X$ to this intersection point as follows. The sign comes from the sign of the intersection; in the case that $X$ is non-orientable this is as in Remark \ref{localorientations}. The element of $\pi_1 X$ comes from taking an arc $a\subset A$ from $\ast_A$ to $p$ and an arc $b\subset B$ from $\ast_B$ to $B$, then $\alpha\cdot a\cdot b^{-1}\cdot\beta^{-1}$ gives the element of $\pi_1 X$; see Figure~\ref{fig:algebraic-intersection}.
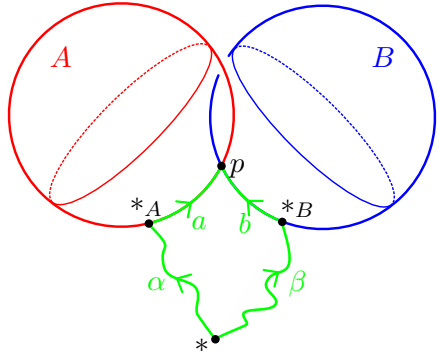
\begin{figure}[htb]
	\centering
	\begin{tikzpicture}
		\node[anchor=south west,inner sep=0] (image) at (0,0) {\includegraphics[width=0.45\textwidth]{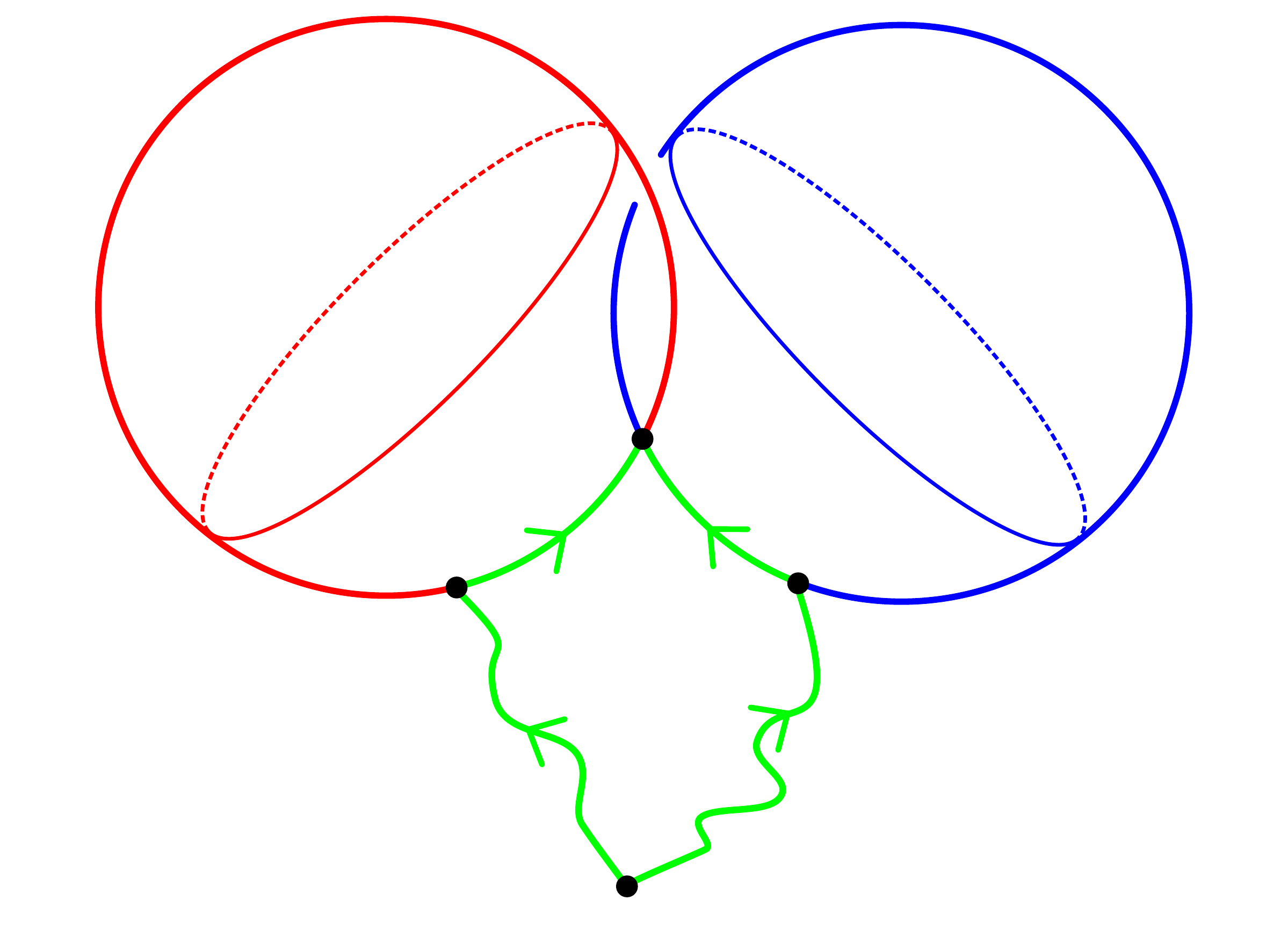}};

		\begin{scope}[x={(image.south east)},y={(image.north west)}]
			\node[text=red]  at (0.18,0.83)  {$A$};
			\node[text=blue]  at (0.82,0.83)  {$B$};
			\node  at (0.53,0.52)  {$p$};
			\node  at (0.35,0.415)  {$\ast_{A}$};
			\node  at (1-0.35,0.41)  {$\ast_{B}$};
			\node at (0.46,0.03) {$\ast$};
			\node[text=green] at (0.455, 0.37) {$a$};
			\node[text=green] at (1-0.455, 0.37) {$b$};
			\node[text=green] at (0.37, 0.2) {$\alpha$};
			\node[text=green] at (0.65, 0.2) {$\beta$};
		\end{scope}
	\end{tikzpicture}
	\caption{\label{fig:algebraic-intersection}The loop $\alpha\cdot a\cdot b^{-1}\cdot\beta^{-1}$.}
\end{figure}

We permit the arcs $a$ and $b$ may run over the intersection points and self-intersection points of $A$ and $B$, but they may not change sheets of the surface at these double points; equivalently, one must be able to make $a$ and $b$ disjoint from the intersection and self-intersection points. Note that since $A$ and $B$ are $\pi_1$-trivial $\alpha\cdot a\cdot b^{-1}\cdot\beta^{-1}$ is independent of the choice of $a$ and $b$.

We can define an intersection number $A\cdot B\in \mathbb{Z} [\pi_1 X]$ by summing over all $p\in A\cap B$. Note that this intersection number depends on the choice of arcs $\alpha$ and $\beta$ from $A$ and $B$ to $\ast$; when we wish to consider intersection numbers in $\mathbb{Z} [\pi_1 X]$ we will fix some choice of arcs to the basepoint.
%
%
\subsection{Postnikov towers and \texorpdfstring{$k$}{k}-invariants}\label{postnikovdef}

Since many of our theorems refer to the first Postnikov invariant (or $k$-invariant) $k_1 X$, we recall the basics of Postnikov towers and $k$-invariants below. For a more detailed treatment we direct the reader to \cite[pg. 421–437]{whitehead}.

Recall that an \emph{Eilenberg–MacLane space} $K(\pi,n)$ is a space whose $n$th homotopy group is isomorphic to $\pi$, with all other homotopy groups trivial.

\begin{definition}
Given a path connected space $X$, suppose we have spaces 
\[X_{(0)}, X_{(1)},\ldots, X_{(n)},\ldots\] 
and maps $p_n\colon X_{(n)}\rightarrow X_{(n-1)}$, $f_n\colon X\rightarrow X_{(n)}$, such that

\begin{enumerate}
\item the diagram below commutes.
\[
\begin{tikzcd}[column sep=huge, row sep=1.5em]
\centering
& \vdots\arrow[d, "p_{n+1}"]\\
& X_{(n)} \arrow[d, "p_n"]\\
& \vdots \arrow[d, "p_2"]\\
& X_{(1)} \arrow[d, "p_1"]\\
X \arrow[r, "f_0"]
\arrow[ru, "f_1",end anchor=south west ]
\arrow[ruu, "\vdots", draw=none]
\arrow[ruuu, "f_n",end anchor={[yshift=1ex]south west}] & X_{(0)}
\end{tikzcd}
\]
\item The induced map $(f_n)_\ast\colon\pi_i(X)\rightarrow \pi_i(X_{(n)})$ is an isomorphism for $i\leq n$, and that~$\pi_i(X_{(n)})=0$ for $i>n$. 
\item The map $p_n\colon X_{(n)}\rightarrow X_{(n-1)}$ is a fibration, with fiber a $K(\pi_n(X),n)$ space.
\end{enumerate}
We call such a tower of spaces and maps a \emph{Postnikov system} (or \emph{Postnikov tower}) and we call $X_{(n)}$ the Postnikov $n$-type of $X$.
\end{definition}

Postnikov towers always exist for $X$ a connected CW-complex, and uniquely determine $X$ up to weak homotopy; see \cite[Chapter XI]{whitehead}.

For each $n$ consider the fibration
\[K(\pi_{n+1}X,n+1)\longrightarrow X_{(n+1)}\xrightarrow{p_{n+1}} X_{(n)}.\]
there is a well defined homology class $k_n\in H^{n+2}(X_{(n)}; \pi_{n+1} X)$ which classifies this fibration, see \cite[pg. 421–437]{whitehead}. In particular there exists a section of this fibration $X_{(n)}\rightarrow X_{(n+1)}$ if and only if $k_n = 0$. We call $k_n$ the \emph{$n$th k-invariant}. Note that our indexing is different to that of \cite{whitehead}, in order to agree with the indexing of Igusa \cite{Igusa}.

We are particularly interested in $k_1 X$. Since $\pi_n(X_{(1)})=0$ for $n>1$, in fact $X_{(1)}$ is itself an Eilenberg–MacLane space, $X_{(1)}=K(\pi_1 X, 1)$, hence
\[k_1 X\in H^3(K(\pi_1 X,1),\pi_2 X) = H^3(\pi_1 X;\pi_2 X)\]
where $H^3(\pi_1 X;\pi_2 X)$ is group cohomology with coefficients twisted by the action of $\pi_1 X$ on $\pi_2 X$.

\section{Functional approach to pseudo-isotopies}\label{functional-approach}
We begin this section with a review of Cerf's view of pseudo-isotopies as paths of Morse functions as set out by Hatcher and Wagoner in \cite[Chapter 1, \S 2]{HatcherWagoner}.

Let $I=[0, 1]$ and let $\mathcal{F} = \mathcal{F}(X,\partial X)$ be the space of $C^\infty$ functions $f\colon X\times I\rightarrow I$ such that $f(x,0)=0$, $f(x,1)=1$ $\forall x$, and such that $f$ has no critical points near $X\times 0$, $X\times 1$ or~$\partial X\times I$. Let $\mathcal{E}\subset\mathcal{F}$ be the subset of all such functions with no critical points.

Denote the standard projection to $I$ by $p\colon X\times I\rightarrow I$. We define a map
\begin{align*}
\Pi  \colon\mathcal{P}&\longrightarrow \mathcal{F} \\
  F &\longmapsto p\circ F
\end{align*}

Since any $F\in\mathcal{P}$ is a diffeomorphism, $p\circ F$ has no critical points so
\[\Pi(\mathcal{P})\subset \mathcal{E}.\]
In fact $\Pi(\mathcal{P}) = \mathcal{E}$. Given $f\in \mathcal{E}$ we construct $F\in \mathcal{P}$ with $p\circ F = f$ by defining
\[F(x,s)=\phi_{f, s}(x,0).\]
Further, $\Pi$ is a fibration
$\mathcal{I}\rightarrow\mathcal{P}\xrightarrow{\Pi} \mathcal{E}$
with fiber
\[\mathcal{I} = \{F\colon X\times I\rightarrow I \;|\; F|_{X\times 0} = \mathds{1}_{X\times 0},\; F(X\times t) = X\times t\;\forall t\in I\}\]
that is, the space of isotopies of $X$ fixing $X\times 0$. The fiber $\mathcal{I}$ is contractible via
\begin{align*}
H_s\colon \mathcal{I}&\rightarrow\mathcal{I}\\
H_s(F)(x,t) &= F(x, (1-s)t).
\end{align*}
Hence $\Pi$ is a homotopy equivalence. Additionally $\mathcal{F}$ is contractible, so we have
\[\pi_0\mathcal{P}=\pi_0\mathcal{E}=\pi_1(\mathcal{F},\mathcal{E}).\]
In order to measure whether a given pseudo-isotopy $F$ is isotopic to the identity, our strategy will be to join $\Pi(F)$ to $p$ by a path $f_t\in\mathcal{F}$, and try to deform this path, fixing the ends, to lie in $\mathcal{E}$; if we succeed then the path $f_t$ is the trivial element of $\pi_1(\mathcal{P},\mathcal{E})$, so~$F$ is the trivial element of $\pi_0\mathcal{P}$. Conversely, our obstructions will be obstructions to finding such a path deformation, and so obstruct $F$ from being the trivial element of $\pi_0\mathcal{P}$.

%
%
%

\subsection{Generic paths of functions in \texorpdfstring{$\mathcal{F}$}{F}}
We recall genericity theorems of Cerf and Hatcher-Wagoner for paths $f_t\in\mathcal{F}$. Hatcher and Wagoner also consider 2-parameter families of functions in $\mathcal{F}$; this is important to show that the various invariants are well defined however we will not discuss them here and refer the reader to \cite{HatcherWagoner}. 

Following \cite{Cerf}, a generic path $f_t\in\mathcal{F}$ has the following properties. Except for at finitely many discrete values of $t$, $f_t$ is a Morse function with no $j/i$ trajectories for $j\leq i$. At the exceptional values of $t$, $f_t$ may additionally have either a single birth-death critical value (corresponding to the creation or cancellation of an $i+1$, $i$-handle pair), or a single $i/i$ trajectory (corresponding to a handle slide); otherwise $f_t$ has only Morse critical points, and no other $j/i$ trajectories for $j\leq i$ as above. 

We can display the critical value information of $f_t$ as follows.
\begin{definition}
The \emph{Cerf graphic} of a generic path $f_t$ is the subset of $I\times I$ given by
\[\bigcup_{t\in I} t\times\{\text{critical values of } f_t\}\subset I\times I.\]
We further annotate our Cerf graphics with an arrow whenever there is an $i/i$ trajectory; we draw this arrow between the two critical values. See Figure \ref{fig:cerf_arrows} for an example.
\end{definition}

\begin{figure}[htb]

	\begin{tikzpicture}
		\node[anchor=south west,inner sep=0] (image) at (0,0) {\includegraphics[width=0.35\textwidth]{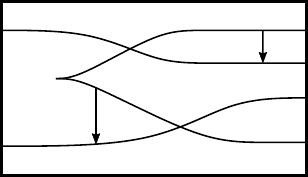}};
		\begin{scope}[x={(image.south east)},y={(image.north west)}]
		\end{scope}
	\end{tikzpicture}
	\caption{\label{fig:cerf_arrows}A Cerf graphic annotated by arrows to show trajectories between handles of the same index. This path satisfies the 1-parameter ordering condition of Proposition \ref{one-parameter ordering}. Left to right we see a birth, an $i/i$ handle slide, an $i+1$ crossing, and $i$ crossing, then an $i+1/i+1$ handle slide.}
\end{figure}

Just as we can deform Morse functions to be self-indexing, we can deform paths of functions to have desirable properties. We recall the one-parameter ordering theorem of Cerf \cite{Cerf}; see \cite{HatcherWagoner} for a detailed treatment of this.

\begin{proposition}[One-parameter ordering]\label{one-parameter ordering}
	Let $f_t\in\mathcal{F}$ be a path whose endpoints $f_0$ and $f_1$ are Morse functions with ordered, distinct critical values $($i.e.\ almost self-indexing, but perturbed so the critical values are distinct$)$. We can deform this family fixing the endpoints so that $f_t$ is a Morse function with ordered, distinct critical values for all but finitely many values of $t$. At these exceptional values of $t$, either two critical points of index $i$ have the same critical value (shown as a crossing on the Cerf Graphic), or there is a single birth-death point, or a single $i/i$ trajectory. Note that this necessarily means that each $(i+1)/i$ birth-death critical point has critical value between the critical values of index $i$ and those of index $i+1$.
	Further, we may arrange that the birth and death points are independent, meaning that there are no trajectories between any birth/death point and another critical value.
	If a path has these properties then we say it satisfies the \emph{one-parameter ordering condition}. See Figure \ref{fig:cerf_arrows} for the Cerf graphic of a 1-parameter family satisfying the one-parameter ordering condition.
\end{proposition}

If the endpoints of $f_t$, $f_0$ and $f_1$ contain only index $i$ and $i+1$ critical points for some~$i$, then in fact we can do better
\begin{theorem}{\cite[Theorem 3.1]{HatcherWagoner}}\label{twoindices}
Suppose $f_t\in\mathcal{F}$ is a path such that $f_0$ and $f_1$ are Morse functions with only index $i$ and $i+1$ critical points for some $2\leq i\leq n-2$. Then we may deform $f_t$ fixing the boundary so that for all $t\in I$, $f_t$ has only critical points of index $i$ and $i+1$. If additionally $f_0$ and $f_1$ are Morse functions with ordered, distinct critical values then we can further deform~$f_t$ fixing the endpoints so that it additionally satisfies the 1-parameter ordering condition.
\end{theorem}

In particular when $f_0, f_1\in\mathcal{E}$ we can deform $f_t$ to satisfy the conditions of Theorem \ref{twoindices} for any $i$ of our choosing, provided $2\leq i\leq n-2$. 

\begin{remark}\label{birth-death-order}
If additionally $f_0, f_1\in\mathcal{E}$, then since the birth and death points can be made independent by Theorem \ref{one-parameter ordering}, we may arrange that the births occur in a neighbourhood of 0, and that the deaths occur in a neighbourhood of 1, and that in this neighbourhood no handle slides, critical value crossings, or handle slides occur; see Figure \ref{fig:cerf_ordered}. For details on how to do this, see \cite[Chapter 1, \S 7]{HatcherWagoner}.
\end{remark}

\begin{figure}[htb]

	\begin{tikzpicture}
		\node[anchor=south west,inner sep=0] (image) at (0,0) {\includegraphics[width=0.35\textwidth]{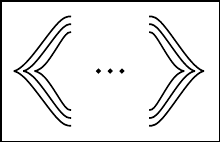}};
		\begin{scope}[x={(image.south east)},y={(image.north west)}]
		\end{scope}
	\end{tikzpicture}
	\caption{\label{fig:cerf_ordered}As in Remark \ref{birth-death-order} any path with $f_0, f_1\in\mathcal{E}$ can be deformed to satisfy the 1-parameter ordering condition, to have only critical points of index $i$ and $i+1$ for some $i$, and to have the pictured Cerf graphic in neighbourhoods of 0 and 1. Here we depict 3 births and 3 deaths, but an arbitrary number is possible (note that the number of births and deaths must be the same).}
\end{figure}

\section{Geometric picture and conventions in dimension 4}\label{geometricpicturedim4}

In this section we describe the geometric picture for paths of functions $f_t\in\mathcal{F}$ in dimension 4 and set some conventions. We will henceforth assume that $f_0, f_1\in\mathcal{E}$, i.e.\ they have no critical points.

By Theorem \ref{twoindices}, after a deformation, we may assume $f_t$ only has index 2 and 3 critical points, and that the critical values are ordered, indeed we assume that $f_t(p)<1/2$ for critical points of index 2, $f_t(p)>1/2$ for critical points of index 3, and $f_t(p)=1/2$ for births and deaths. We may also assume all births happen before time $\varepsilon$ and all deaths happen after time $1-\varepsilon$ as in Remark \ref{birth-death-order}, for some $\varepsilon\in(0,1/4)$.

Away from births, deaths, and handle slides, $f_t$ gives a handle decomposition of $X\times I$, relative to $X\times 0$, with only 2-handles and 3-handles. We wish to look at the ``middle 4-manifold'' after attaching the 2-handles, but before attaching the 3-handles. We make the following identification

\[\bigcup_{t\in[\varepsilon,1-\varepsilon]} f^{-1}_t(1/2) = V\times [\varepsilon,1-\varepsilon]\]
where $V\times t = f^{-1}_t(1/2) \cong X\#^m (S^2\times S^2)$ $\forall t$. Here $m$ is the number of births (also deaths), and we see one copy of $S^2\times S^2$ for each 2-3 handle pair. Directly after the births we see the belt sphere for a 2-handle as $S^2\times p$ and the attaching spheres of the corresponding 3-handle as $q\times S^2$ in each $S^2\times S^2$ summand. As we move forward in the $t$ direction we see an isotopy of these attaching spheres. The spheres will also change when a handle slide occurs. To keep track of these spheres we establish the following conventions.

\begin{enumerate}
	\item We denote the 2-handle belt spheres in any given $t$ slice by $A_1^t,\ldots, A_n^t\subset V\times t$. We refer to these collectively as the $A$-spheres.
	\item We denote the 3-handle attaching spheres in any given $t$ slice by $B_1^t,\ldots, B_n^t$. We refer to these collectively as the $B$-spheres.
	\item We orient the $A_i^t$s and $B_i^t$s so that the intersection between $A_i^t$ and $B_i^t$ directly after their birth is positive. There is a consistent choice of orientation for all $t$ so that the 3-manifolds $\cup_t A^t_i$ and $\cup_t B^t_i$ are oriented; note that these 3-manifolds have $S^2$ boundary components at the handle slides.
	\item We pick a basepoint $\ast\in V$; this gives us a basepoint for $V\times I$ or indeed $X\times I\times I$ by taking $\ast\times 0\in V\times I \subset X\times I\times I$. We will often abusively refer to $\ast$ as the base point in any given $t$ slice; really we mean $\ast \times t\in V\times t$. If we refer to a path to this basepoint, we implicitly take the further path to the ``true basepoint'' $\ast\times 0$ by taking a path through $\ast\times I$.
	\item We make a continuous choice of basepoint in the $A$-spheres and $B$-spheres for each value of t; that is $\ast_{A_i^t}\in A_i^t\subset V\times t$, $\ast_{B_i^t}\in B_i^t\subset V\times t$.
	\item We make a continuous choice for all $t$ of path from the basepoint of $V$ to the basepoints of the spheres; $\alpha_i^t\subset V\times t$ from $\ast$ to $\ast_{A_i^t}\in A_i^t$ and $\beta_i^t\subset V\times t$ from $\ast$ to $\ast_{B_i^t}\in B_i^t$. We do so such that directly after the creation of each pair, we have~$\alpha_i^t\cdot{\beta_i^t}^{-1}=1\in\pi_1 V$.
\end{enumerate}

\subsection{Handle slides}

At each $2/2$ trajectory we see a handle slide between the 2-handles, and similarly at each $3/3$ trajectory we see a handle slide between the 3-handles; we may assume that there are finitely many of these and that they occur at distinct values of $t$ in $(\varepsilon,1-\varepsilon)$.

We describe the effect of the handle slides on the $A$-spheres and $B$-spheres. Consider a $3/3$ handle slide at time $t$, where we slide the handle attached to $B_k^{t-\delta}$ over the handle attached to $B_j^{t-\delta}$. After the handle slide $B_j^{t+\delta} = B_j^{t-\delta}$, while $B_k^{t+\delta}$ is a connect sum of $B_k^{t-\delta}$ to a parallel copy of $B_j^{t-\delta}$.

The picture at $2/2$ handle slides is similar. In this case we slide the handle with belt sphere $A_k^{t-\delta}$ over the handle with belt sphere $A_j^{t-\delta}$ at time $t$. Then $A_k^{t+\delta} = A_k^{t-\delta}$, while $A_j^{t+\delta}$ is a connect sum of $A_j^{t-\delta}$ to a parallel copy of $A_k^{t-\delta}$.  We can see this by turning the handle decomposition upside down and considering the belt spheres of the 2-handles as attaching regions of some 3-handles; note that after turning upside down the $A_j$ 3-handle is being slid over the $A_k$ 3-handle.

\begin{remark}
In fact, the handle slides are determined by the (framed) arc in $V\times (t-\delta)$ from $A_k^{t-\delta}$ to $A_j^{t-\delta}$ or $B_k^{t-\delta}$ to $B_j^{t-\delta}$. We refer to this choice of arc as the \emph{handle slide arc}.
\end{remark}

\subsection{Intersections of spheres}

Throughout $f_t$, by our genericity assumptions the $A_i$s are disjoint from each other, as are the $B_j$s. There may be intersections between the $A_i$s and the $B_j$s however. Again by genericity of $f_t$, we can also assume that the $A_i$s and $B_j$s intersect transversely. 

Initially after the births, the handles $A_i^t$ and $B_j^t$ intersect in $\delta_{i,j}$ points. At later $t$ this may no longer be true. Note that the intersections form a 1-manifold in $V\times I$, and that the endpoints of any arcs in the 1-manifold occur at the births, the deaths or the handle slides.

Away from the handle slides we see a regular homotopy of the spheres which restricts to an ambient isotopy of each of the families $\{A_i^t\}$ and $\{B_i^t\}$. Hence new intersections are introduced and removed by finger moves and Whitney moves between some $A$ sphere and some $B$ sphere. 

\begin{remark}\label{1-param-handles}
When we wish to construct paths of functions $f_t\in\mathcal{F}$, we can do so by creating a 1-parameter family of handle structures. The rules for constructing 1-parameter families of handle structures are the same as those for creating generic 1-parameter families of functions in $\mathcal{F}$; we can create cancelling pairs of $i$, $(i+1)$-handles, perform handle slides, and perform isotopy of the attaching regions of handles, and cancel pairs of handles that intersect in a single point. Given such a 1-parameter family of handle structures there is certainly some 1-parameter family in $f_t\in\mathcal{F}$ which induces this family of handle structures. In dimension 4, when we only have 2 and 3 handles we can do all of this by considering deformation of the $A$ and $B$ spheres in the middle level; see \cite{Quinn} for Quinn's treatment of this in dimension 4.
\end{remark}

\section{Review of Hatcher and Wagoner's \texorpdfstring{$\Wh_2(\pi_1 X)$}{Wh2(pi1X)} invariant \texorpdfstring{$\Sigma$}{Sigma}}\label{review-of-sigma}

In this section we recall the definition of the map $\Sigma\colon\pi_0\mathcal{P}\rightarrow \Wh_2(\pi_1 X)$ of Hatcher and Wagoner, and the key result which we will use, namely the reduction to eyes for elements in the kernel of $\Sigma$; see \cite[Chapter VI]{HatcherWagoner}.

\subsection{Algebra of \texorpdfstring{$\Wh_2$}{Wh2}}\label{thesteinberggroup}
We begin by recalling some algebraic definitions.

Let $\Lambda=\mathbb{Z}[\pi_1 X]$. Let $\GL(\Lambda)=\lim_{n\rightarrow \infty} \GL_n(\Lambda)$.
For $\lambda\in\Lambda$ let $e_{i,j}^\lambda\in \GL(\Lambda)$ be the matrix which is the identity on the diagonal, has $\lambda$ in the $(i,j)$ position, and is zero elsewhere. We call $e_{i,j}^\lambda$ an \emph{elementary matrix}, and let~$\E(\Lambda)\subset \GL(\Lambda)$ be the subgroup generated by the elementary matrices.

One can easily verify the following relations in $\E(\Lambda)$:
\begin{enumerate}[(i)]
	\item $e_{i,j}^\lambda\cdot e_{i,j}^\mu=e_{i,j}^{\lambda+\mu}$,
	\item $[e_{i,j}^\lambda, e_{k,l}^\mu]=0$ for $i\neq l$ and $j\neq k$, and
	\item $[e_{i,j}^\lambda, e_{j,l}^\mu]=e_{i,l}^{\lambda\mu}$ for $i,j,l$ distinct.
\end{enumerate}
This motivates the following definition.
\begin{definition}
	The \emph{Steinberg group} $\St(\Lambda)$ is the group freely generated by symbols $x_{i,j}^\lambda$ for $i,j\in\mathbb{N}$ and $\lambda\in\Lambda$ subject to the relations
	\begin{enumerate}[(i)]
		\item $x_{i,j}^\lambda\cdot x_{i,j}^\mu=x_{i,j}^{\lambda+\mu}$,
		\item $[x_{i,j}^\lambda, x_{k,l}^\mu]=0$ for $i\neq l$ and $j\neq k$, and
		\item $[x_{i,j}^\lambda, x_{j,l}^\mu]=x_{i,l}^{\lambda\mu}$ for $i,j,l$ distinct.
	\end{enumerate}
	Note that we have a surjective homomorphism $\pi\colon \St(\Lambda)\rightarrow \E(\Lambda)$ sending $x_{i,j}^\lambda\mapsto e_{i,j}^\lambda$.
\end{definition}
We define $K_2(\Lambda)$ to be the kernel of $\pi\colon \St(\Lambda)\rightarrow \E(\Lambda)$. Hence we have the short exact sequence:
\[0\rightarrow K_2(\Lambda)\rightarrow\St(\Lambda)\xrightarrow{\pi}\E(\Lambda)\rightarrow 0.\]

For $g\in\pi_1X$ let $w_{i,j}^{\pm g}=x_{i,j}^{\pm g}x_{j,i}^{\mp g^{-1}} x_{i,j}^{\pm g}$, and let $W(\pm \pi_1 X)\subset \E(\Lambda)$ be the subgroup generated by the words $w_{i,j}^{\pm g}$. Then we define the \emph{second Whitehead group} to be,
\[Wh_2(\pi_1 X)=K_2(\Lambda)\; \text{ mod }\; W(\pm \pi_1 X)\cap K_2(\Lambda)\]

In order to define $\Sigma$ we need the following lemma.
\begin{lemma}[{\cite[Chapter III, Lemma 1.6]{HatcherWagoner}}]\label{lemmaW}
	Let $P\in GL(\Lambda)$ be a permutation matrix, and let $D\in GL(\Lambda)$ be diagonal with entries in $\pm\pi_1 X$. Then there exists some~$w\in W(\pm\pi_1 X)$ such that $\pi(w)=P\cdot D$.
\end{lemma}
This follows from the fact that $P\cdot D$ can be written as a product \[P\cdot D=\prod_k e_{i_k,j_k}^{\pm g_k}e_{j_k,i_k}^{\mp g_k^{-1}} e_{i_k,j_k}^{\pm g_k}.\]

\subsection{Definition of \texorpdfstring{$\Sigma$}{Sigma}}

We now recall the definition of $\Sigma\colon\mathcal{P}\rightarrow \Wh_2(\pi_1 X)$. For full details see \cite[Chapter IV]{HatcherWagoner}, and see \cite[Chapter V, \S 6]{HatcherWagoner} for the version of the definition we give here. The approach is to construct a map (which we also abusively call $\Sigma$) 
\[\Sigma\colon \pi_1(\mathcal{F},\mathcal{E})\rightarrow \Wh_2(\pi_1 X),\]
and use $\Pi$ to identify $\pi_0\mathcal{P}$ with $\pi_1(\mathcal{F},\mathcal{E})$.

As in Remark \ref{birth-death-order} we deform $f_t$ so that it satisfies 1-parameter ordering, has only critical points of index $i$ and $i+1$ for some $2\leq i\leq n-2$, that the birth points appear in time interval $(0,\varepsilon)$, that all death points appear in the time interval $(1-\varepsilon,1)$ and that no critical lines cross and no handle slides occur in either of these intervals, so that in these intervals the Cerf graphic is as in Figure \ref{fig:cerf_ordered}.

We fix a base point $\ast\in X\times I$ and a local orientation at $\ast$. For each birth point $c\in X\times I$ we choose a label $j\in\mathbb{N}$ for this critical point, and a path in $\gamma_{c}\subset X\times I$ from $\ast$ to $c$. For $t\in [\varepsilon,1-\varepsilon]$ denote the index $i+1$ and $i$ critical points created by this birth by $z^{j}_t$ and $b^{j}_t$ respectively. Using the path $\gamma_c$ we can obtain a continuous choice of path from $\ast$ to $z^{j}_t$ and $b^{j}_t$ for each $t$, denoted $\gamma_{b^{j}_t}$ and $\gamma_{z^{j}_t}$ respectively. We also choose an orientation of the stable sets. This gives a basis for $C_i(f_t,\eta_t)$ as in Section \ref{morsetheorysection} and a matrix $\partial_{i+1}^{f_t}=\partial_t\GL(\mathbb{Z}[\pi_1 X])$.

After the birth points $\partial_{\varepsilon}$ is the identity matrix, as there is precisely one trajectory from $z^j$ to $b^j$ for each $j$.

The matrix $\partial_t$ remains constant until passing an $i/i$ or $i+1/i+1$ trajectory. On passing an $i+1/i+1$ trajectory $\varphi$ from $z^j_t$ to $z_k^t$, with associated $\pm\pi_1 X$ element given by $\lambda=\gamma_{z^j_t}\cdot\varphi\cdot\gamma_{z^k_t}^{-1}\in \pm \pi_1 X$ (with sign depending on the orientations), we have that
\[\partial_{t+\delta}=\partial_{t-\delta}\circ  e_{j,k}^\lambda.\]

On passing an $i/i$ trajectory $\varphi$ from $b_j^t$ to $b_k^t$ with associated $\pm\pi_1 X$ element given by~$\lambda=\pm\gamma_{b^j_t}\circ\varphi\circ\gamma_{b^k_t}^{-1}\in \pm \pi_1 X$, we have
\[\partial_{t+\delta}=\left(e_{j,k}^{\lambda}\right)^{-1}\circ\partial_{t-\delta}=e_{j,k}^{-\lambda}\circ\partial_{t-\delta}\]

Hence just before the death points at time $1-\varepsilon$ we have,
\begin{align*}
\partial_{1-\varepsilon} &= e_{j'_m,k'_m}^{-\lambda'_m}\ldots e_{j'_{2},k'_{2}}^{-\lambda'_{2}} e_{j'_{1},k'_{1}}^{-\lambda'_{1}}\partial_{\varepsilon}e_{j_{1},k_{1}}^{\lambda_{1}}e_{j_{2},k_{2}}^{\lambda_{2}}\ldots e_{j_{l},k_{l}}^{\lambda_{l}}\\
&= e_{j'_m,k'_m}^{-\lambda'_m}\ldots e_{j'_{2},k'_{2}}^{-\lambda'_{2}} e_{j'_{1},k'_{1}}^{-\lambda'_{1}}e_{j_{1},k_{1}}^{\lambda_{1}}e_{j_{2},k_{2}}^{\lambda_{2}}\ldots e_{j_{l},k_{l}}^{\lambda_{l}}.
\end{align*}

On the other hand, at time $1-\varepsilon$ each critical point $b_j$ will cancel with some critical point $z_j$, so we know there is precisely one trajectory between them along some arc~$\lambda\in\pm\pi_1 X$, and no other trajectories. Hence,
\[\partial_{1-\varepsilon}(z_j)=\lambda\cdot b_k\]
for some $k$. Hence we can write $\partial_{1-\varepsilon}$ as a product,
\[\partial_{1-\varepsilon}=P\cdot D\]
where $P$ is a permutation matrix, and $D$ is diagonal with entries in $\pm \pi_1 X$. By Lemma~\ref{lemmaW}, there exists~$w\in W(\pm\pi_1 X)$ such that $\pi(w)=P\cdot D$.

We can now define $\Sigma\colon\pi_1(\mathcal{F},\mathcal{E})\rightarrow \Wh_2(\pi_1 X)$ by sending the path $f_t$ to 
\[x_{j'_m,k'_m}^{-\lambda'_m}\ldots x_{j'_{2},k'_{2}}^{-\lambda'_{2}} x_{j'_{1},k'_{1}}^{-\lambda'_{1}}x_{j_{1},k_{1}}^{\lambda_{1}}x_{j_{2},k_{2}}^{\lambda_{2}}\ldots x_{j_{l},k_{l}}^{\lambda_{l}}\cdot w^{-1}\; \text{ mod }\; W(\pm\pi_1 X)\cap K_2(\Lambda) \in \Wh_2(\pi_1 X).\]
Note that $\pi\left(x_{j'_m,k'_m}^{-\lambda'_m}\ldots x_{j'_{1},k'_{1}}^{-\lambda'_{1}}x_{j_{1},k_{1}}^{\lambda_{1}}\ldots x_{j_{l},k_{l}}^{\lambda_{l}}\cdot w^{-1}\right)=\partial_{1-\varepsilon}\partial_{1-\varepsilon}^{-1}=1$ so this is indeed an element of $K_2(\Lambda)$.

Hatcher and Wagoner that in dimension $n\geq 4$ that $\Sigma(f_t)$ does not depend on the various choices made above, and indeed is independent of the homotopy class of the path in $(\mathcal{F},\mathcal{E})$ as required. See \cite[Chapter IV, \S 3]{HatcherWagoner} for the proof that $\Sigma$ is well defined.

\begin{remark} If $P\cdot D$ were the identity matrix, then the product 
\[x_{j'_m,k'_m}^{-\lambda'_m}\ldots x_{j'_{1},k'_{1}}^{-\lambda'_{1}}x_{j_{1},k_{1}}^{\lambda_{1}}\ldots x_{j_{l},k_{l}}^{\lambda_{l}}\]
 would give an element of $K_2(\pi_1 X)$. In the case that it is not, we could deform the 1-parameter family so that $P\cdot D$ is the identity by adding a trivial pseudo-isotopy corresponding to word the $w$ defined above. Indeed by \cite[Chapter IV, Lemma 2.7]{HatcherWagoner} for $n\geq 4$, for every such $w\in W(\pm \pi_1 X)$ there exists a path $f_t$ with handle slides corresponding to $w$, such that $f_t$ is isotopic to the trivial path fixing the end points.
\end{remark}

\subsection{Reduction to eyes and \texorpdfstring{$\ker\Sigma$}{ker(Sigma)}}
Hatcher and Wagoner prove the following about paths in the kernel of $\Sigma$.
\begin{theorem}[{\cite[Chapter~VI,~Theorem~2]{HatcherWagoner}}]\label{reductiontoeyethm}
	Let $n\geq 4$. Given a 1-parameter family $f_t$ representing a class in $\pi_1(\mathcal{F},\mathcal{E})$ for which $\Sigma(f_t)=0$, there is a deformation of~$f_t$ fixing the end points to a 1-parameter family $f_t^\prime$ whose Cerf graphic consists of a nested collection of \emph{eyes}, with only index $i$ and $i+1$ critical points for some $2\leq i\leq n-2$, and no $i/i$ or $i+1/i+1$ trajectories; this necessarily means that each $(i+1)$-handle cancels at a death point with the $i$-handle it was created with. See Figure \ref{fig:nested_eyes}. We can also assume that the births and deaths remain independent.
\end{theorem}

Hatcher and Wagoner in fact prove that when $n\geq 5$ any one-parameter family $f_t$ with $\Sigma(f_t)=0$ can be deformed to a 1-parameter family consisting of only a \emph{single} eye. As noted by Quinn in \cite{Quinn} the reduction to multiple eyes holds in 4-dimensions, and only the step reducing from many eyes to a single eye requires dimension $n\geq 5$.

\begin{figure}[h]
	\begin{tikzpicture}
		\node[anchor=south west,inner sep=0] (image) at (0,0) {\includegraphics[width=0.35\textwidth]{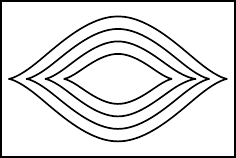}};
		\begin{scope}[x={(image.south east)},y={(image.north west)}]
		\end{scope}
	\end{tikzpicture}
	\caption{\label{fig:nested_eyes}A Cerf graphic that is a collection of nested eyes.}
\end{figure}

\section{Stable surjectivity of \texorpdfstring{$\Sigma$}{Sigma} in dimension 4}
In \cite{HatcherWagoner} they show the following result.
\begin{theorem}[{\cite[Theorem 2, Chapter VI ]{HatcherWagoner}}]\label{high-dim-surjectivity}
	$\Sigma\colon\pi_0\mathcal{P}\rightarrow\Wh_2(\pi_1 X)$ is surjective in dimension $n\geq 5$.
\end{theorem}
The proof is simple, we recall it here to motivate our upcoming proof.
\begin{proof}
	Given $x\in \Wh_2(\pi_1 X)$, let $x_{j_1,k_1}^{s_1\lambda_1}\ldots x_{j_m,k_m}^{s_m\lambda_m} \in K_2(\mathbb{Z}[\pi_1 X])$ with $\lambda_l\in\pi_1 X$ and $s_l\in\{+1,-1\}$ be a word representing $x$; note that any word in $K_2(\mathbb{Z}[\pi_1 X])$ can be written in this way. 
	
We build a path $f_t$ in $\mathcal{F}$ as in Remark \ref{1-param-handles}. Let~$N=\max(\max_l(j_l),\max_l(k_l))$. Starting with the trivial element $p\in\mathcal{F}$, we first create $N$ pairs of cancelling $i$ and $(i+1)$-handles for some $2\leq i\leq n-2$, and label them $1,\ldots N$. 

We then perform $m$ handle slides of the $(i+1)$-handles over $(i+1)$-handles. For $l=1,\ldots, m$ we slide the $j_l$th $(i+1)$-handle over the $k_l$th $(i+1)$-handle along an arc so that the associated element of $\pi_1 X$ is $\lambda_l$. For each handle slide we can also choose to perform an oriented or unoriented handle slide when $s_l=+1$ or $s_l=-1$ respectively. After these handle slides, the resulting differential is
	\[\partial_{i+1}= e_{j_1,k_1}^{s_1\lambda_1}\ldots e_{j_m,k_m}^{s_m\lambda_m}=\pi(x_{j_1,k_1}^{s_1\lambda_1}\ldots x_{j_m,k_m}^{s_m\lambda_m})=1.\]
	
Now as in the proof of the $s$-cobordism theorem we may realise this differential geometrically; that is in the middle level, we know that algebraically the attaching region for the~$j$th $(i+1)$-handle and belt sphere for the $k$th $i$-handle intersect in $\delta_{j,k}$ points. We make use of the Whitney trick to deform these spheres (and correspondingly the handle decomposition) so that they intersect in $\delta_{j,k}$ points geometrically. Note that the Whitney trick requires dimension $n\geq 5$. 

Now the handles intersect geometrically in $\delta_{j,k}$ points. We may cancel the pairs of handles, creating $N$ deaths. After the deaths, the resulting Morse function is an element of $\mathcal{E}$ as it has no critical points. Hence we have described a path in $\pi_1(\mathcal{F},\mathcal{E})$ as required. It is clear that this path has $\Sigma(f_t)= x_{j_1,k_1}^{s_1\lambda_1}\ldots x_{j_m,k_m}^{s_m\lambda_m}$.
\end{proof}

This proof does not work in dimension $4$ since we cannot apply the Whitney trick. In this section we present an proof of a stable version of this theorem in dimension 4.

We first prove a further useful lemma for constructing transverse spheres.
\begin{lemma}\label{handle-slide-transverse}
Let $X$ be a 4-manifold. Consider a handle decomposition of $X\times I$ relative to $X\times 0$ with $N$ 3-handles. Denote the attaching regions for the 3-handles by $B_1,\ldots, B_N\subset V$ where $V$ is the middle level between the 2 and 3 handles. Suppose there are disjointly embedded spheres $T_1,\ldots, T_N\subset V$ such that $B_i$ intersects $T_j$ in $\delta_{i,j}$ points. Assume that~$T_1,\ldots, T_N$ have trivial normal bundle. Suppose we now perform some number of handle slides between the 3-handles to obtain a new set of 3-handle spheres $B_1^\prime,\ldots , B_N^\prime$. We claim there exists disjointly embedded spheres $T_1^\prime,\ldots ,T_N^\prime\subset V$ such that $B_i^\prime$ intersects $T_j^\prime$ in $\delta_{i,j}$ points, and that $T_i^\prime$ and $T_j$ are disjoint for all $i$ and $j$. We also claim that we may choose $T_1^\prime,\ldots, T_N^\prime$ to have trivial normal bundle.
\end{lemma}
\begin{proof}
	We prove the claim by induction on the number of handle slides. Denote the 3-handle attaching spheres after $k$ handle slides by $B_1^k,\ldots, B_N^k$, we construct $T_1^k,\ldots, T_N^k$ which are dual to $B_1^k,\ldots, B_N^k$ and have trivial normal bundle. Since $T_i$ has trivial normal bundle, letting $T_i^0$ be a parallel copy of $T_i$ is sufficient for $k=0$. 
	
	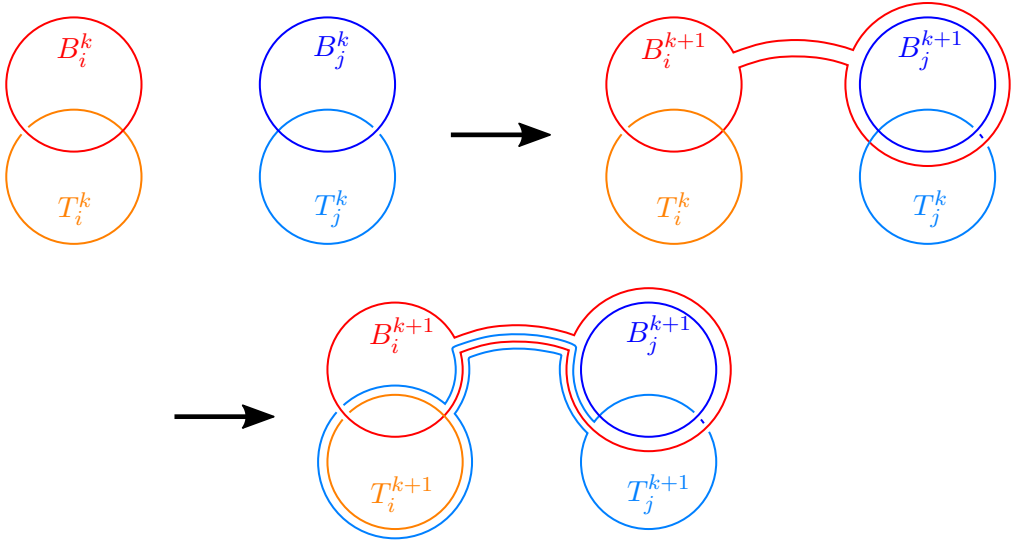
\begin{figure}[h!]
		\begin{tikzpicture}
			\node[anchor=south west,inner sep=0] (image) at (0,0) {\includegraphics[width=0.9\textwidth]{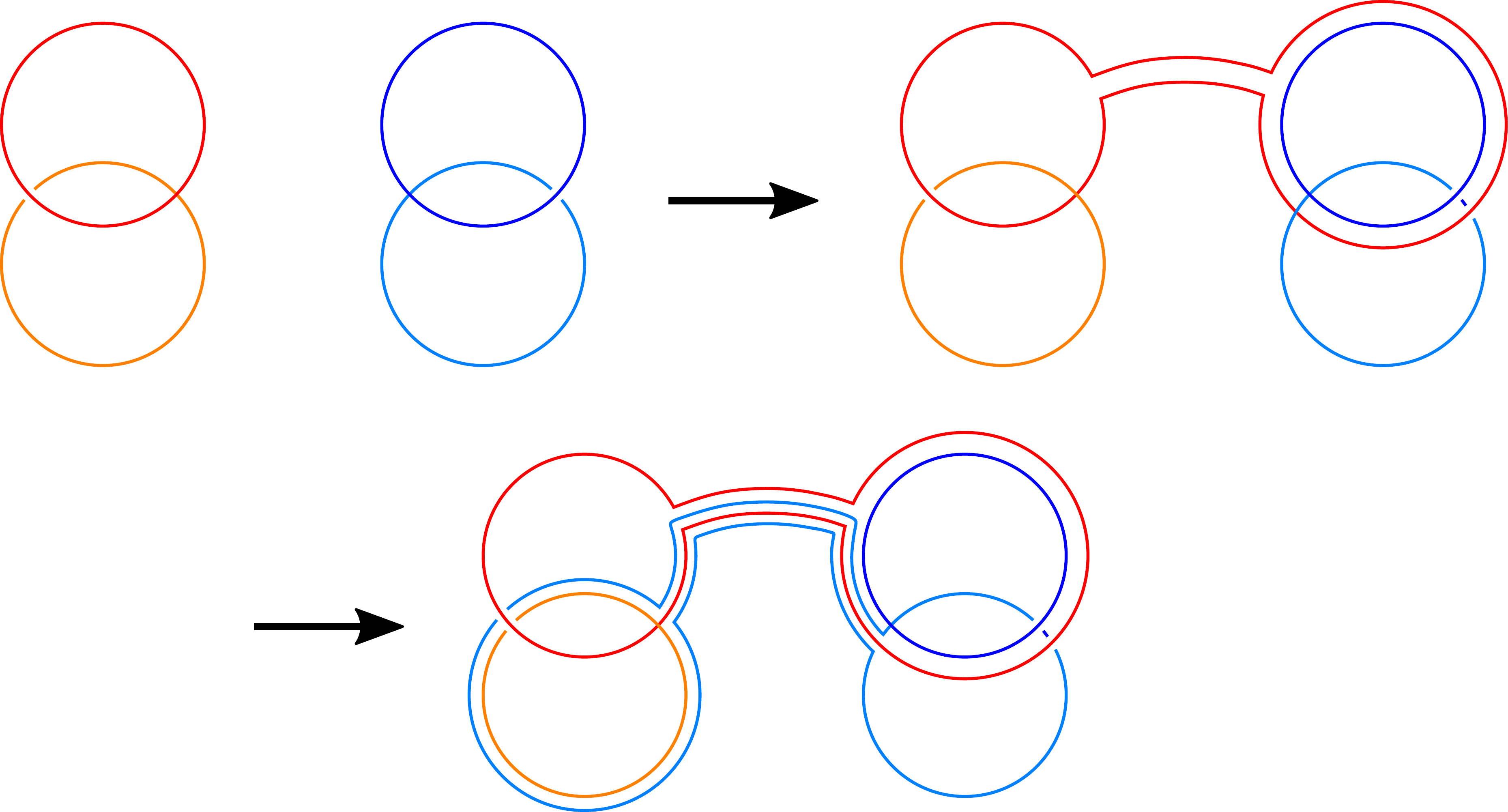}};
			\begin{scope}[x={(image.south east)},y={(image.north west)}]
				\node[red] at (0.07,0.91) {$B_i^k$};
				\node[orange] at (0.07,0.615) {$T_i^k$};
				\node[blue] at (0.325,0.91) {$B_j^k$};
				\node[light-blue] at (0.325,0.615) {$T_j^k$};
				
				\node[red] at (0.07+0.595,0.91) {$B_i^{k+1}$};
				\node[orange] at (0.07+0.595,0.615) {$T_i^k$};
				\node[blue] at (0.325+0.595,0.91) {$B_j^{k+1}$};
				\node[light-blue] at (0.325+0.595,0.615) {$T_j^k$};
				
				\node[red] at (0.07+0.325,0.91-0.53) {$B_i^{k+1}$};
				\node[orange] at (0.07+0.325,0.615-0.53) {$T_i^{k+1}$};
				\node[blue] at (0.325+0.325,0.91-0.53) {$B_j^{k+1}$};
				\node[light-blue] at (0.325+0.325,0.615-0.53) {$T_j^{k+1}$};

			\end{scope}
		\end{tikzpicture}
		\caption{\label{fig:transversespheres}Performing the $(k+1)$th handle slide, then performing the Norman trick to remove the extra intersection.}
	\end{figure}

Suppose we have a family of spheres $T_1^k,\ldots T_N^k$ with trivial normal bundle, which are dual to $B_1^k,\ldots B_N^k$ and disjoint from $T_1,\ldots T_N$. We now perform the $(k+1)$th handle slide, which we may assume slides $B_i^k$ over $B_j^k$ using some arc $\gamma\subset V$ from  $B_i^k$ to $B_j^k$. By general position we may perform an isotopy of $\gamma$ so that it is disjoint from $T_1^k,\ldots T_N^k$ and $T_1,\ldots T_N$ (it is necessarily disjoint from $B_1^k,\cdots B_N^k$ except at the endpoints). Now $B_i^{k+1}=B_i^k\#_\gamma B^k_j$. Hence $B_i^{k+1}$ intersects $T_i^k$ in a single point $p$, and $T_j^k$ in a single point $q$. We remove the intersection between $B_i^{k+1}$ and $T_j^k$ by choosing an arc in $B_i^{k+1}$ from $q$ to $p$, and tubing $T_j^k$ to a parallel copy of $T_i^k$ using the Norman trick. We call the resulting sphere $T_j^{k+1}$; see Figure \ref{fig:transversespheres}. Note that $B_i^{k+1}$ may intersect $T_1,\ldots T_N$ in finitely many points so we must ensure we choose the arc from $q$ to $p$ so that it misses these intersections, and so we do not introduce any intersections with $T_1,\ldots T_N$. Since $T_i^k$ is disjoint from $T_1,\ldots T_N$ provided we choose a small enough neighbourhood in which to take the parallel copy we can ensure that the $T^{k+1}_j$  sphere does not intersect $T_1,\ldots T_N$. Taking $T_r^{k+1}=T_r^k$ for $r\neq j$ we see that $T_1^{k+1},\ldots,T_N^{k+1}$ are a dual family for $B_1^{k+1},\ldots B_N^{k+1}$ as required, and do not intersect $T_1,\ldots T_N$.

We also note that $T_1^k,\ldots, T_N^k$ have trivial normal bundle, since tubing spheres with trivial normal bundle together results in a sphere with trivial normal bundle.
\end{proof}

\begin{remark}
By turning the handle decompositions in the above proof upside down, we can prove the same fact about the belt spheres of 2-handles in $V$, performing 2-handle slides instead of 3-handle slides.
\end{remark}

\subsection{Proof of stable surjectivity in dimension 4}
In this section we prove that $\Sigma$ is stably surjective in dimension $4$.
\stablesurjectivity*
\begin{proof}
	As in the proof of Theorem \ref{high-dim-surjectivity} we consider the word $x_{i_1,j_1}^{s_1\lambda_1}\cdots x_{i_m,j_m}^{s_m\lambda_m}\in K_2(\mathbb{Z}[\pi_1 X])$ representing $x$, where $s_k\in \{+1,-1\}$, $\lambda_k\in\pi_1(X)$, and $i_k, j_k\in {1,\ldots, N}$ for some $N$. 
	
Let $Y_N=\#^N S^2\times S^2$ and let $X' = X\# Y_N$. We claim that $N$ is sufficient to find a pseudo-isotopy of $X'$ such that 
\[\Sigma(F)=\Big[x_{i_1,j_1}^{s_1\lambda_1}\cdots x_{i_m,j_m}^{s_m\lambda_m}\Big]=x\in\Wh_2(\pi_1 (X\#^N S^2\times S^2)).\]
	
	To show the existence of such a pseudo-isotopy we construct a path in $\mathcal{F}(X')$ by considering a deformation of handle structures of $X'\times I$ as in Remark \ref{1-param-handles}. We start with the trivial element of $\mathcal{F}(X')$ and correspondingly the trivial handle structure for~$X'\times I$. 
	
	We denote the $X$ summand of $X'=X\# Y_N$ by $\widehat{X}=X\setminus B^3\subset X'$. We initially describe a handle deformation just in $\widehat{X}\times I$ and leave the handle structure of 
	\[(Y_N\setminus B^3)\times I\subset X'\times I\] fixed. Note that $\pi_1(\widehat{X})=\pi_1(X)$.
	
	We first create $N$ cancelling 2 and 3-handle pairs in $\widehat{X}$ and then perform $m$ handle slides, again within $\widehat{X}$, of 3-handles over 3-handles in accordance with the word 
\[x_{i_1,j_1}^{s_1\lambda_1}\cdots x_{i_m,j_m}^{s_m\lambda_m}\in K_2(\mathbb{Z}[\pi_1 X])=K_2(\mathbb{Z}[\pi_1 X']).\]
	
	As usual we consider the middle level $V=X'\#^N S^2\times S^2$ between the 2-handles and 3-handle. In order to differentiate this copy of $\#^N S^2\times S^2$ from $Y_N$, we let $Z_N = \#^N S^2\times S^2$ and let $V=X'\# Z_N$. Since the cancelling pairs were created in $\widehat{X}\times I$, we consider the middle level of $\widehat{X}\times I$, namely $\widehat{V}=\widehat{X}\# Z_N\subset V$. Note that 
	\[V=X' \# Z_N = X\# Y_N \# Z_N = (\widehat{X} \# Z_N)\cup (Y_N\setminus B^4)  = \widehat{V}\cup (Y_N\setminus B^4)\]
Note also that that the 2-handle belt spheres and 3-handle attaching spheres lie within $\widehat{V}$ throughout the handle slides.

We fix a time $t_0$ \emph{after} the handle slides have taken place and work in this fixed time $t_0$. In the middle level at time $t_0$ let $A_1,\ldots,A_N\subset \widehat{V}$ be the belt spheres of the 2-handles and $B_1,\ldots,B_N\subset \widehat{V}$ be the attaching spheres for the 3-handles. Note that since the 2-handle spheres did not move during the handle slides, they are still $S^2\times p$ slices of $Z_N=\#^N S^2\times S^2$.
	
Since prior to the handle slides the $i$th 3-handle attaching sphere and the $j$th 2-handle belt sphere (which since it remained fixed is $A_j\subset V$) intersect in $\delta_{i,j}$ points, by Lemma \ref{handle-slide-transverse} there exist spheres $C_1,\ldots, C_N$ with trivial normal bundle such that $B_i$ and $C_j$ intersect in $\delta_{i,j}$ points, and that $A_i$ and $C_j$ are disjoint for all $i$, $j$.

	As $x_{i_1,j_1}^{s_1\lambda_1}\cdots x_{i_m,j_m}^{s_m\lambda_m}\in K_2(\mathbb{Z}[\pi_1 X])$, the differential $\partial_3$ is the identity, so $A_i$ and $B_j$ have algebraic intersection $\delta_{i,j}\in \mathbb{Z}[\pi_1 X]$. This means that we can pair all of the intersections up (except for a single intersection for each $A_i$, $B_i$ pair) so that for each of the paired intersection points $p,q\in A_i\cap B_j$ the intersection as measured in $\pm\pi_1 X$ is $+\gamma$ for~$p$ and~$-\gamma$ for~$q$ for some $\gamma\in\pi_1 X$. We can also pick Whitney arcs for each pair; note that removing disjointly embedded arcs from a sphere does not make the sphere disconnected, so we can pick disjoint Whitney arcs for all of the pairings. The resulting Whitney circles for each pair of intersections vanish in~$\pi_1 X=\pi_1 \widehat{V}$, and so there exists a Whitney disc~$W\subset\widehat{V}$ for each pair of points. We can perform boundary twists between~$W$ and $B_j$ so that~$W$ is a correctly framed Whitney disc. We do this for all of the pairs, and refer to the collection of these Whitney discs as the $W$-discs.
	
	The $W$-discs may self intersect, and also intersect other $W$-discs. They may also intersect the $A$-spheres, the $B$-spheres and the $C$-spheres. We will manipulate these discs to construct a new family of discs which are embedded and disjoint from the $A$-spheres and $B$-spheres.
	
	First, noting that $A_i$ is a copy of $S^2\times p$ in $\widehat{V}= \widehat{X}\# Z_N = \widehat{X}\#^N S^2\times S^2$, there exists some transverse sphere for $A_i$, $A_i^\ast=q\times S^2$. For each intersection $r$ between a $W$-disc and an $A_i$ we use the Norman trick to tube the $W$-disc into a parallel copy of $A_i^\ast$. Doing this requires a choice of arc in $A_i$ from $r$ to $A_i\cap A_i^*$, when choosing this arc we ensure it misses any other intersection points in $A_i$ and any Whitney arcs; this is possible as these do not disconnect~$A_i$. We do this for each intersection between an $A$-sphere and $W$-disc successively. We obtain new $W$-discs which are disjoint from the $A$-spheres; note that this may introduce intersections between $W$-discs and $W$-discs, as well as new intersections between the $W$-discs and the $B$-spheres and $C$-spheres. The $W$-discs are also still correctly framed Whitney discs as the dual spheres $A_i^\ast$ have trivial normal bundle. 
	
	Next, we push down the self intersections of the $W$-discs into the $B$-spheres, as in Section \ref{section:push-down}. After doing this for all intersections between the $W$-discs, the $W$-discs are disjoint and embedded and disjoint from the $A$-spheres, but still possibly intersect the $C$ and $B$ spheres. The $W$-discs are also still correctly framed Whitney discs.
	
	Since everything we have done so far has been within $\widehat{V}=\widehat{X}\# Z_N$, all of the $W$-discs, $A$-spheres, $B$-spheres and $C$-spheres are disjoint from the $Y_N\setminus B^4$ subset of \[V= (\hat{X}\# Z_N)\cup (Y_N \setminus B^4).\] We now make use of the $S^2\times S^2$s in $Y_N =\#^N S^2\times S^2$. For each $C$-sphere $C_j$, we create a new sphere $C_j'$ by tubing $C_j$ into $S^2\times p\subset S^2\times S^2$ in $Y_N$. Ensuring to tube each $C$-sphere into a different $S^2\times S^2$, we obtain spheres $C_1',\ldots, C_N'$ which are disjoint and embedded. Taking $D_j= p\times S^2$ in each summand of $Y_N$, we also have spheres $D_1,\ldots D_n$ such that $D_i$ and $C_j'$ intersect in $\delta_{i,j}$ points. Clearly the $D$-spheres are disjoint from the $W$-discs, the $A$-spheres and the $B$-spheres. Note also that the $C'$-spheres have trivial normal bundle, as do the $D$-spheres.
	
	We can now remove intersections between the $W$ discs and the $C'$-spheres by using the Norman trick; for each intersection between a $W$-disc and $C_j'$ we create a new Whitney disc by tubing into a parallel copy of $D_j$, noting that after each tubing $D_j$ is still disjoint from the $W$-discs so we may repeat the process without introducing new intersections between $W$-discs. After doing this for each intersection we obtain embedded $W$-discs which are disjoint from the $A$-spheres and $C'$-spheres but still intersect the $B$-spheres. The $W$-discs are also still correctly framed Whitney discs as each $D_j$ has trivial normal bundle.
	
We now remove the intersections between the $B$-spheres and the $W$-discs; for each intersection $p$ between a $W$-disc and $B_j$, pick an arc from $p$ to $B_j\cap C_j'$ which is disjoint from the Whitney arcs and other intersection points (note that the Whitney arcs do not disconnect $B_j$ so this is always possible), then use this arc to perform the Norman trick, tubing the $W$-disc into a parallel copy of $C_j'$. Note that $C_j'$ is disjoint from the $W$-discs so this does not introduce new intersections between the $W$-discs. Since after each tubing $C_j'$ is still disjoint from the $W$-discs, we may repeat the process for every intersection between $B$-spheres and $W$-discs. Note also that $C_j'$ does not intersect the $A$-spheres so this does not introduce intersections between $W$-discs and $A$-spheres. At the end of this process we obtain embedded $W$-discs which are disjoint from all the $A$-spheres and $B$-spheres. Additionally the $W$-discs are still framed as $C_j'$ has trivial normal bundle for all $j$.
	
	Finally, having built embedded $W$-discs in $V$ at time $t_0$, we now use them to perform Whitney moves for each pair of points to remove the pairs of intersections; note that to do this we stop working in time $t_0$. After performing these Whitney moves, $A_i$ and $B_j$ intersect in $\delta_{i,j}$ points, so we can cancel all of the handles handles. After we do this the resulting handle structure has no handles, so the corresponding Morse function lies in $\mathcal{E}(X\#^N S^2\times S^2)$. Hence we have built a path in $\mathcal{F}(X\#^N S^2\times S^2)$ with endpoints in $\mathcal{E}(X\#^N S^2\times S^2)$. It is clear that $\Sigma(f_t)=x_{i_1,j_1}^{s_1\lambda_1}\cdots x_{i_m,j_m}^{s_m\lambda_m}$ as required.
\end{proof}

\section{The \texorpdfstring{$\Wh_1(\pi_1 X;\mathbb{Z}_2\times \pi_2 X )$}{Wh1(pi1(X);Z2xpi2(X)} invariant \texorpdfstring{$\Theta$}{Theta}}

In Section \ref{review-of-sigma} we described the map $\Sigma$, which in dimension $n\geq 4$ is a complete obstruction to reducing a 1-parameter family $f_t$ with endpoints in $\mathcal{E}$ to one whose Cerf graphic is a collection of nested eyes as in Figure \ref{fig:nested_eyes}. That is, any pseudo-isotopy in $\ker\Sigma$ is represented by such a 1-parameter family.

In this section we recall the definition of the map described by Hatcher and Wagoner in \cite[Chapter VII]{HatcherWagoner} \[\Theta\colon\ker\Sigma\rightarrow\Wh_1(\pi_1(X);\mathbb{Z}_2\times \pi_2(X))/ \chi( K_3 \pi_1 X).\] 
In dimension $\geq 6$, Hatcher and Wagoner (with later clarification by Igusa in \cite{Igusa}) show that $\Theta$ is a complete obstruction to removing all the eyes from such a nested collection of eyes (at which point we are left with an empty Cerf graphic, and so a path in $\mathcal{E}$). Hence, in dimension $\geq 6$, $\Sigma$ together with $\Theta$ provide a complete obstruction to a pseudo-isotopy being isotopic to an isotopy.

When $k_1 X=0$, Igusa in \cite{Igusa} describes a map 
\[\Theta_\sigma\colon\pi_0\mathcal{P}\rightarrow \Wh_1(\pi_1(X);\mathbb{Z}_2\times \pi_2(X))\] 
 dependent on a choice of section $\sigma\colon X_{(1)}\rightarrow X_{(2)}$. We will not address this extension here. When restricted to $\ker\Sigma$, Igusa's map $\Theta_\sigma|_{\ker\Sigma}$ agrees with the $\Theta$ described in \cite[Chapter VII]{HatcherWagoner} and is independent of the section $\sigma$.

\subsection{Definition of \texorpdfstring{$\Wh_1(\pi_1 X;\mathbb{Z}_2\times\pi_2 X)$}{Wh1(pi1(X);Z2xpi2(X))}}\label{wh1def}

Let $\Gamma$ be an abelian group acted on by a group~$\pi$. For our purposes $\Gamma$ will be the group $\mathbb{Z}_2\times \pi_2(X)$ and $\pi$ will be $\pi_1(X)$, with the usual action on $\pi_2(X)$ and the trivial action on $\mathbb{Z}_2$. We write elements of $R=\Gamma[\pi]\times\mathbb{Z}[\pi]$ as finite formal sums $\sum_i (\alpha_i +n_i) \sigma_i$ for $\alpha_i\in\Gamma$, $\sigma_i\in\pi$, $n_i\in\mathbb{Z}$. We give $R$ a ring structure via
\[\bigg(\sum_i (\alpha_i+n_i)\sigma_i\bigg)\cdot \bigg(\sum_j(\beta_j+m_j)\tau_j\bigg)=\sum_{i,j}(n_i \beta_j^{\sigma_i}+m_j\alpha_i+n_i m_j)\sigma_i\tau_j,\]
where $\alpha_i,\beta_j\in\Gamma$, $n_i,m_j\in\mathbb{Z}$, $\sigma_i,\tau_j\in\pi$ and $\beta^\sigma$ denotes the action of $\sigma$ on $\beta$. Note that~$\Gamma[\pi]$ is an ideal of $R$. Note also that the multiplication is trivial on $\Gamma[\pi]$; that is~$x\cdot y=0$, for all~$x,y\in\Gamma[\pi]$. We define
\[\GL(\Gamma[\pi])=\ker\big(\GL(R)\rightarrow\GL(R/\Gamma[\pi])\big)=\{ I+A\;| \; A\text{ has entries in }\Gamma[\pi]\}\]
Note that $R$ and $R/\Gamma[\pi]$ are rings, and here $\GL(R)$ and $\GL(R/\Gamma[\pi])$ denote the usual general linear group of matrices in these rings (that is, the union over $n$ of $GL_n(R)$ and $\GL_n(R/\Gamma[\pi])$ respectively). Note also that $I+ A$ is always in $\GL(R)$ if $A$ has entries in $\Gamma[\pi]$, since
\[(I+A)\cdot(I-A)= I - A\cdot A= I,\]
noting that $A\cdot A=0$ as multiplication in the ideal $\Gamma[\pi]$ is trivial.

\begin{definition}
	We define $K_1\Gamma[\pi]$ by
	$K_1\Gamma[\pi]=\GL(\Gamma[\pi])/[\GL(R),\GL(\Gamma[\pi])]$
\end{definition}
Hatcher proves the following.
\begin{proposition}[{\cite[Proposition 1.1]{Hatcher}}]
	The trace map 
	\[\tr:K_1\Gamma[\pi]\rightarrow\Gamma[\pi]/\langle ar-ra\;|\;a\in\Gamma[\pi],\ r\in R\rangle\] 
	is an isomorphism. The subgroup $\langle ar-ra\;|\;a\in\Gamma[\pi],\ r\in R\rangle$ can also be expressed as~$\langle \alpha\sigma -\alpha^\tau \tau\sigma\tau^{-1}|\alpha\in\Gamma ,\ \tau,\sigma\in \pi\rangle $, where $\alpha^\tau$ denotes the result of acting on $\alpha$ by $\tau$. Hence we obtain
	\[ K_1\Gamma[\pi] \cong \Gamma[\pi]/ \langle \alpha\sigma -\alpha^\tau \tau\sigma\tau^{-1}\mid\alpha\in\Gamma\;\tau,\sigma\in \pi\rangle.\]
\end{proposition}

We can now define the first Whitehead group of the pair $(\pi,\Gamma)$.
\begin{definition}
	$\Wh_1(\pi ;\Gamma )=\coker \left(\Gamma [1]\rightarrow K_1\Gamma[\pi]\right)$, where $\Gamma [1]\rightarrow K_1\Gamma[\pi]$ is defined by~$1\alpha\mapsto [1\alpha]$
\end{definition}
It follows easily that
\begin{corollary}\label{whtrace}
	$\Wh_1(\pi ;\Gamma )= \Gamma[\pi]/ \langle \alpha\sigma -\alpha^\tau \tau\sigma\tau^{-1}, \beta\cdot 1\mid\alpha,\beta\in\Gamma\; \tau,\sigma\in \pi\rangle$
\end{corollary}

\begin{remark}
It is clear from the definitions in this section that $\Wh_1(\pi_1 X;\mathbb{Z}_2\times\pi_2 X)$ splits as 
\[\Wh_1(\pi_1 X;\mathbb{Z}_2\times\pi_2 X) = \Wh_1(\pi_1 X;\mathbb{Z}_2)\oplus \Wh_1(\pi_1 X;\pi_2 X).\]
\end{remark}
\subsection{Construction of \texorpdfstring{$\Theta$}{Theta}}\label{constructionoftheta}
In this section we recall the definition of $\Theta$ set out in \cite[Chapter VII]{HatcherWagoner}. Let $f_t$ be a path in $\mathcal{F}(X)$ with endpoints in $\mathcal{E}(X)$. Suppose~$f_t$ lies in the kernel of $\Sigma$. We will associate to $f_t$ an element of $\Wh_1(\pi_1 X;\mathbb{Z}_2\times\pi_2 X)$. In the quotient~$\Wh_1(\pi_1 X;\mathbb{Z}_2\times\pi_2 X)/ \chi( K_3 \pi_1 X)$ this element will be an invariant of the homotopy class in $(\mathcal{P},\mathcal{E})$. This gives a map from $\ker\Sigma\subset\pi_1(\mathcal{F},\mathcal{E})$, and hence from~$\ker\Sigma\subset\pi_0(\mathcal{P})$, into $\Wh_1(\pi_1(X);\mathbb{Z}_2\times \pi_2(X))/ \chi( K_3 \pi_1 X)$ which in both cases we denote by $\Theta$. 

By Section \ref{reductiontoeyethm}, after a deformation we may assume $f_t$ is a collection of nested eyes, and that it has only handles of dimension $i$ and $i+1$ for $2\leq i\leq n-2$. As before we may assume that the births appear at times in $(0,\varepsilon)$ and the deaths occur at times in~$(1-\varepsilon,1)$, and that index $i$ critical points have critical value $<1/2$ and that index $i+1$ critical points have value $>1/2$. 

As previously we denote the middle level $f_t^{-1}(1/2)$ by $V_t$ for $t\in[\varepsilon,1-\varepsilon]$, and label the belt spheres for the $i$-handles and attaching spheres for the $(i+1)$-handles by $A_1,\ldots, A_N$ and $B_1,\ldots,B_N$ in $V_t$ respectively. In $V_\varepsilon$, $A_i\cap B_j$ consists of $\delta_{i,j}$ points.  As we move forward in the time direction we see a homotopy of the $A$-spheres and $B$-spheres which keeps the $B$-spheres disjoint and embedded, keeps the $A$-spheres disjoint and embedded, but possibly introduces intersections between the $A$-spheres and the $B$-spheres. By general position we may assume in that in each $t$ slice the intersection between $A$-spheres and~$B$-spheres consists of disjoint double points. In $V_{1-\varepsilon}$ again we see $A_i$ and~$B_j$ intersect in $\delta_{i,j}$ points. Note that there are no handle slides. 

Consider the trace of this homotopy in the trace of the middle level. We adopt the notation $I_\varepsilon = [\varepsilon,1-\varepsilon]$, and denote
\[V\times I_\varepsilon=\bigcup_{t\in I_\varepsilon} V_t \]
for the trace of $V_t$. For the trace of the $A$-spheres we adopt the notation
\[A_i\times I_\varepsilon=\bigcup_{t\in I_\varepsilon} A_i\]
and for the trace of the $B$-spheres
\[B_i\times I_\varepsilon=\bigcup_{t\in I_\varepsilon} B_i.\]

The intersections of $A_i\times I_\varepsilon$ and $B_j\times I_\varepsilon$ for each $i$, $j$ is a collection of lines and circles. In fact there is a single line of intersection in $A_i\times I_\varepsilon\cap B_i\times I_\varepsilon$, and all other intersection components are circles. 

We will associate to each circle of intersection $C$, elements $\gamma_C\in\pi_1(X)$, $\sigma_C\in\pi_2(X)$ and $s_C\in\mathbb{Z}_2$. We construct a matrix $M$ as follows
\[M_{i,j} = \sum_{\text{Circles of intersection } C \text{ between } A_i \text{ and } B_j}(\sigma_C + s_C)\gamma_C\]
Now $I+M$ is in $GL((\mathbb{Z}_2 \times \pi_2 X)[\pi_1 X])$; recall from Subsection \ref{wh1def} that
\[GL((\mathbb{Z}_2 \times \pi_2 X)[\pi_1 X]) = \{I+A\mid A\text{ has entries in }(\mathbb{Z}_2 \times \pi_2 X)[\pi_1 X]\}.\]
Hence $I+M$ specifies an element of $\Wh_1(\pi_1 X;\mathbb{Z}_2 \times \pi_2 X)$ as required.

\begin{remark} \leavevmode
	\begin{enumerate}
		\item To motivate why these circles of intersection matter, if there were no circles and only a ``straight line'' of intersection in $A_i\times I_\varepsilon\cap B_i\times I_\varepsilon$, that is a line intersecting each~$t$ slice in a single point, then in fact we could remove this pair from the family and ``cancel the eye''; that is deform the path $f_t$ to one whose Cerf graphic did not contain this eye.
		\item One might hope to remove a circle of intersection $C$ between $A_i\times I_\varepsilon$ and $B_j\times I_\varepsilon$  by some `higher dimensional Whitney move'. That is, by finding discs $D_a\subset A_i\times I_\varepsilon$ and~$D_b\subset B_j\times I_\varepsilon$ (the analogue of Whitney arcs), and finding an embedded 3-ball~$B^3\subset V\times I$ such that $\partial B^3=D_a\cup D_b$ (the analogue of a Whitney disc). One would also ask that this 3-ball is framed, with framing on the boundary agreeing with a ``Whitney Framing'' specified by the discs $D_a$ and $D_b$. One would also require that the intersection with any $V_t$ contained no $S^2$ components. If we could find such a framed 3-ball we could perform an analogue of a Whitney move to remove the circle of intersection $C$. The element $\sigma_C\in \pi_2 X$ is an obstruction to finding such a ball, and the element $s_C\in\mathbb{Z}_2$ is the obstruction to correctly framing such a ball. In high dimensions the vanishing of these obstructions imply we can find and frame such a ball, and so perform a 3-dimensional Whitney trick to remove circles of intersection.
	\end{enumerate}
\end{remark}

\subsubsection{Construction of $\gamma_C\in\pi_1 X$}
To obtain an element of $\pi_1 X$ associated to $C$, choose some time $t$ such that $V\times t$ intersects $C$ in some point $p$. pick an arc $a$ from $\ast_{A_i}$ to $p$, and an arc $b$ from $\ast_{B_j}$ to $p$. Then define
\[ \gamma_C = \alpha_i\cdot a\cdot b^{-1}\cdot\beta_j^{-1}\]
in the same way that we defined the algebraic intersection in Section \ref{algebraic-intersection}. To see independence from the choices of $t$, $p$, $a$ and $b$, note that $C$ is a circle in $A_i\times I_\varepsilon \cap B_j\times I_\varepsilon $, and that both $A_i\times I_\varepsilon $ and $B_j\times I_\varepsilon $ are simply connected, and so $C$ is null homotopic in both of these manifolds, and in $V\times I_\varepsilon $.

\subsubsection{Construction of $\sigma_C\in\pi_2 X$}
Since $A_i\times I_\varepsilon $ and $B_j\times I_\varepsilon $ are simply connected, the circle $C$ bounds discs $D_a\subset A_i\times I_\varepsilon $ and~$D_b\subset B_j\times I_\varepsilon $; note that these are possibly not embedded. We orient these discs by giving $C$ an orientation; here we use the convention that at any point which is a positive intersection of $A_i$ and $B_j$ the orientation of $C$ points in the positive $t$ direction. Now the union $D_a\cup D_b$, along with the arc $\alpha_i$ (by convention) to the base point $\ast$ gives an element of $\pi_2 X$. 

Note that both $A_i$ and $B_j$ are null-homotopic in $X$ for all $t$ (since they are belt spheres of $2$-handles and attaching sphers of $3$-handles), so $A_i\times I_\varepsilon$ and $B_i\times I_\varepsilon$ are null-homotopic in $X\times I_\varepsilon$ and so the above construction does not depend on the choice of $D_a$ and $D_b$

\subsubsection{Construction of $s_C\in\mathbb{Z}_2$}\label{z2contstruction}
First note that $A_i$ and $B_j$ intersect transversely, so we can identify the bundles
\[P\vcentcolon = \nu(C,A_i\times I_\varepsilon )=\nu(B_j\times I_\varepsilon , V\times I_\varepsilon )|_C\]
This bundle is a $(n-2)$-dimensional bundle over the circle. Since $A_j\times I_\varepsilon\cong S^{n-2}\times I$ is orientable, so is $\nu(C,A_j\times I_\varepsilon )$, so it must be the trivial bundle. In dimension 4 are $\mathbb{Z}$ many choices of trivialisation (framing) for this bundle, in higher dimensions there are~$\mathbb{Z}_2$ framings. The \emph{difference} of any two choices of framing gives a well defined element of~$\mathbb{Z}$ in dimension 4 or~$\mathbb{Z}_2$ in higher dimensions as in Remark \ref{Zframings}. We will construct two framings of this bundle, and the difference (taking mod 2 when in dimension four) is $s_C\in\mathbb{Z}_2$.

\begin{remark}
	One might ask, as Hatcher and Wagoner do in \cite{HatcherWagoner}, if it is possible that in dimension 4 there is in fact a $\mathbb{Z}$ valued invariant instead of a $\mathbb{Z}_2$ one. Igusa in \cite{Igusanew} proves that there is such an integer valued invariant on the space of ``marked lens space models'', that is paths in $\mathcal{F}$ which have a single eye, with some additional ``marking'' information. He shows that when the marking information is dropped the $\mathbb{Z}$ valued invariant only survives mod 2, suggesting a negative answer to this question; see \cite[Lemma 1.12]{Igusanew}.
\end{remark}

\subsubsection{Two framings of $P$}\label{two-framings}

For the first framing of $P$, consider $\nu(C,A_i\times I_\varepsilon )$; since $C$ is null-homotopic in $A_i\times I_\varepsilon $, it has a canonical \emph{0-framing}. We call this the $A$-0-framing.

For the second framing, we note that $B_j$ is the attaching sphere for a handle, so~$\nu(B_j,V)$ has a canonical framing that specifies how the handle is attached. Along with the isotopy of $B_j$ this induces a framing of $\nu(B_j\times I_\varepsilon,V\times I_\varepsilon)$. Taking the restriction of this framing to $\nu(B_j\times I_\varepsilon,V\times I_\varepsilon)|_C$ gives the second framing. We call this the $B$-attaching-framing.

The difference of the $A$-0-framing and the $B$-attaching-framing up to homotopy gives an element of $\mathbb{Z}_2$, (or in dimension 4 an element of $\mathbb{Z}$, which we then take mod 2 of) as in Remark \ref{Zframings}.

\subsection{Geometric description of 1-parameter families in \texorpdfstring{$\ker \Sigma$}{ker(Sigma)} in dimension 4}

We now describe some geometric features specific to dimension 4. We specify that~$X$ is a 4-manifold, with $f_t$ as in Section \ref{constructionoftheta}.

Now we have that $V_t=X\#^N S^2\times S^2$, and $A_1, \ldots, A_n\in V_t$ and $B_1,\ldots , B_n\in V_t$, are collections of 2-spheres. Again at $t=\varepsilon$ and $t=1-\varepsilon$. $A_i$ intersects~$B_j$ in $\delta_{i,j}$ points, and at times in-between there is a homotopy of the spheres with possibly extra intersections between the $A$-spheres and the $B$-spheres. We can perturb this homotopy so that it can be seen as a sequence of finger moves and Whitney moves between the $A$-spheres and the $B$-spheres. It is a well known fact that any homotopy can be deformed so that all of the finger moves occur before the Whitney moves; this is a consequence of the fact that by dimensionality the guiding arcs for the finger moves may be freely isotoped whilst avoiding any Whitney discs, see \cite[Section 4.1]{Quinn} for a detailed treatment. We arrange that the finger moves occur before time $1/2$ and the Whitney moves occur after time $1/2$.

All of the data of this homotopy can now be seen in the level $V_{1/2}$; we call this copy of~$V$ the \emph{middle-middle} level since it is $V_{1/2}=f^{-1}_{1/2}(1/2)$. At this time the spheres $A_i,B_j\in V_{1/2}$ intersect algebraically as~$\delta_{i,j}\in\mathbb{Z}[\pi_1 X]$, but possibly have more geometric intersections. We see a family of framed embedded Whitney discs~$W_1,\ldots, W_m$ which describe the Whitney moves that will be performed in the rest of the homotopy, and we also see framed embedded Whitney discs $U_1,\ldots, U_m$ which undo the finger moves that were just performed. The Whitney discs $U_1,\ldots,U_m$ can be made disjoint from each other, as the finger moves are determined by arcs, so by dimensionality we can perturb these arcs to be disjoint from each other. Dually (considering the homotopy running backwards) we can arrange that the discs $W_1,\ldots, W_m$ are disjoint from each other. We refer to the finger move discs as the \emph{$U$-discs} and the Whitney move discs as the \emph{$W$-discs}. The $U$-discs may intersect the $W$ discs.

\subsubsection{Dual Spheres}\label{dualspheres}

There are also several dual spheres present. At time $t=\varepsilon$, $A_i\cap B_i=\{p\}\in V_{\varepsilon}$, and both $A_i$ and $B_i$ are embedded with trivial normal bundle; indeed they are $p\times S^2$ and $S^2\times p$ fibers of an $S^2\times S^2$ summand in $V_\varepsilon=X\#^N S^2\times S^2$. Let $A_i^*$ be a parallel copy of $B_i$, and let $B_i^*$ be a parallel copy of $A_i$; see Figure \ref{fig:pushofftransversespheres}. Then $A_i^*$ is a dual sphere to $A_i$ since $A_i^*\cap A_i$ is a single point, $A_i^*\cap B_j=\emptyset$ for all $j$, and $A_i^*\cap A_j =\emptyset$ for~$i\neq q$. Similarly for $B_i^*$. Note that $A_i^*$ and $B_j^*$ consists of $\delta_{i,j}$ points.

\begin{figure}[htb]
	\begin{tikzpicture}
		\node[anchor=south west,inner sep=0] (image) at (0,0) {\includegraphics[width=0.18\textwidth]{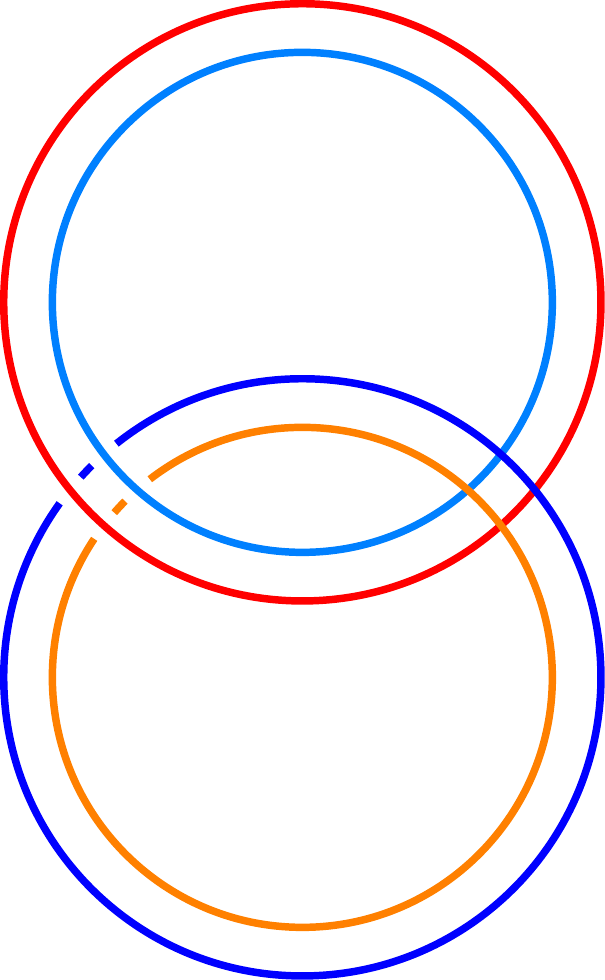}};

		\begin{scope}[x={(image.south east)},y={(image.north west)}]
			\node[text=red]  at (1.1,0.8)  {$A_i$};
			\node[text=blue]  at (1.1,0.2)  {$B_i$};
			\node[text=light-blue]  at (0.7,0.8)  {$B_i^*$};
			\node[text=orange]  at (0.7,0.2)  {$A_i^*$};

		\end{scope}
	\end{tikzpicture}
	\caption{\label{fig:pushofftransversespheres}The spheres $A_i$ and $B_i$, and the parallel copies which give transverse spheres $A_i^*$ and $B_i^*$. We see that $A_i^*$ intersects $A_i$ in a single point, is disjoint from~$B_i$, but intersects $B_i^*$ in a single point; similarly $B_i^*$ intersects $B_i$ in a single point but is disjoint from $A_i$.}
\end{figure}
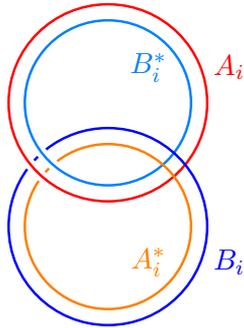

By dimensionality we can arrange that the arcs which determine the finger moves are disjoint from all dual spheres $A_1^*,\ldots,A_n^*, B_1^*,\ldots B_n^*$. It follows that in the middle-middle level the spheres $A_1^*,\ldots,A_n^*, B_1^*,\ldots B_n^*$ still intersect $A_1,\ldots,A_n, B_1,\ldots,B_n$ as above, and further are disjoint from the $U$-discs; they may however intersect the $W$-discs.

Noting that just before the deaths $A_i\cap B_j$ consists of $\delta_{i,j}$ points, we can repeat this argument running the homotopy backwards to obtain another set of dual spheres. Doing so we obtain (likely different) dual spheres $A_1^*,\ldots,A_n^*, B_1^*,\ldots B_n^*$ in the middle-middle level which are disjoint from the $W$-discs but not the $U$-discs. When we wish to emphasise the difference we will refer to the dual spheres disjoint from the $U$-discs as the \emph{initial dual spheres}, and as the dual spheres disjoint from the $W$-discs as the \emph{terminal dual spheres}.

\subsection{Calculation for geometrically simple families}\label{z2calculation}

For $X$ a 4-manifold we consider the following simple family of paths $f_t$ in $\ker\Sigma$. First, a 2-handle 3-handle cancelling pair is created with spheres $A$ and $B$ in the middle level. Then, a single finger move is performed between $A$ and $B$. Then, a single Whitney move is performed to remove the two intersections created by the finger move. Since $A$ and $B$ now intersect in a single point we then cancel them. 

Considering the middle-middle level $V=V_{1/2}$, we may without loss of generality assume that the initial Whitney disc $U$ (from the finger move) and the terminal Whitney disc $W$ share the same Whitney arcs $\alpha\in A$ and $\beta \in B$. Indeed if they did not, we can isotope~$W$ to arrange that this is true as follows. Denote the Whitney arcs for $W$ by $\alpha_W\subset A$ and $\beta_W\subset B$, and similarly $\alpha_U\subset A$, $\beta_U\subset B$ for the Whitney arcs of $U$. Note that $A$ intersects $B$ in three points, two of which are the endpoints of $\alpha_W$. Since two arcs in a disc $D^2$ with the same endpoints are always isotopic, there is an isotopy taking $\alpha_W$ to $\alpha_U$, which avoids the intersection points. We can extend this isotopy to the disc $W$ in a small neighbourhood of $A$. We may do the same in a neighbourhood of $B$ to make $\beta_W$ agree with $\beta_U$

The family $f_t$ has a single of circle of intersection $C$, so 
\[\Theta(f_t)= \left(\left(s_C+\sigma_C)\gamma_C\right)\right)\in \Wh_1(\pi_1(X),\mathbb{Z}_2\times\pi_2(X) ).\]
It is also clear that $\gamma_C\in\pi_1(X)$ is the element of $\pi_1(X)$ associated to the finger move arc as in Remark \ref{pi1finger}.

To calculate $\sigma_C$ we claim the following.
\begin{proposition}\label{claimone}
The element $\sigma_C\in\pi_2(X)$ can be represented by the 2-sphere $U\cup W$ joined to the basepoint $\ast$ by the basepoint arc for $A$, $\gamma_A$.
\end{proposition}

To describe $s_C$ we must work a little harder. Consider the bundle $\nu(B,V)|\beta$. For either disc $U$ and $W$ we have 1-dimensional sub-bundles $c_U, c_W$ of this bundle by taking the intersection with the normal bundle for the disc,
\begin{gather*}
c_U=\nu(B,V)|_\beta\cap \nu(U,V)|_\beta\\
c_W=\nu(B,V)|_\beta\cap \nu(W,V)|_\beta.
\end{gather*}
We arrange that the 1-bundles agree on $\nu(B,V)|_{\partial\beta}=\nu(\{p,q\}, A)=\{p,q\}\times D^2\subset A$. Note that both $U$ and $W$ have Whitney arc $\alpha$ in $A$ and so $c_U|_{p}$ and $c_W|_{p}$ are normal to $\alpha$ in $\nu(B, V)|_p = \nu(p,A) = p\times \mathring{D}^2$. Since in $\nu(p,A)$ there is only one choice of 1-bundle  which is normal to $\alpha\subset \nu(p,A)= p\times \mathring{D}^2$ (up to isotopy), we can arrange that $c_U|_{p}=c_W|_{p}$. Similarly we can arrange that  $c_U|_{q}=c_W|_{q}$. Arbitrarily we pick an orientation for these bundles to obtain two sections $c_W$ and $c_U$ of $\nu(B,V)|_{\beta}= S^2\times \mathring{D}^2$ which agree on $\nu(B,V)|_{\partial\beta}$. 

\begin{remark} Note that these sections are the Whitney sections for $W$ and $U$, but considered as a section of $\nu(B,V)|_{\beta}$ instead of $\nu(W,V)|_{\beta}$ and $\nu(U,V)|_{\beta}$. Recall that the definition of the Whitney section is asymmetric; we require the Whitney section for $W$ be a section of $\nu(W,V)|_{\partial W}$ which is parallel to $A$ on $\alpha$ and normal to $B$ on $\beta$. Hence on~$\beta$ the Whitney section is precisely $c_W=\nu(B,V)|_\beta\cap \nu(W,V)|_\beta$.
\end{remark}

Since $c_W$ and $c_U$ are sections of the same bundle and agree on the boundary we may consider the difference of these sections, which gives a well defined element of $\mathbb{Z}$ as in Remark \ref{Zframings}; note that Remark \ref{Zframings} refers to framings of $S^1\times D^2$, here we consider framings of $I\times D^2$ which are fixed on $\partial I\times D^2$, which can also be seen to determine an element of $\pi_1\GL_2(\mathbb{Z})=\mathbb{Z}$. 
\begin{proposition}\label{claimtwo}
Consider the difference of $c_U$ and $c_W$ as an element of $\mathbb{Z}$ as above. Taking this element of $\mathbb{Z}$ $\bmod\ 2$ gives $s_C\in\mathbb{Z}_2$.
\end{proposition}

We now proceed to prove both propositions.
\begin{proof}[Proof of Propositions \ref{claimone} and \ref{claimtwo}]

Consider $V\times I_\varepsilon$, say the finger move happens at time $\varepsilon \leq t_1<1/2$ and the Whitney move happens at time $1/2<t_2\leq 1-\varepsilon$. We consider the sweep out of the Whitney disc $W$ from time 1/2 to time $t_2$ as the Whitney move is performed; considering the isotopy the Whitney disc takes as the Whitney move is performed, the sweep out is the union over all $t$ of the Whitney disc. We denoted this sweep out by $Q_W\subset V\times[1/2,t_2]\subset V\times I_\varepsilon$. We similarly consider the sweep out of $U$, which we denote~$Q_U\subset V\times[t_1,1/2]\subset V\times I_\varepsilon$. See Figure \ref{fig:quarterball} for a depiction of $Q_U$.

\begin{figure}[htb]
	\centering
	\begin{tikzpicture}
		\node[anchor=south west,inner sep=0] (image) at (0,0) {\includegraphics[width=0.2\textwidth]{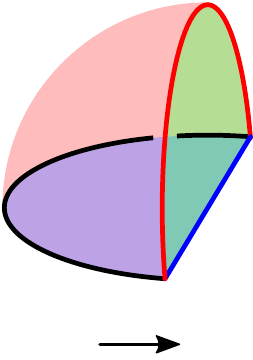}};

		\begin{scope}[x={(image.south east)},y={(image.north west)}]
			\node at (0.14*3.8,-0.05)  {$t$};
			\node[text=red]  at (0.05*3.8,0.92)  {$D^U_\beta$};
			\node[text=red]  at (0.27*3.8-0.05,0.95)  {$\beta$};
			\node[text=blue]  at (0.08*3.8-0.05,0.18)  {$D^U_\alpha$};
			\node[text=blue]  at (0.225*3.8,0.325)  {$\alpha$};
			\node  at (-0.03,0.325)  {$C$};
			
			\node[text=dark-green]  at (0.28*3.8,0.79)  {$U$};

		\end{scope}
	\end{tikzpicture}
	\caption{\label{fig:quarterball}The ball $Q_U$ swept out by the Whitney disc $U$ through the time direction (denoted $t$).}
\end{figure}
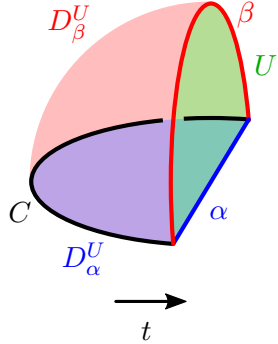

Clearly $Q_U$ is a ball, with boundary consisting of three discs, one of which is the disc~$U\subset V\times \{1/2\}$. Another boundary component is the sweep out of the Whitney arc~$\alpha_U\subset A$ denoted $D_U^\alpha =Q_U\cap (A\times I_\varepsilon)$. The final boundary component is the sweep out of the Whitney arc $\beta$ denoted $D^U_\beta= Q_U\cap (B\times I_\varepsilon)$. We similarly denote the sweep out of~$\alpha_W\subset A\times \{1/2\}$ by $D^W_\alpha\subset A\times[1/2, t_2]$ and sweep out of $\beta_W\subset A\times \{1/2\}$ by~$D^W_\beta\subset A\times[1/2, t_2]$.

To prove Proposition \ref{claimone}, note that $D^W_\beta\cup D^U_\beta$ gives a disc in $B\times I_\varepsilon$ with boundary $C$, and $D^W_\alpha\cup D^U_\alpha$ gives a disc in $A\times I_\varepsilon$ with boundary $C$. Hence $\sigma_C$ is represented by 
\[\Big(D^W_\alpha\cup D^U_\alpha\Big)\cup \Big(D^W_\beta\cup D^U_\beta\Big)\] along with the arc $\gamma_A$ from $A$ to the basepoint. Since $D^U_\alpha\cup D^U_\beta\cup U$ bounds $Q_U$, the disc $D^U_\alpha\cup D^U_\beta$ is isotopic to $U$. Similarly the disc $D^W_\alpha\cup D^W_\beta$ is isotopic to $W$ via $Q_W$. Hence
\[\Big(D^W_\alpha\cup D^U_\alpha\Big)\cup \Big(D^W_\beta\cup D^U_\beta\Big)=\Big(D^U_\alpha\cup D^U_\beta\Big)\cup \Big(D^W_\alpha\cup D^W_\beta\Big)\simeq U\cup W\]
proving the Proposition \ref{claimone}.

To prove Proposition \ref{claimtwo}, define a 1-dimensional sub-bundle of $\nu(B\times I_\varepsilon,V\times I_\varepsilon)|_{D^U_\beta}$ via
\[S_U = \nu(B\times I_\varepsilon,V\times I_\varepsilon)|_{D^U_\beta}\cap\nu(Q_U,V\times I_\varepsilon)|_{D^U_\beta}\]
and a 1-dimensional sub-bundle of $\nu(B\times I_\varepsilon,V\times I_\varepsilon)|_{D^W_\beta}$ via
\[S_W = \nu(B\times I_\varepsilon,V\times I_\varepsilon)|_{D^W_\beta}\cap\nu(Q_W,V\times I)|_{D^W_\beta}.\]
Note that on $\beta$
\[S_U|_\beta=\nu(B\times I_\varepsilon,V\times I_\varepsilon)|_{\beta}\cap\nu(Q_U,V\times I)|_{\beta}=\nu(B,V)|_{\beta}\cap\nu(U,V)|_{\beta}=c_U\]
and similarly $S_W|_\beta = c_W$. As we did for $c_W$ and $c_U$, we arrange that $S_W$ and $S_U$ agree on $\{p,q\}=\partial\beta$. We also pick orientations for $S_W$ and $S_U$ so that we can consider them as sections (picking the orientations so that they agree on $\{p,q\}$, and with $c_U$ and $c_W$).

As  $S_W|_C$ and $S_U|_C$ agree where they meet at $p,q$ the union $S_W|_C\cup S_U|_C$ gives a section of $\nu(B\times I_\varepsilon,V\times I_\varepsilon)|_C=\nu(C,A\times I_\varepsilon)=S^1\times D^2$; recall that this is the bundle which is used to define $s_C$. Consider the section $S$ of $\nu(C,A\times I_\varepsilon)$ which points into the disc~$D^W_\alpha\cup D^U_\alpha\subset A\times I_\varepsilon$. Clearly $S$ is the 0-framed section of $C$ in $A$ (by definition). Since $S_W$ is normal to $Q_W$, it is normal to $D^W_\alpha\subset Q_W$, and so $S_W|_C$ is normal to~$S$. Similarly~$S_U|_C$ is normal to $S$. 

Hence $S_W|_C\cup S_U|_C$ is normal to $S$ in $\nu(C,A\times I_\varepsilon)$. Since~$\nu(C,A\times I_\varepsilon)=S^1\times D^2$ this means that $S$ and $S_W|_C\cup S_U|_C$ must be isotopic. Hence~$S_W|_C\cup S_U|_C$ is the 0-framed section of $C$ in $A\times I$. Hence the framing induced by~$S_U|_C\cup S_W|_C$ is the $A$-0-framing of Section~\ref{two-framings}.

We divide the remainder of the proof into two cases.

\textbf{Case 1:}\textit{ on $\beta$ the sections $c_W=S_W|_\beta$ and $c_U=S_U|_\beta$ agree up to isotopy.}
 In this case (after possibly an isotopy) $S_W\cup S_U$ is a section of $\nu(B\times I_\varepsilon,V\times I_\varepsilon)_{D^U_\beta\cup D^W_\beta}$. This section necessarily extends to a section of $\nu(B\times I_\varepsilon,V\times I_\varepsilon)= (S^2\times I)\times D^2$. This section induces a trivialisation~$T$ of $\nu(B\times I_\varepsilon,V\times I_\varepsilon)$. Since $(S^2\times I)\times D^2$ has a unique trivialisation, $T$ must agree with the sweep out of the trivialisation used to attach the~3-handle.

Consider the trivialisation induced by~$(S_W\cup S_U)|_C$ on~$\nu(B\times I_\varepsilon,V\times I_\varepsilon)|_C$. This is necessarily the restriction of $T$ to $\nu(B\times I_\varepsilon,V\times I_\varepsilon)|_C$, and hence is the restriction of the~3-handle framing to $C$. This is precisely the $B$-attaching-framing of Section \ref{two-framings}, hence $(S_W\cup S_U)|_C$ induces the $B$-attaching-framing. 

Since by the above $(S_W\cup S_U)|_C=S_W|_C\cup S_U|_C$ induces the $A$-0-framing, and the $B$-attaching framing we see that they agree and so~$s_C=0$ as required.

\begin{figure}[htb]
	\centering
	\begin{tikzpicture}
		\node[anchor=south west,inner sep=0] (image) at (0,0) {\includegraphics[width=0.99\textwidth]{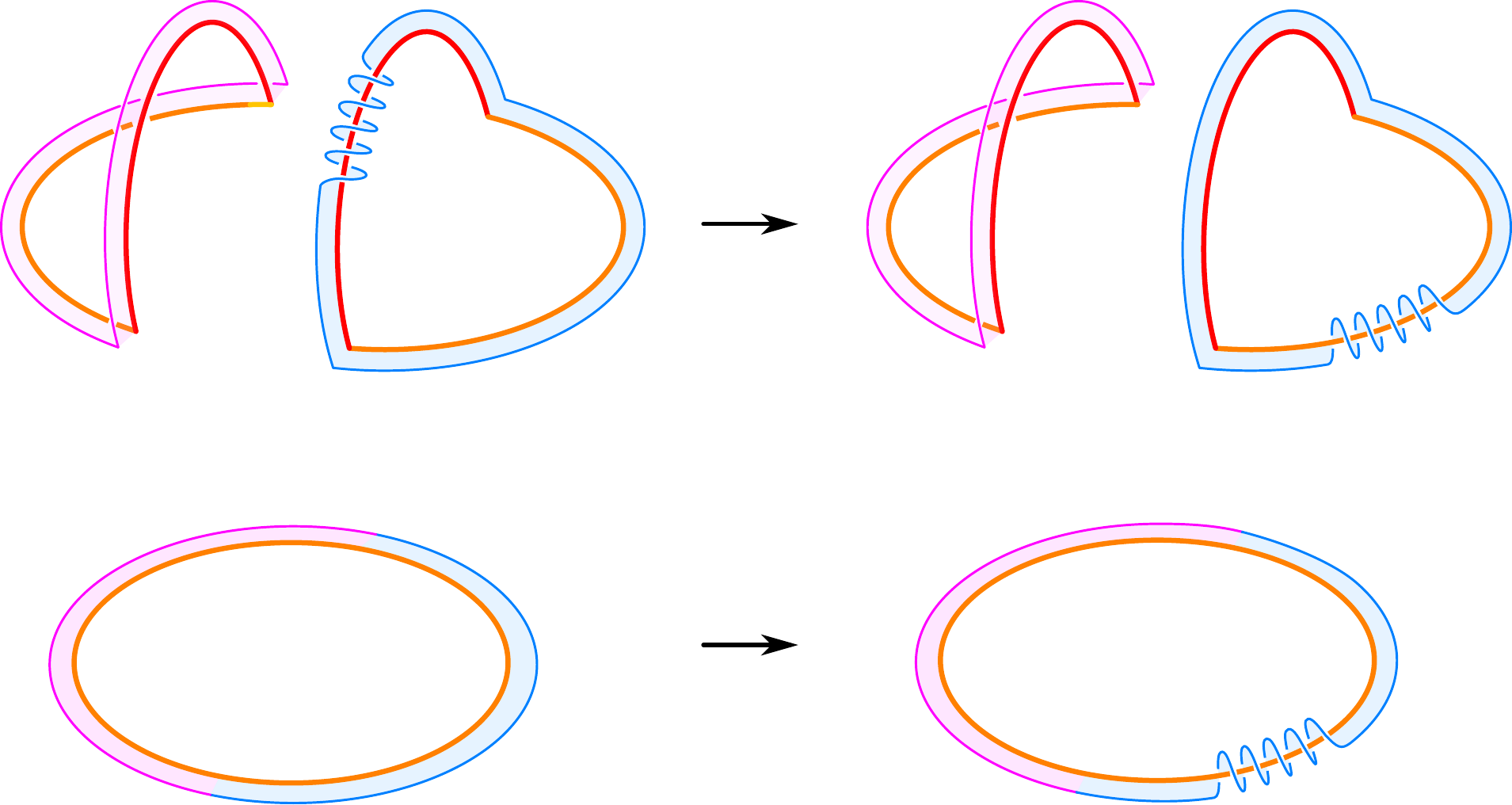}};

		\begin{scope}[x={(image.south east)},y={(image.north west)}]
			\node[text=red]  at (0.1,0.67)  {$\beta$};
			\node[text=red]  at (0.238,0.67)  {$\beta$};
			\node[text=orange]  at (0.15,0.82)  {$C\cap Q_U$};
			\node[text=orange, anchor=west]  at (0.3,0.7)  {$C\cap Q_W$};
			\node[text=mpurp, anchor=west]  at (-0.005,0.95)  {$S_U|_{\partial D_\beta^U}$};
			\node[text=light-blue, anchor=west]  at (0.325,0.95)  {$S_W|_{\partial D_\beta^W}$};
			
			\node[text=red]  at (0.1+0.573,0.67)  {$\beta$};
			\node[text=red]  at (0.238+0.573,0.67)  {$\beta$};
			\node[text=orange]  at (0.15+0.573,0.82)  {$C\cap Q_U$};
			\node[text=orange, anchor=west]  at (0.3+0.573,0.7)  {$C\cap Q_W$};
			\node[text=mpurp, anchor=west]  at (-0.005+0.573,0.95)  {$S_U|_{\partial D_\beta^U}$};
			\node[text=light-blue, anchor=west]  at (0.315+0.573,0.95)  {$S_W^\prime|_{\partial D_\beta^W}$};

			\node[text=mpurp, anchor=west]  at (0.03,0.369)  {$S_U|_{C\cap Q_U}$};
			\node[text=light-blue, anchor=west]  at (0.295,0.325)  {$S_W|_{C\cap Q_W}$};
			\node[text=orange]  at (0.15,0.08)  {$C$};
			
			\node[text=mpurp, anchor=west]  at (0.03+0.573,0.369)  {$S_U|_{C\cap Q_U}$};
			\node[text=light-blue, anchor=west]  at (0.295+0.573,0.325)  {$S_W^\prime|_{C\cap Q_W}$};
			\node[text=orange]  at (0.15+0.573,0.08)  {$C$};

		\end{scope}
	\end{tikzpicture}
	\caption{\label{fig:framing-twist} The operation of changing the section $S_W$ so that it agrees with $S_U$ on $\beta$. In the top left we depict $\nu(B\times I_\varepsilon,V\times I_\varepsilon)|_{D_\beta^W}$ and $\nu(B\times I_\varepsilon,V\times I_\varepsilon)|_{D_\beta^W}$, but for clarity only depict $\partial D_\beta^W$ and $\partial D_\beta^U$, and $S_W|_{\partial D_\beta^W}$ and $S_U|_{\partial D_\beta^U}$. In the top left we see that the sections $S_U|_\beta$ and $S_W|_\beta$ differ by $n=5$ twists. To obtain $S_W^\prime$ pictured on the right, we perform an isotopy which moves these twists to $\nu(B\times I_\varepsilon,V\times I_\varepsilon)|_{C\cap Q_W}$. 
Below we see that $S_W|_{C\cap Q_W}\cup S_U|_{C\cap Q_U}$ (bottom left) and $S_W^\prime|_{C\cap Q_W}\cup S_U|_{C\cap Q_U}$ (bottom right) differ by $\pm n=\pm 5$ twists (depending on orientation).}
\end{figure}
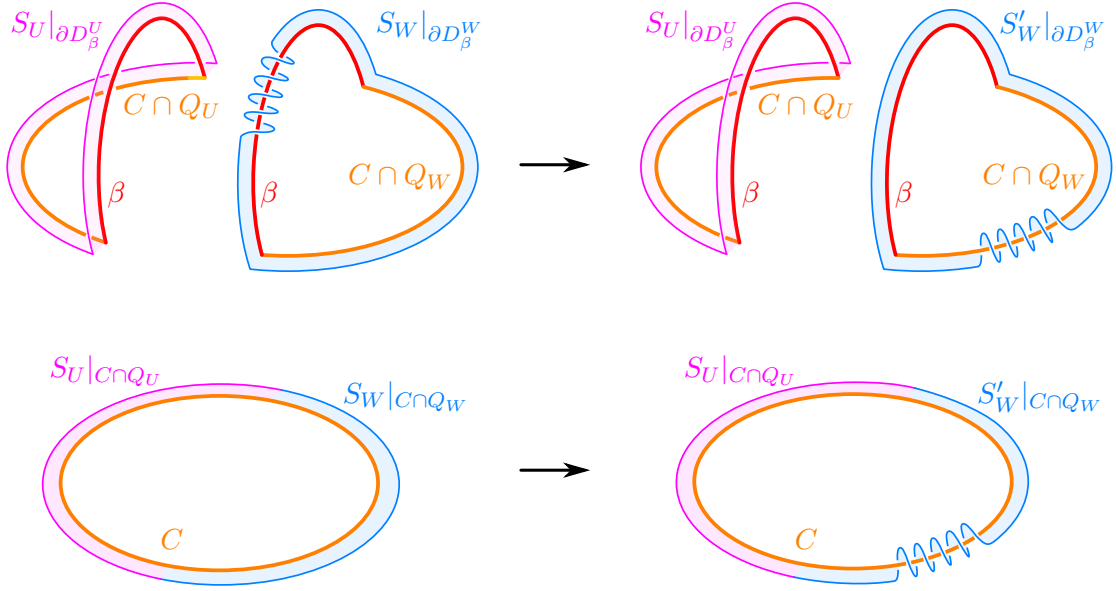

\textbf{Case 2:} \textit{on $\beta$ the sections $c_W=S_W|_\beta$ and $c_U=S_U|_\beta$ do not agree, but rather differ by $n$ twists.} Given a trivial disc bundle $D^2\times D^2\rightarrow D^2$, a section of this bundle~$s\colon D^2\rightarrow D^2\times D^2$, an arc $\gamma\subset\partial D^2$, and a section $r\colon \gamma\rightarrow \gamma\times D^2$, we can always perform an isotopy of the section $s$ to obtain a new section $s^\prime$ such that $s|_\gamma=r$. Applying this to the disc bundle $\nu(B\times I_\varepsilon,V\times I_\varepsilon)|_{D^W_\beta}$ and the section $S_W$, we can perform an isotopy of $S_W$ to obtain a section $S_W^\prime$ such that $S_W^\prime|_\beta$ agrees with $S_U|_\beta$. We note that doing so changes $S_W^\prime|_{C\cap Q_W}$. Since $S_W^\prime|_\beta$ differs from $S_W|_\beta$ by $n$ twists, it must be the case that~$S_W^\prime|_{C\cap Q_W}$ and $S_W|_{C\cap Q_W}$ differ by $\pm n$ twists (depending on the orientations), since the sections $S_W^\prime|_{\partial W}$ and $S_W|_{\partial W}$ must agree up to isotopy. See Figure \ref{fig:framing-twist} for a depiction of this operation.

Repeating the above argument for when the framings agreed, using the framings $S_W^\prime$ and $S_U$, we see that the framing induced by $(S_W^\prime\cup S_U)|_C$ is the~$B$-attaching-framing. Again the framing induced by $S_W|_C\cup S_U|_C$ is the the $A$-0-framing. Since $(S_W^\prime\cup S_U)|_C$ and $S_W\cup S_U$ differ by $\pm n$ twists, so do the~$B$-attaching-framing and the $A$-0-framing, hence $s_C = \pm n \mod 2 = n\mod 2$ as required.
\end{proof}

\section{Realisation theorem for \texorpdfstring{$\Wh_1(\pi_1 X; \mathbb{Z}_2 \times \pi_2 X)$}{Wh1(pi1(X);Z2xpi2(X))}}
	
In dimension $\geq 5$,
\[\Theta\colon\ker \Sigma\rightarrow\Wh_1(\pi_1 X; \mathbb{Z}_2 \times \pi_2 X)/\chi(K_3\mathbb{Z}[\pi_1 X])\]
is surjective. Hatcher and Wagoner prove this when $k_1 X=0$ \cite[Chapter VII]{HatcherWagoner}; note that they claim to prove surjectivity for general $k_1 X$, however their definition of $\Theta$ is only valid when $k_1 X=0$. The more general result was proved by Igusa~\cite{Igusa}. Hatcher and Wagoners proof considers $\Wh_1(\pi_1 X;\mathbb{Z}_2 \times \pi_2 X)$ as a quotient of \[(\mathbb{Z}_2 \times \pi_2 X)[\pi_1 X] = \langle(s+\sigma)\gamma\;|\; \sigma\in\pi_2 X,\;s\in\mathbb{Z}_2,\;\gamma\in\pi_1 X\rangle.\] 
They show that there exists a path $f_t\in\mathcal{F}$ with a single circle of intersection $C$ such that~$\gamma_C=\gamma$,~$\sigma_C=\sigma$ and~$s_C=s$ for arbitrary triples~$\gamma\in\pi_1 X$,~$\sigma\in\pi_2 X$,~$s\in\mathbb{Z}_2$. In dimension 4 we cannot see a way to realise all possible triples.

Given a 4-manifold $X$, we denote the ``Stiefel-Whitney classes'' by $w_1^X\colon\pi_1(X)\rightarrow \{\pm 1\}$, and~$w_2^X\colon\pi_2(X)\rightarrow \{0,1\}$. These are the usual Stiefel-Whitney classes composed with the Hurewicz map and reduction mod 2, except that for $w_1^X$ it will be useful for the target to be the multiplicative group of order two rather than the additive one, so if $w_1^X(\gamma)$ is trivial (usually 0) we send it to 1, if non-trivial (usually 1) we send it to $-1$. The map $w_1^X$ is often known as the ``orientation character''. We prove

\begin{proposition}\label{wh1realisation}
Let $X$ be a 4-manifold. Let $\sigma\in\pi_2 X$, $\gamma\in\pi_1 X$, and~$s\in\mathbb{Z}_2$. If~$w_2^X(\sigma)\neq 0$ or~$s=0$ then there exists a 1-parameter family $f_t$ with a single circle of intersection $C$ for which $\sigma_C=\sigma$, $\gamma_C=\gamma$, $s_C=s$.

That is to say we can realise all elements $(\sigma+s)\gamma\in(\mathbb{Z}_2 \times \pi_2 X)[\pi_1]$ except for those with~$w_2^X(\sigma)=0$ and $s=1$.
\end{proposition}

As a corollary we obtain Theorem \ref{Whoneimage}

\Whoneimage*

While we cannot realise all triples, we also cannot obstruct those exceptional values from being realised. Indeed \emph{topologically} one can find a homotopy of transverse spheres $A$ and $B$ with one circle of intersection realising all triples, as shown by \cite{Kwasik}.

We can however realise all values in $\Wh_1(\pi_1 X; \mathbb{Z}_2 \times \pi_2 X)$ by taking a single stabilisation with $S^2\times S^2$.

\begin{proposition}\label{proposition:stablewh1realisation}
Let $X$ be a 4-manifold. Note that $\pi_1 X$ and $\pi_2 X$ include as subgroups of $\pi_1 (X\# S^2\times S^2)$ and $\pi_2 (X\# S^2\times S^2)$ respectively. For all values $\sigma\in\pi_2 X$, $\gamma\in\pi_1 X$ and~$s\in\mathbb{Z}_2$ there exists a 1-parameter family $f_t\colon X\# S^2\times S^2\rightarrow I$ with a single circle of intersection $C$, such that~$\sigma_C=\sigma$, $\gamma_C=\gamma$, $s_C=s$.
\end{proposition}

Note that unlike the stable theorem for $\Sigma$, in this case, one $S^2\times S^2$ stabilisation is enough. Indeed for any generator $(\sigma+s)\gamma$ we can come up with a 1-parameter family $f_t$ and hence a corresponding pseudo-isotopy $F\in\mathcal{P}(X\#S^2\times S^2)$ with $\Theta(F)=(\sigma+s)\gamma$. The composition of such pseudo-isotopies gives sums of generators without needing additional $S^2\times S^2$ summands. This yields Theorem \ref{theorem:stablewh1realisation}.

\whonestablethm*

Note that unlike in the stable $\Sigma$ case $\Wh_1(\pi_1 X\# S^2\times S^2; \mathbb{Z}_2 \times \pi_2 (X\# S^2\times S^2))$ is potentially a larger group, and there may be additional elements that we cannot realise (at least not without taking further connect sums).

We proceed with the proofs of both propositions.

\begin{proof}[Proof of Proposition \ref{wh1realisation}]
Fix $\sigma\in\pi_2 X$ and $\gamma\in\pi_1 X$. As in Remark \ref{1-param-handles}, we build a path of Morse functions $f_t\in \mathcal{F}(X)$ by building a 1-parameter family of handle structures on~$X\times I$, as in the proof of Theorem \ref{stablesurjectivity}.

We can represent $\sigma$ by an immersed 2-sphere $S\subset X$ (with an arc $\eta$ to the base point), and we can arrange that $S$ intersects itself only at transverse double points.

We start with the trivial handle structure on $X\times I$. We then first create a cancelling 2-3 handle pair. When we do so, we ensure that the 2-3 handle creation takes place disjointly from $S\subset X$. We consider the middle level $V=X\# S^2\times S^2$.

We denote the 2-handle belt sphere and 3-handle attaching sphere by $A$ and $B$ in $V$ respectively; initially these are $p\times S^2$ and $S^2\times q$ in $V=X\# S^2\times S^2$. Note that we also have dual spheres $A^*$ and $B^*$ as in Section \ref{dualspheres}. Since we ensured the handle creation was disjoint from $S$, we see that $S$ is disjoint from~$A$,~$B$,~$A^*$, and~$B^*$.

Since $A\cup B$ is $\pi_1$-negligible in $V$, $\gamma$ determines a finger move from $A$ to $B$ as in Remark~\ref{pi1finger}. By dimensionality, we can ensure the finger move misses $S$ and the dual spheres $A^*$ and $B^*$. We perform this finger move to obtain a new handle structure. Afterwards, we see the Whitney disc $U$ that undoes this finger move. The interior of $U$ is disjoint from~$A$,~$B$,~$A^*$ and~$B^*$

Our plan now is to tube $U$ into $S$ to create a new Whitney disc $W$, then make $W$ embedded and perform a Whitney move. We will do this so that $U\cup W$ represents $\sigma$ and hence $\sigma_C=\sigma$, and so that $W$ has the correct framing on $\beta$ so that $s_C=s$. We deal with this in three cases.

\textbf{Case 1:}\textit{ $w_2^X(\sigma)=0$ and $s=0$.} First we perform interior twists to $S$ so that $\nu(S,V)$ is the trivial disc bundle over the sphere. This is possible because $w_2^X(\sigma)=0$. Now we tube $U$ to $S$ along an arc homotopic to $\gamma_A^{-1}\cdot \eta$ where $\gamma_A$ is the arc from $A$ to the base point; we ensure this arc misses any spheres. Since $U$ and $S$ are disjoint from $A$, $B$, $A^*$ and $B^*$, so is $W$. Note that $W\cup U$ (with the basepoint arc $\gamma_A$) represents $\sigma$. 

Since $\nu(S,V)$ is trivial, the framing of $\nu(U,V)$ extends to a framing of $\nu(W,V)$. Hence the disc-framing of $U$ and the disc framing of $W$ agree, so $W$ is a framed (immersed) Whitney disc. We now remove each self intersection of $W$ by pushing down into $A$ as in Section \ref{section:push-down}. Pushing down each intersection one by one into $A$, $W$ becomes embedded but now intersects $A$.

We may remove an intersection $p$ between $W$ and $A$ by using the dual sphere $A^\star$ and the Norman trick; we take a path through $A$ from $p$ to $A\cap A^*$, then tubing $W$ to a parallel copy of $A^*$ using this path. Note that $A^*$ is non-trivial in $\pi_2(V)$ but becomes trivial when included in $\pi_2(X\times I)=\pi_2(X)$, so this does not change the $\pi_2 X$ element given by $U\cup W$. We remove the intersections between $W$ and $A$ one by one, until $W$ and $A$ are disjoint.

Now $W$ is a framed embedded Whitney disc, disjoint from $A$ and $B$. We use~$W$ to perform a Whitney move. After this move the spheres intersect in a single point so we cancel the 2 and 3 handles. The resulting 1-parameter family $f_t$ has a single circle of intersection $C$. By Section \ref{z2calculation} we see that $\gamma_C=\gamma$. Since $U\cup W$ along with the base point arc $\gamma_A$ represents $\sigma$, we have that $\sigma_C=\sigma$ by Section \ref{z2calculation}. We also see that the sections $c_W$ and $c_U$ of Section \ref{z2calculation} agree; note that in a neighbourhood of the Whitney arc $\beta\subset B\subset V$, $W$ and $U$ agree by construction. Again by Section \ref{z2calculation} we see that $s_C=0$ as required.

\textbf{Case 2:}\textit{ $w_2^X(\sigma)=1$ and $s=0$.} In this case we cannot arrange $\nu(S,V)$ to be trivial, but we can perform interior twists to arrange it to be a bundle of Euler number 1 (note that an interior twist changes the Euler number of the normal bundle by two). We do this, then again tube $U$ into $S$ to form $W$. Doing so does not change the Whitney framing but does change the disc framing by a single twist, so $W$ is not framed. We can fix this however by performing a single boundary twist between $W$ and $A$. This makes~$W$ framed, but adds an intersection between $A$ and $W$. We can now push down the self-intersections of $W$ into $A$ as before, then remove intersections between $W$ and~$A$ (including the one from the boundary twist) using the dual sphere $A^*$.

As in Case 1 $W$ is now embedded so we perform a Whitney move, then cancel the handles.

Again, we see that in a neighbourhood of $\beta\subset B\subset V$, $W$ and $U$ agree (since we boundary twisted with $A$), so again the sections $c_W$ and $c_U$ agree, so as in Case 1 we have that $s_C=0$. It similarly follows that $\gamma_C=\gamma$ and $\sigma_C=\sigma$.

\textbf{Case 3:}\textit{ $w_2^X(\sigma)=1$ and $s=1$.} As in Case 2 we interior twist so that $\nu(S,V)$ has normal bundle one, then tube $U$ to $S$ to form $W$. To make $W$ framed we now perform a single boundary twist with $B$. 

We now push down the self intersections of $W$ into $B$, then resolve the intersections of $W$ and $B$ using $B^*$ to make $W$ embedded. Again we now perform the Whitney move and cancel the handles.

In the neighbourhood of $\beta$, $U$ and $W$ differ by a single twist (from the boundary twist we performed). Hence $c_W$ and $c_U$ also differ by a single twist. Hence $\gamma_C=\gamma$, $\sigma_C=\sigma$ and~$s_C=1$ by Section \ref{z2calculation}.
\end{proof}

Before we prove the stable version of the theorem, it is useful to see what goes wrong when we try to realise $s_C=1$ when $w_2^X(\sigma)=0$. Using the notation from the proof of Proposition \ref{wh1realisation} we consider the Whitney disc $U$ which undoes the finger move and tube it to $S$, which we arranged to have trivial normal bundle, to form a new Whitney disc~$W$. Since we would like that $c_W$ and $c_U$ do not agree (mod 2), we would like to do an odd number of boundary twists of $W$ with $B$. However doing so unframes the disc, and this cannot be fixed using interior twists (since these change the disc framing by an even number). Hence we must perform at least one boundary twist of $W$ with $A$. However if we do then $W$ intersects both $A$ and $B$. This is a problem as the dual spheres $A^*$ and $B^*$ intersect; we cannot use the Norman trick on both $A^*$ and~$B^*$ without creating new self intersections of $W$. 

We can however make the above construction work in the case that we have an $S^2\times S^2$ summand. We will use this summand to build an additional dual sphere, which can be used to resolve the intersections.

\begin{proof}[Proof of Proposition \ref{proposition:stablewh1realisation}]

The only case left to realise is when $w_2^X (\sigma)=0$ and $s=1$. Consider~$X\# S^2\times S^2$, let $Y=S^2\times S^2$, let $X'= X\# Y$ and let $\widehat{X}=X\setminus B^4\subset X'$. 

Given $\sigma \in \pi_2 X$, we again represent $\sigma$ by an immersed sphere $S\subset \widehat{X}$ with intersection points transverse double points; note that $\pi_2\widehat{X}=\pi_2 X$. We again arrange that $\nu(S,\widehat{X})$ is the trivial disc bundle over $S^2$ by performing interior twists to $S$.

We again construct a 1-parameter family $f_t$ for the 4-manifold $X\# S^2\times S^2$. Following the steps in the proof of Proposition \ref{wh1realisation} we first create a cancelling 2-3 handle pair in~$\widehat{X}\times I$, and consider the middle level of $\widehat{X}\times I$, denoted $\widehat{V}=\widehat{X}\# S^2\times S^2$ which is a subset of the middle level of $X\times I$, denoted $V=X\# Y\# S^2\times S^2$. Again we ensure that this 2-3 handle pair is created in the complement of of $S$. To distinguish this copy of $S^2\times S^2$, we let $Z=S^2\times S^2$ and say $\widehat{V} = \widehat{X}\#Z$, and note that $V = \widehat{V}\cup (Y\setminus B^4)$

We again denote the 2-handle belt sphere and 3-handle attaching spheres by $A$ and $B$ respectively. We now perform a finger move, in $\widehat{V}$, corresponding to $\gamma\in\pi_1 X=\pi_1 \widehat{X}$, ensuring this finger move misses $A$, $B$, $A^*$, and $B^*$. In $\widehat{V}$ we again see a Whitney disc~$U$ that undoes the move. We again tube $U$ to $S$ to create an immersed framed Whitney disc~$W$. 

To obtain the correct value for $s_C$, we now perform a single boundary twist of $W$ with $B$. We then perform another boundary twist on $W$ with $A$ (an opposite twist), so that $W$ remains a framed Whitney disc. However $W$ now intersects $A$ and $B$. We resolve the intersection single between $W$ and $A$ using the Norman trick on $A^*$. We then the resolve the single intersection with $B$ using the Norman trick on $B^*$, noting that this adds a single further self-intersection of $W$ as the parallel copy of $A^*$ intersects the parallel copy of $B^*$.

We consider the Clifford torus $T\subset \widehat{V}$ for the Whitney disc $W$. The torus $T$ intersects~$W$ in exactly one point, but is disjoint from $A$, $B$, $A^*$ and $B^*$. See Figure \ref{fig:cliffordtorus} for a depiction of this torus.

\begin{figure}[ht]
	\begin{tikzpicture}
		\node[anchor=south west,inner sep=0] (image) at (0,0) {\includegraphics[width=\textwidth]{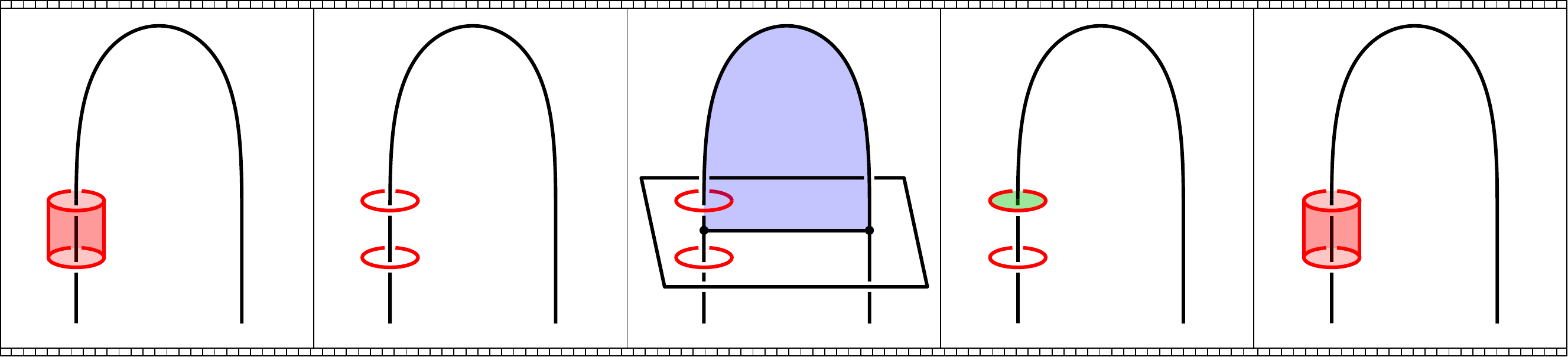}};
		\begin{scope}[x={(image.south east)},y={(image.north west)}]
			\node at (0.53,0.253)  {$B$};
			\node  at (0.563,0.8)  {$A$};
			\node[text=dark-green]  at (0.687,0.42)  {$D_A$};
			\node[text=red] at (0.082,0.36)  {$T$};
			\node[text=blue] at (0.5,0.7)  {$W$};
		\end{scope}
	\end{tikzpicture}
	\caption{\label{fig:cliffordtorus}We present 4-dimensional space as slices of 3-dimension space. In the middle slice we see the Whitney disc $W$ and a disc subset of the sphere $B$. In each slice we see a line from $A$, which sweep out a disc subset from $A$. We picture the Clifford torus $T$, and see that it intersects $W$ in exactly one point, but does not intersect $A$ or $B$. In the fourth frame we depict the cap $D_A$, which intersects $A$ in exactly one point. To see $D_B$ one can draw a new picture with the roles of $A$ and~$B$ reversed. }
\end{figure}
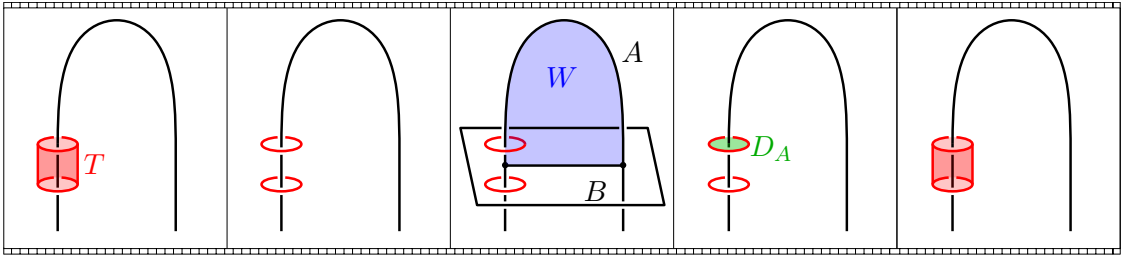

There are two \emph{caps} for $T$, namely embedded discs $D_A$ and~$D_B$, which intersect $T$ exactly on $\partial D_A$ and $\partial D_B$; we picture $D_A$ in the fourth frame in Figure \ref{fig:cliffordtorus}. The disc $D_A$ intersects $A$ in exactly one point but is disjoint from $W$, $B$, $A^*$ and $B^*$, and the disc~$D_B$ intersects~$B$ in exactly one point but is disjoint from $W$, $A$, $A^*$ and $B^*$. Further~$D_A$ and~$D_B$ intersect only in a single point on their boundary; see Figure \ref{fig:cliffordcaps}.

\begin{figure}[ht]
	\begin{tikzpicture}
		\node[anchor=south west,inner sep=0] (image) at (0,0) {\includegraphics[width=0.35\textwidth]{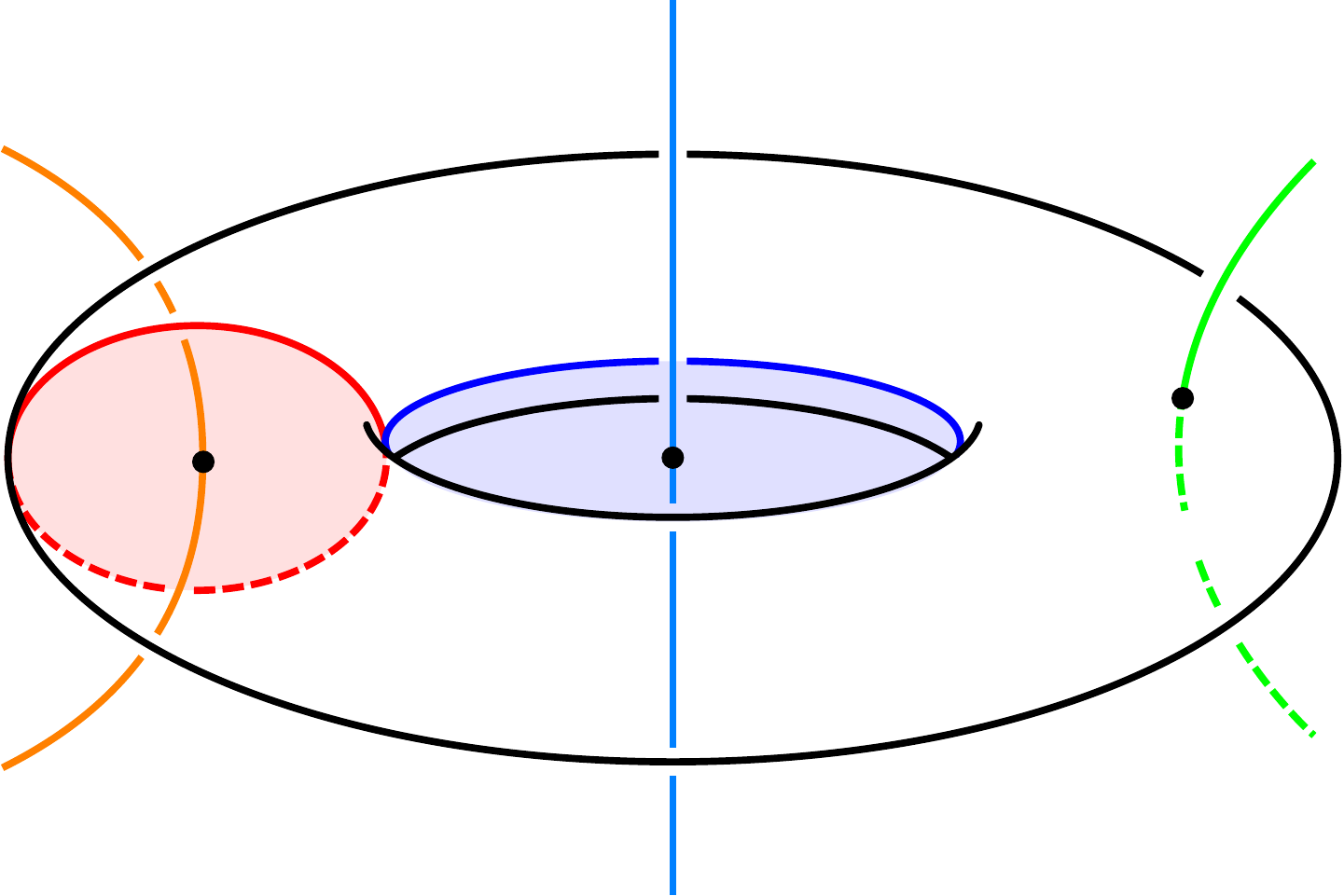}};
		\begin{scope}[x={(image.south east)},y={(image.north west)}]
			\node[text=red]  at (0.21,0.49)  {$D_A$};
			\node[text=blue]  at (0.57,0.49)  {$D_B$};
			\node[text=orange]  at (0.09,0.82)  {$A$};
			\node[text=light-blue]  at (0.535,0.93)  {$B$};
			\node[text=green]  at (0.92,0.8)  {$W$};
			
			\node at (0.73,0.32)  {$T$};
		\end{scope}
	\end{tikzpicture}
	\caption{\label{fig:cliffordcaps}The Clifford torus with the caps. Note that $A$ and $B$ do not intersect~$T$, and that $W$ only intersects $T$ in a single point.}
\end{figure}
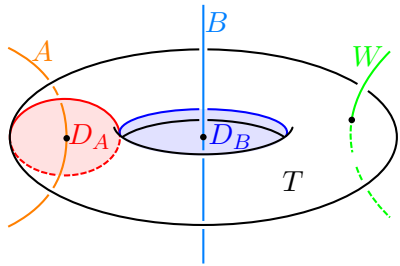
We now manipulate the caps. We direct the reader to Figure \ref{fig:wh1hardcase} for this manipulation.

\begin{figure}[ht]
	\begin{tikzpicture}
		\node[anchor=south west,inner sep=0] (image) at (0,0) {\includegraphics[width=0.7\textwidth]{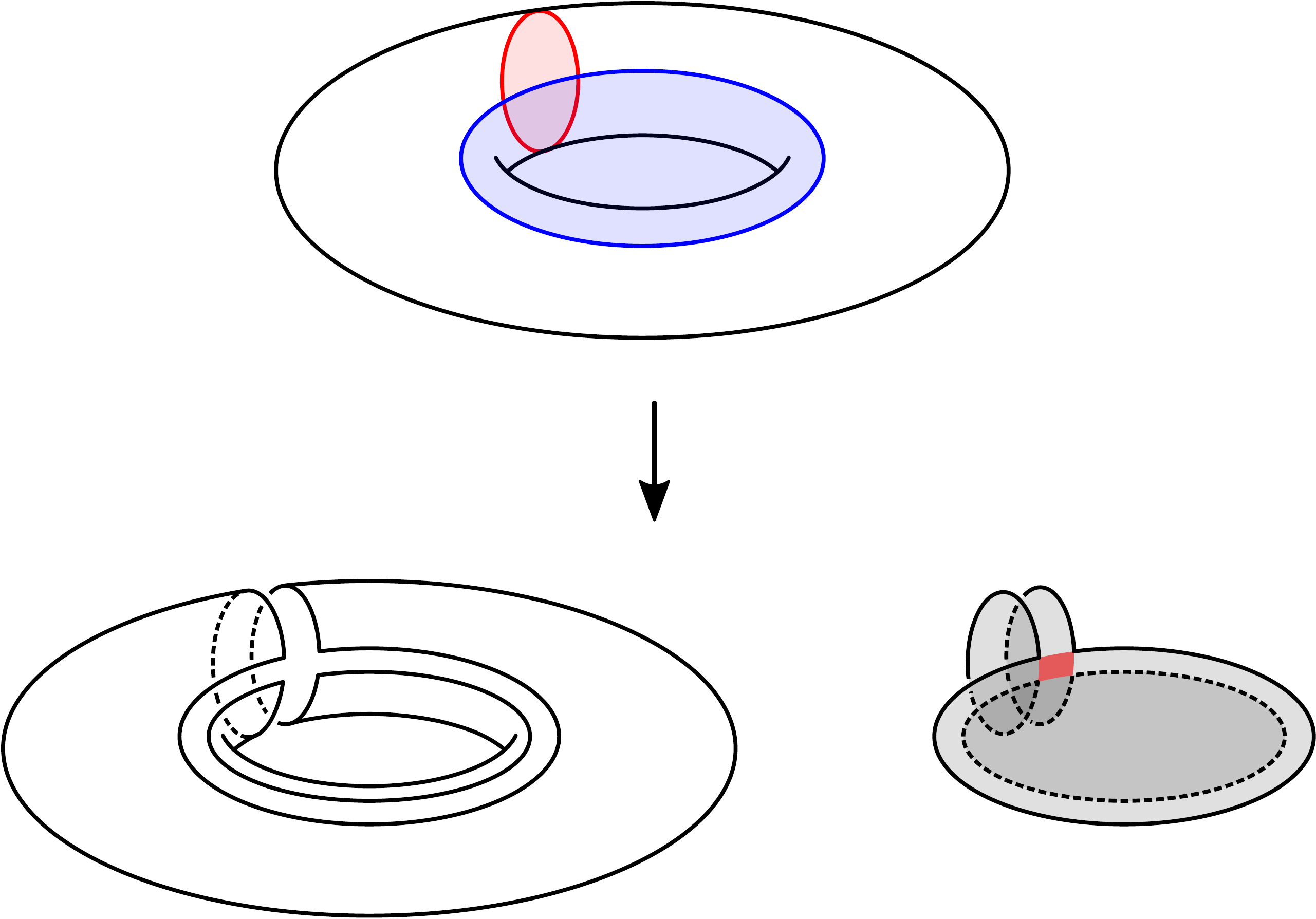}};
		\begin{scope}[x={(image.south east)},y={(image.north west)}]
			\node[text=red] at (0.35,0.92) {$D_A$};
			\node[text=blue] at (0.49,0.81) {$D_B$};
			\node  at (0.69,0.82)  {$T$};
			\node at (0.64,0.195)  {$\bigcup$};
		\end{scope}
	\end{tikzpicture}
	\caption{\label{fig:symetriccapping}The symmetric capping operation. In the top picture we see $T$ and the caps $D_A$ and $D_B$. In the bottom picture we see the two parts of the symmetric capping; left we see $T$ with the neighbourhoods of $\partial D_A$ and $\partial D_B$ removed, and right the parallel copies of the caps, with the extra square around $\partial D_A\cap\partial D_B$ which we glue back in; we have highlighted this square glued back in in red.}
\end{figure}

We remove the intersection of $D_A$ and $A$ using the Norman trick, tubing it into $A^*$. We similarly remove the intersection of $D_B$ and $B$ using $B^*$, noting that this adds intersections with $W$, and with the dual spheres $A^*$ and $B^*$, and adds a single point of intersection between $D_A$ and $D_B$. See the first picture of Figure \ref{fig:wh1hardcase} for a depiction of the resulting discs.

We can now perform a ``symmetric capping'' operation of Freedman and Quinn \cite{Freedman1990} using the two discs $D_A$ and~$D_B$. To do this, we remove the neighbourhood of $\partial_A\cup\partial_B$ from $T$ (Noting these circles intersect in a single point) and glue back in two parallel copies of $D_A$, two parallel copies of $D_B$, and we glue back in a square in $T$ around $\partial D_A\cap \partial D_B\in T$ to fill the resulting hole; see Figure \ref{fig:symetriccapping}. We smooth the edges of the resulting sphere and denote the result of this operation by $P$.

Note that since $D_A$ and $D_B$ have a single point of intersection, $P$ is an immersed sphere with four double points. $P$ is disjoint from $D_A$ and $D_B$, but since the caps intersected $W$, $P$ has additional points of intersection with $W$. See the top left and bottom left pictures of Figure \ref{fig:intersectionremoval}; in the top left we see an intersection of $D_B$ and $W$, in the bottom left we see how this introduces two points of intersection between $W$ and $P$.

We may however remove these new intersections of $W$ and $P$ in the following way. First, consider the intersections of $D_B$ and $W$. We may push down all such intersections into $T$ so that the intersection points of $W$ and $T$ lie on $\partial D_A$; see the top right picture in Figure \ref{fig:intersectionremoval}. Now when we take the symmetric capping to obtain $P$, $W$ will not intersect~$P$; see the bottom right picture in Figure \ref{fig:intersectionremoval}. We now do the same thing for intersections between~$W$ and~$D_A$. Note that this operation introduces extra self-intersections of~$W$; indeed each pair of intersections $p\in D_A\cap W$, $q\in D_B\cap W$ gives rise to a single self-finger move of $W$, which is performed when we push these intersections down into~$T$. See the second and third picture of Figure \ref{fig:wh1hardcase} for another depiction of the symmetric capping operation. 

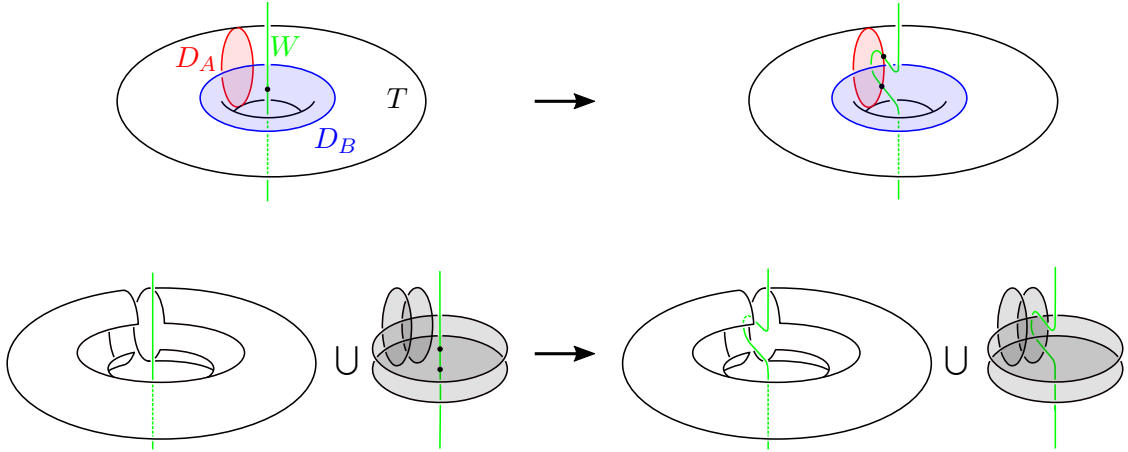
\begin{figure}[ht]
	\begin{tikzpicture}
		\node[anchor=south west,inner sep=0] (image) at (0,0) {\includegraphics[width=\textwidth]{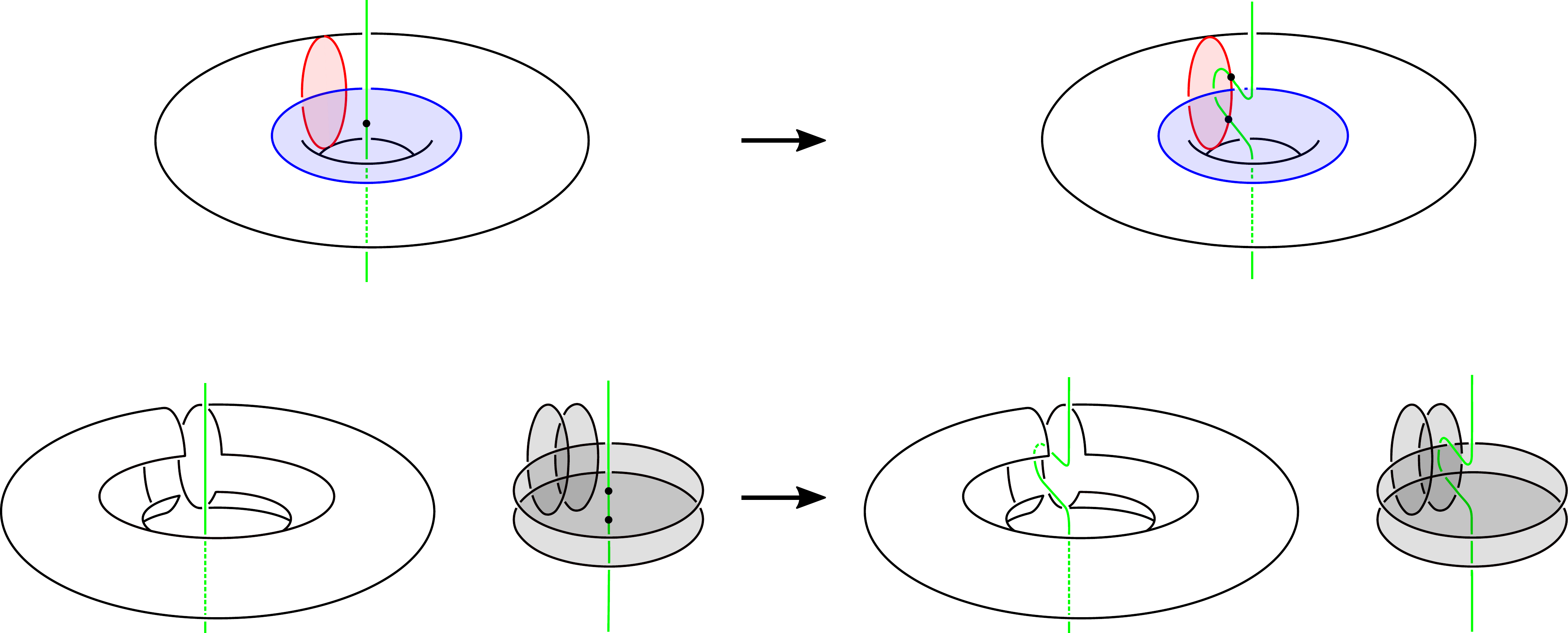}};
		\begin{scope}[x={(image.south east)},y={(image.north west)}]
			
			\node[text=green] at (0.25,0.901) {$W$};
			\node[text=blue] at (0.295,0.69) {$D_B$};
			
			\node[text=red] at (0.17,0.872)  {$D_A$};
			\node  at (0.35,0.785)  {$T$};
			\node  at (0.305,0.2)  {$\bigcup$};
			\node  at (0.85,0.2)  {$\bigcup$};
		\end{scope}
	\end{tikzpicture}
	\caption{\label{fig:intersectionremoval}Removing intersections with $S$ which arise from the intersections with the caps. In the top figure we see $T$, the caps $D_A$ and $D_B$ and a point of intersection between $U$ and $D_B$ (note that $W$ continues into the past and future). In the bottom left we see the two parts of the symmetric capping $P$; we draw them separately to make the picture clear. We picture $W$ in both these parts (we draw $W$ twice to show how it interacts with both parts of the symmetric capping). In the top right figure we see the result of pushing down the intersection of $W$ and $D_B$ into $T$. In the bottom left we see that $W$ now misses the symmetric capping.}
\end{figure}

\begin{figure}[p]
	\begin{tikzpicture}
		\node[anchor=south west,inner sep=0] (image) at (0,0) {\includegraphics[width=0.78\textwidth]{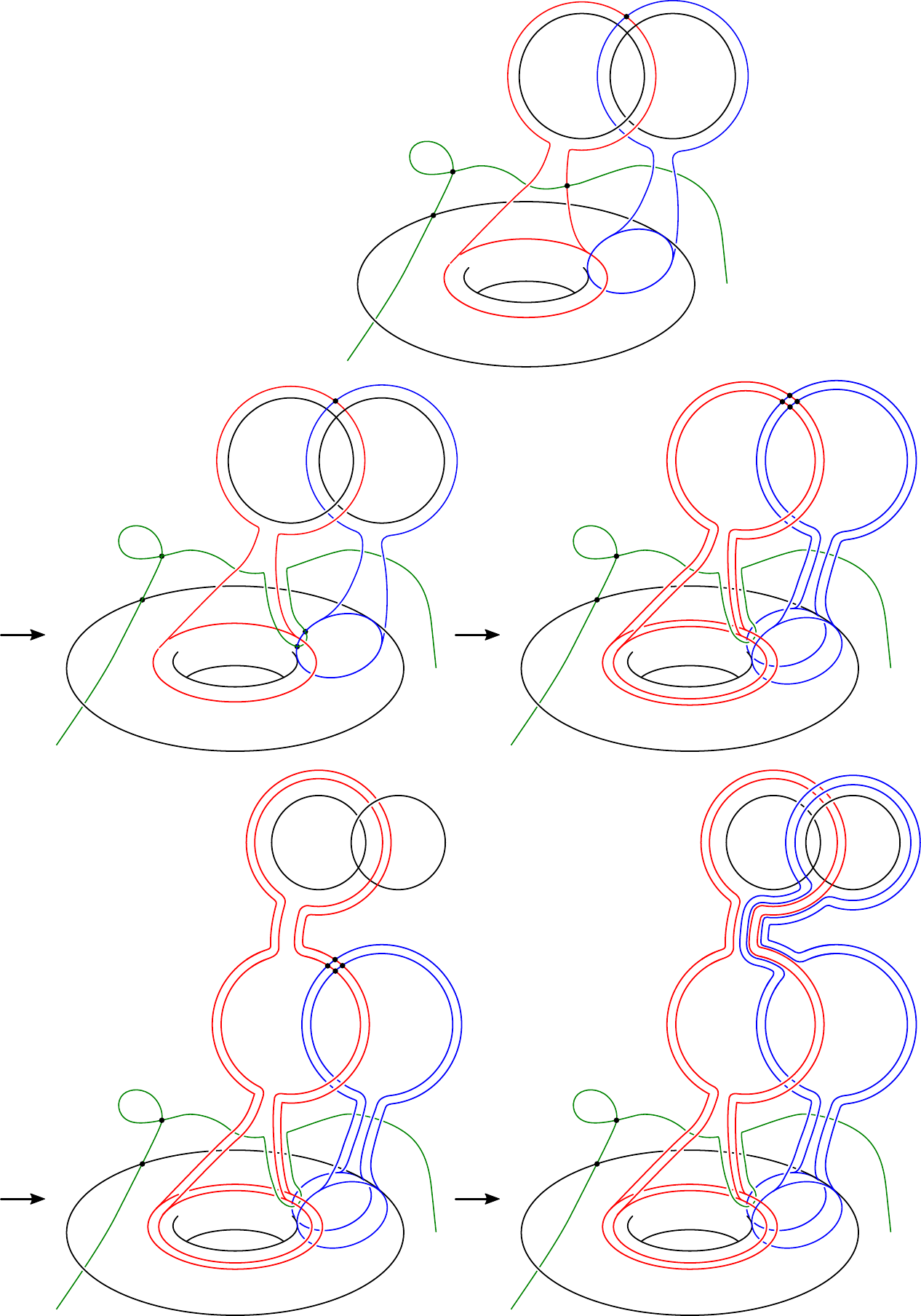}};
		\begin{scope}[x={(image.south east)},y={(image.north west)}]
			\node[text=red]  at (0.52,0.95)  {$D_A$};
			\node[text=blue]  at (0.85,0.95)  {$D_B$};
			\node[text=dark-green]  at (0.43,0.9)  {$W$};
			\node  at (0.4,0.825)  {$T$};
			
			\node[text=red]  at (0.2,0.95-0.29)  {$D_A$};
			\node[text=blue]  at (0.53,0.95-0.29)  {$D_B$};
			\node[text=dark-green]  at (0.43-0.32,0.9-0.29)  {$W$};
			\node  at (0.4-0.32,0.825-0.29)  {$T$};
			
			\node[text=red]  at (0.205+0.49,0.95-0.29)  {$D_A^+$};
			\node[text=red]  at (0.205+0.49+0.075,0.95-0.29)  {$D_A^-$};
			\node[text=blue]  at (0.53+0.49,0.95-0.29)  {$D_B^+$};
			\node[text=blue]  at (0.53+0.49-0.075,0.95-0.29)  {$D_B^-$};
			\node[text=dark-green]  at (0.43-0.32+0.49,0.9-0.29)  {$W$};
			\node  at (0.4-0.32+0.49,0.825-0.29)  {$T$};

			\node[text=red]  at (0.205,0.95-0.29-0.43)  {$D_A^+$};
			\node[text=red]  at (0.205+0.075,0.95-0.29-0.43)  {$D_A^-$};
			\node[text=blue]  at (0.53,0.95-0.29-0.43)  {$D_B^+$};
			\node[text=blue]  at (0.53-0.075,0.95-0.29-0.43)  {$D_B^-$};
			\node[text=dark-green]  at (0.43-0.32,0.9-0.29-0.43)  {$W$};
			\node  at (0.4-0.32,0.825-0.29-0.43)  {$T$};
			\node[anchor=west]  at (0.155,0.35)  {$p\times S^2$};
			\node[anchor=west]  at (0.475,0.35)  {$S^2\times p$};

			\node[text=red]  at (0.205+0.49,0.95-0.29-0.43)  {$D_A^+$};
			\node[text=red]  at (0.205+0.49+0.075,0.95-0.29-0.43)  {$D_A^-$};
			\node[text=blue]  at (0.53+0.49,0.95-0.29-0.43)  {$D_B^+$};
			\node[text=blue]  at (0.53+0.49-0.075,0.95-0.29-0.43)  {$D_B^-$};
			\node[text=dark-green]  at (0.43-0.32+0.49,0.9-0.29-0.43)  {$W$};
			\node  at (0.4-0.32+0.49,0.825-0.29-0.43)  {$T$};
			
		\end{scope}
	\end{tikzpicture}
	\caption{\label{fig:wh1hardcase}We depict the manipulation of the Clifford caps. To obtain the first picture we tube $D_A$ into $A^*$ and $D_B$ into $B^*$. To obtain the second picture from the first we push down the intersections between $W$ and $D_A$ into $T$ (we do the same for any intersections between $W$ and $D_B$). To obtain the third picture we perform the symmetric capping operation, and note that $W$ is now disjoint from the discs, and from $T\setminus(\partial D_A\cup \partial D_B)$. In the third picture we see the sphere $P$ as the union of $T\setminus(\partial D_A\cup \partial D_B)$, the caps, and the square around $\partial D_A\cap \partial D_B$. To obtain the fourth picture we perform a two sheeted Norman trick, tubing $D_A^\pm$ into $p\times S^2\subset Y\setminus B^4$. To obtain the final picture we perform a further two sheeted Norman trick, tubing $D_B^\pm$ into  $S^2\times p\subset Y\setminus B^4$. In the final picture we see $P'$ as the union of $T\setminus(\partial D_A\cup \partial D_B)$, the caps, and the square around $\partial D_A\cap \partial D_B$. We note that $P'$ is embedded with a single point of intersection with $W$.}
\end{figure}

After this process, $W$ and $P$ intersect only in a single point (the point in which the original Clifford torus intersected $W$). $W$ has many self intersections, but still does not intersect $A$ or $B$ in its interior. Further, $P$ is disjoint from $A$ and $B$. $P$ is immersed with 4 double points which arose from the point of self intersection between $D_A$ and $D_B$. Hence we can see the intersection points of $P$ as the intersection of two parallel copies of $D_A$, denoted $D_A^+$ and $D_A^-$, and two parallel copies of $D_B$, denoted $D_B^+$ and $D_B^-$; see the third picture of Figure \ref{fig:wh1hardcase}.

Since so far we worked entirely within $\widehat{V}$, we may now use the $S^2\times S^2$ summand
\[Y\setminus B^4= S^2\times S^2\setminus B^4\subset V = \widehat{V}\cup (Y\setminus B^4)\]
to resolve the self-intersections of $P$. First, we take the parallel sheets $D_A^+$ and $D_A^-$ and tube them to two parallel copies of $p\times S^2\subset Y\setminus B^4$; see Figure \ref{fig:twosheetedtubing}. We tube using some arc which we make disjoint from any spheres and discs by dimensionality. See the fourth picture of Figure \ref{fig:wh1hardcase} for a depiction of the resulting discs. Another way to see this operation is to consider tubing $D_A$ into $p\times S^2$, then taking two parallel copies of the resulting surface; hence $D_A^\pm$ are still parallel copies of some surface. 

\begin{figure}[ht]
	\begin{tikzpicture}
		\node[anchor=south west,inner sep=0] (image) at (0,0) {\includegraphics[width=0.7\textwidth]{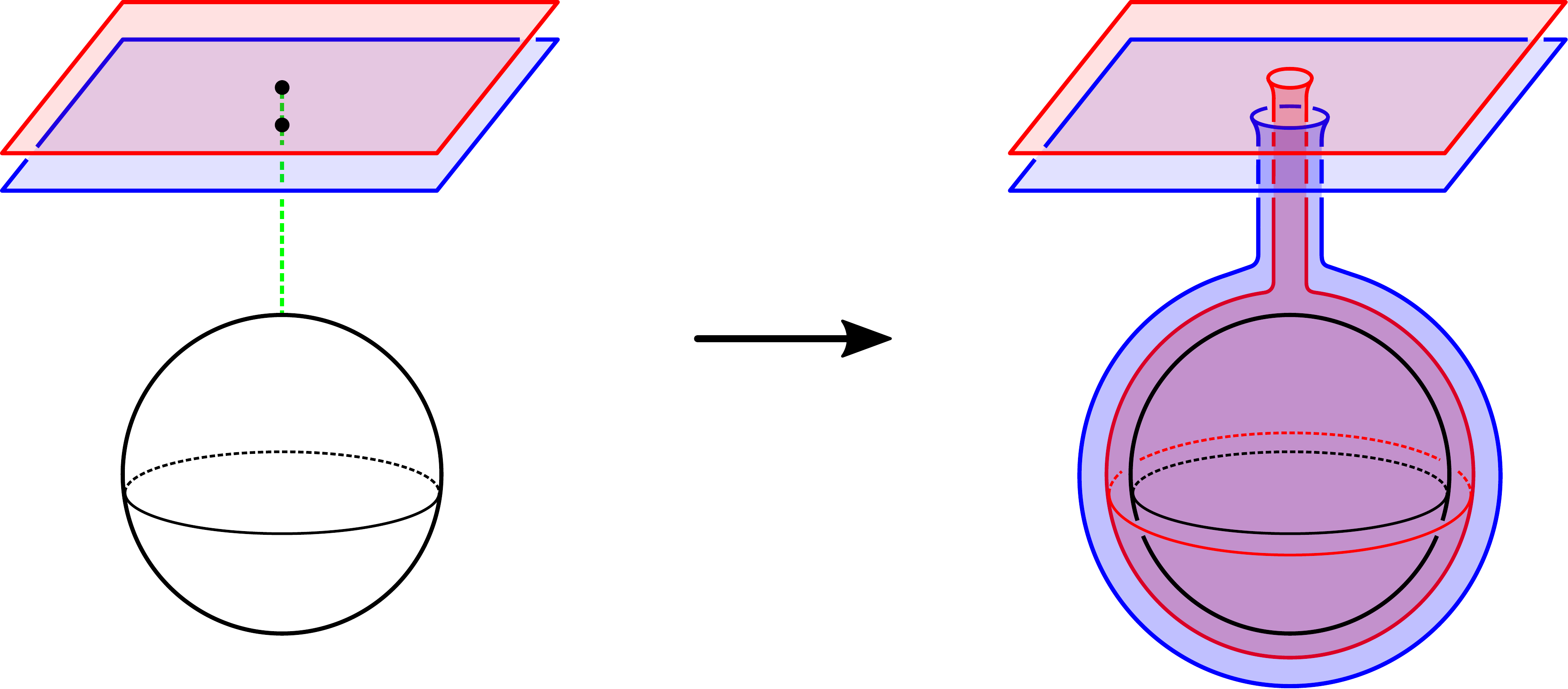}};
		\begin{scope}[x={(image.south east)},y={(image.north west)}]

			\node[text=blue] at (0.32,0.72) {$D_A^-$};
			\node[text=red] at (0.383,0.975) {$D_A^+$};

			\node  at (0.335,0.336)  {$p\times S^2$};
		\end{scope}
	\end{tikzpicture}
	\caption{\label{fig:twosheetedtubing}Tubing the discs $D_A^+$ and $D_A^-$ into two parallel copies of $p\times S^2$.}
\end{figure}

The two parallel sheets $D_A^\pm$ now intersect $S^2\times q$. We can remove the intersections of $D_A^\pm$ with~$D_B^\pm$  by performing a two-sheeted  Norman trick, tubing the two sheets  into~$S^2\times q$ to remove all self-intersections of $P$. See Figure \ref{fig:twosheetednorman} for a depiction of this. See the fifth Figure of \ref{fig:wh1hardcase} for a depiction of the resulting discs. Another way to see this is to view~$D_A^\pm$ as parallel copies of~$D_A$ which intersects~$S^2\times q$ in a single point, and~$D_B^\pm$ as parallel copies of~$D_B$. Performing the Norman trick with~$S^2\times q$ to remove the single point of intersection between~$D_A$ and~$D_B$, then taking parallel copies of both yields the same result. We denote the resulting surface by~$P^\prime$. The sphere~$P^\prime$ is embedded, intersects $W$ in a single point, and is disjoint from~$A$ and~$B$.

\begin{figure}[ht]
	\begin{tikzpicture}
		\node[anchor=south west,inner sep=0] (image) at (0,0) {\includegraphics[width=0.8\textwidth]{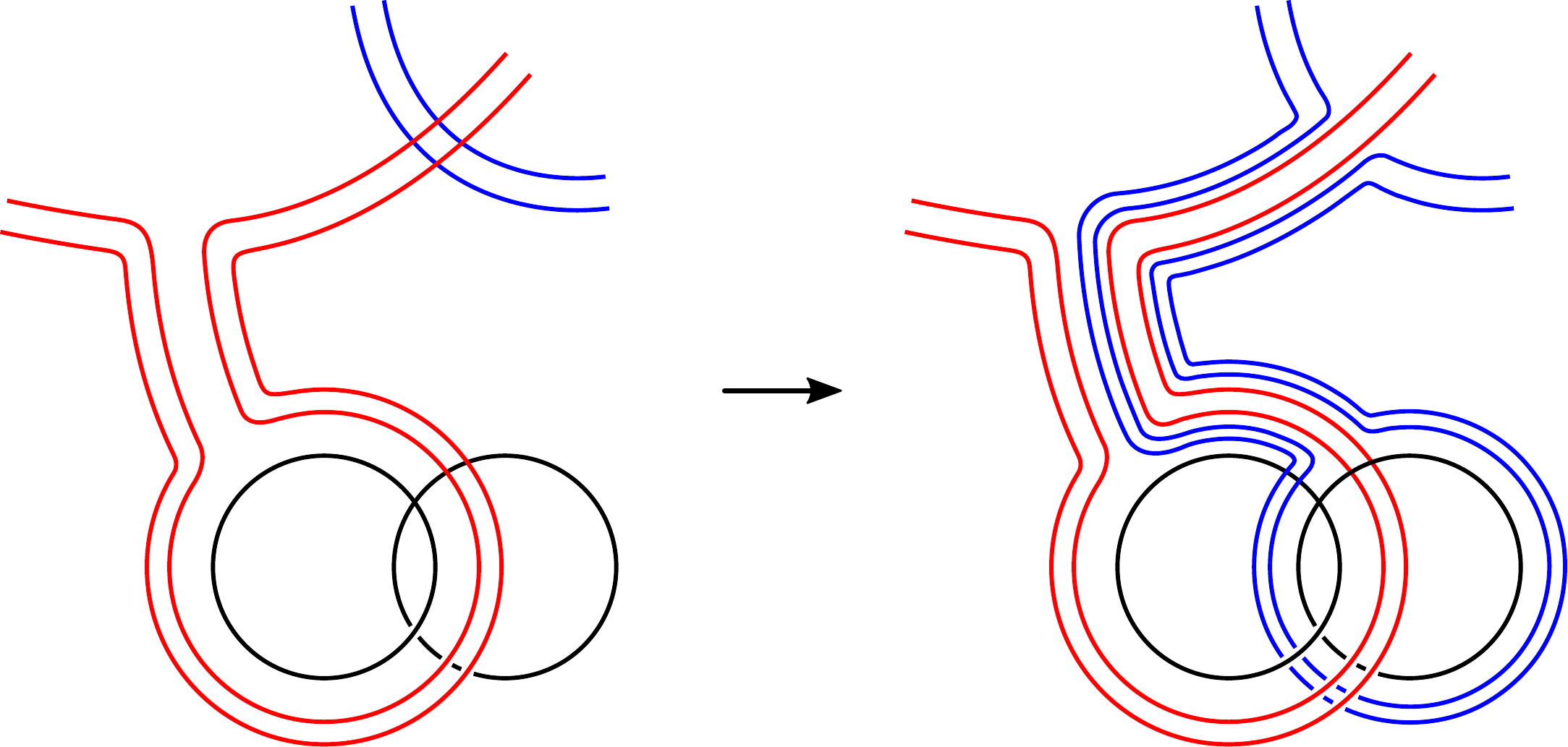}};
		\begin{scope}[x={(image.south east)},y={(image.north west)}]
			\node[text=red] at (0.03,0.775) {$D_A^+$};
			\node[text=red] at (0.03,0.625) {$D_A^-$};
			\node[text=blue] at (0.405,0.815) {$D_B^+$};
			\node[text=blue] at (0.405,0.665) {$D_B^+$};
			\node[anchor=west]  at (0.133,0.25)  {$p\times S^2$};
			\node[anchor=west]  at (0.388,0.25)  {$S^2\times q$};

		\end{scope}
	\end{tikzpicture}
	\caption{\label{fig:twosheetednorman}Left we see $D_A^+$ and $D_A^-$ after we tube them into parallel copies of $p\times S^2$. We also see the four intersection points with $D_B^+$ and $D_B^-$. Right we perform a two sheeted Norman trick, tubing $D_B^+$ and $D_B^-$ into parallel copies of $S^2\times q$, removing the intersections.}
\end{figure}

Note that we could have equivelently used $Y\setminus B^4= S^2\times S^2\setminus B^4$ to make $D_A$ and $D_B$ disjoint, then taken parallel copies of both.

Since $D_A^+$ and $D_A^-$ had opposite orientations, and so did $D_B^+$ and $D_B^-$, $P^\prime$ represents the same homotopy class in $V$ as $P$ (this is precisely why we went to so much trouble with the symmetric capping). Since $P$ is contained in the $S^2\times S^2$ summand of $\widehat{V}=\widehat{X}\# S^2\times S^2$, when we include $P$ into $ X\times I$, it is 0 in $\pi_2 (X\times I)=\pi_2 X$, hence $[P^\prime]=0\in\pi_2 (X\times I)$.

Since $P'$ is embedded with trivial normal bundle (note that all the spheres we tubed into had trivial normal bundle), and $P'$ intersects $W$ in a single point, we can use~$P'$ to perform the Norman trick and remove the self intersections of $W$. After doing so~$W$ is a framed embedded Whitney disc, and does not intersect $A$ or $B$ in its interior. Note also, since $P'$ is $0$ in $\pi_2 X$, the union of the Whitney discs $U\cup W$ (with the basepoint arc) still represents $\sigma$ in $\pi_2 X$.

We now use $W$ to perform a Whitney move, then cancel the resulting handles. As previously we see that $\gamma_C=\gamma$, $\sigma_C=\sigma$ and $s_C=1$, the latter because we performed a single boundary twist with $B$.
\end{proof}

\section{Pseudo-isotopy versus isotopy}\label{section:Pseudo-isotopy versus isotopy}
In this section we construct diffeomorphisms of 4-manifolds which are pseudo-isotopic but not isotopic to the identity.

To do this we make use of the previously described obstructions $\Theta$ and $\Sigma$. Given a diffeomorphism $f\in\Diff_{PI}(X,\partial X)$, we may try to obstruct it from being isotopic to the identity by picking a pseudo-isotopy $F\in \mathcal{P}$ from the identity to $f$, that is with~$F|_{X\times 1}=f$, and then evaluating $\Sigma(F)$ and $\Theta(F)$. However, $\Sigma(F)$ and $\Theta(F)$ clearly depend on the choice of $F$, not just on $f$. Given another choice of pseudo-isotopy $G\in\mathcal{P}$ from the identity to~$f$, we see that the composition $F\circ G^{-1}$ is a pseudo-isotopy fixing the entire boundary of $X\times I$. This motivates the following definition.

\begin{definition}
Let $F\in \mathcal{P}(X)$ be a pseudo-isotopy with $F_{X\times 1}=\id_X$. Note that $F_{X\times 0}$ and $F_{\partial X}$ are the identity by definition of $\mathcal{P}$ so in fact $F_{\partial(X\times I)}=\id_{\partial(X\times I)}$. We say that~$F$ is an \emph{inertial} pseudo-isotopy, and denote the set of inertial pseudo-isotopies by~$\mathcal{J}(X)\subset\mathcal{P}(X)$, or just $\mathcal{J}$ when $X$ is clear from the context.
\end{definition}

\begin{remark}
	Since inertial pseudo-isotopies fix the entire boundary of $X\times I$ in fact 
\[\mathcal{J}(X) = \Diff(X\times I, \partial(X\times I)).\]
\end{remark}

Let $f\colon X\rightarrow X$ be a diffeomorphism that is pseudo-isotopic to the identity. Define 
\begin{align*}
\Sigma\colon\pi_0\Diff_{PI}(X,\partial X)\rightarrow \Wh_2(\pi_1 X)/\Sigma(\mathcal{J})\\
\end{align*}
by saying $\Sigma(f)\vcentcolon=\Sigma(F)$, where $F$ is some pseudo-isotopy from $\id_X$ to $f$; by the discussion above this is independent of the choice of $F$. 

To show $\Sigma(f)$ is an invariant of the isotopy class of $f$ consider $g$ isotopic of $f$, and let $S$ be an isotopy from the identity to $f^{-1}\circ g$. Pick a pseudo-isotopy $F$ from the identity to $f$, then $F\circ S$ is a pseudo-isotopy from the identity to $g$, so 
\[\Sigma(f)=\Sigma(F)= \Sigma(F)+\Sigma(S)=\Sigma(F\circ S)=\sigma(g).\]
Note here that $\Sigma(S)=0$ since $S$ is an isotopy.

We wish to also define $\Theta(f)\in\Big(\Wh_1(\pi_1 X ;\mathbb{Z}_2 \times \pi_2 X)/\chi(K_3\mathbb{Z}[\pi_1 X])\Big)/\Theta(\mathcal{J}\cap\ker\Sigma)$. To define $\Theta$ for $f\in\ker\Sigma$ we first prove the below lemma.

\begin{lemma}\label{kersigmalemma}
Let $f$ be a self-diffeomorphism of $X$ which is pseudo-isotopic to the identity. If $\Sigma(f)=0\in\Wh_2(\pi_1 X)/\Sigma(\mathcal{J})$, then there exists a pseudo-isotopy $F$ from $\id_X$ to $f$ such that~$\Sigma(F)=0\in \Wh_2(\pi_1 X)$. 
\end{lemma}
\begin{proof}
Let $G$ be a pseudo-isotopy from the identity to $f$. Since $\Sigma(f)= 0 \in\Wh_2(\pi_1 X)/\Sigma(\mathcal{J})$, we have~$\Sigma(G)\in\Sigma(\mathcal{J})$. Hence there exists $S\in\mathcal{J}$ with $\Sigma(G)=\Sigma(S)$. Let $F=G\circ S^{-1}$, since $S$ is inertial $F$ is also a pseudo-isotopy from $f$ to the identity. Then $\Sigma(F)=\Sigma(G\circ S^{-1}) = \Sigma(G) -\Sigma(S)=0$ as required.
\end{proof}

Hence given $f\in\ker{\Sigma}$, Taking $F$ to be a pseudo-isotopy from the identity to $f$ with $\Sigma(F)=0$, we define.
\begin{align*}
\Theta\colon \ker\Sigma\subset\pi_0\Diff_{PI}(X,\partial X) &\longrightarrow \Wh_1(\pi_1 X ;\mathbb{Z}_2 \times \pi_2 X)/\chi(K_3\mathbb{Z}[\pi_1 X])/\Theta(\mathcal{J}\cap\ker\Sigma)\\
f &\longmapsto \Theta(F).
\end{align*}
Note that the $\ker\Sigma$ on the right refers to the subset of $\mathcal{P}$, while the $\ker\Sigma$ on the left refers to the subset of $\Diff_{PI}(X,\partial X)$.

To see that this is well defined, suppose $f$ and $g$ are isotopic, and let~$F,G\in\ker\Sigma$ be pseudo-isotopies from the identity to $f$ and $g$ respectively as in Lemma \ref{kersigmalemma}. Also let $S$ be an isotopy from the identity to $f^{-1}\circ g$. Then
\[\Theta(f)-\Theta(g)=\Theta(F)-\Theta(G)=\Theta(F) +\Theta(S)-\Theta(G)= \Theta(F\circ S\circ G^{-1})\in\Theta(\mathcal{J}\cap\ker\Sigma) \]

If $k_1 X=0$ we may also define $\Theta(f)$ when $\Sigma(f)\neq 0$, however in order to obtain something well defined we must define it in the group $\Theta(f)\in\Wh_1(\pi_1 X ;\mathbb{Z}_2 \times \pi_2 X)/\Theta(\mathcal{J})$. Note that we possibly quotient out by a larger subgroup. We conjecture that these groups are the same.
\begin{conjecture}
For $X$ a 4-manifold with $k_1 X=0$, $\Theta(\mathcal{J}(X)) = \Theta(\mathcal{J}(X)\cap\ker\Sigma)$
\end{conjecture}

The images $\Theta(\mathcal{J})$ and $\Sigma(\mathcal{J}\cap\ker\Sigma)$ are in general difficult to determine. For 4-manifolds of a certain form however we can say something about these two subgroups, in particular when $X=M^3\times [0,1]$ where $M$ is a $3$-manifold. 

\subsection{Duality formulae}

We recall the duality formulae and involutions of the Whitehead groups from Chapter VIII of \cite{HatcherWagoner} and Sections 4 and 5 of \cite{Hatcher}. We first give an analogue of turning a Morse function upside down for pseudo-isotopies. 

\begin{definition}
	Let $F$ be a pseudo-isotopy. Denote the reflection map on $X\times I$ which sends $(p, s)$ to $(p, 1-s)$ by $R$. We define the \emph{dual} pseudo-isotopy to $F$ to be
	\[\overline{F}=\left( \left(F|_{X\times 1}\right)^{-1}\times \id_I\right)\circ R\circ F\circ R.\]
	We note $\overline{F}$ sends $X\times i$ to $X\times i$ for $i\in\{0,1\}$ and that $(R\circ F\circ R)|_{X\times 0}=F|_{X\times 1}$, so
	\[\overline{F}|_{X\times 0}=\left(\left( \left(F|_{X\times 1}\right)^{-1}\times \id_I\right)\circ R\circ F\circ R\right)|_{X\times 0}= F|_{X\times 1}^{-1}\circ F|_{X\times 1}=\id_X.\]
	Further $(R\circ F\circ R)|_{\partial X\times I} = \id_{\partial X\times I}$ so $\overline{F}|_{\partial X\times I}=\id_{\partial X\times I}$, so $\overline{F}\in\mathcal{P}(X)$.
\end{definition}

Let $f_t$ be a path in $\mathcal{F}$ representing $F$; that is with $f_1=p\circ F$, $f_0=p$. Denote 
\[\overline{f_t} = R_I\circ f_t\circ R = 1 - f_t\circ R,\] where $R_I$ is the reflection map on $I$. It is clear that~$\overline{f_t}$ is a 1-parameter family for $\overline{F}$. Considering the index of the Morse critical points, we see that a critical point of index~$i$ becomes a critical point of index~$n-i$. 

One can view $\overline{f_t}$ as turning each Morse function upside down (which is usually considered to be taking $R_I\circ f$ for a Morse function $f$), then using the diffeomorphism $R$ to fix the ends, that is to make $\overline{f_t}(X\times i)=i$ for $i\in\{0,1\}$. Because $R$ is a diffeomorphism it does not change the index of the critical points.

Hatcher in \cite{Hatcher} in fact uses the path $1-f_t$ to compute $\Sigma (\overline{F})$ and $\Theta (\overline{F})$. Again because $R_{X\times I}$ is a diffeomorphism it does not change the index of handles, or the intersections between handles, and it is the identity on homotopy groups, so does not change the computation. We describe this computation below.

\subsubsection{Involution of $\Wh_1(\pi_1 X;\mathbb{Z}_2\times \pi_2 X)$}\label{wh_1_involution}

In this section we recall the involution
\begin{align*}
\bar{\cdot}\colon\Wh_1(\pi_1 X;\mathbb{Z}_2\times \pi_2 X)&\longrightarrow\Wh_1(\pi_1 X;\mathbb{Z}_2\times \pi_2 X)\\
\theta&\longmapsto\overline{\theta}
\end{align*}
described by Hatcher in \cite[Lemma 4.3]{Hatcher}. We will use this to relate $\Theta(\overline{F})$ and $\Theta(F)$.

We define maps
\begin{align*}
\bar{\cdot}\colon  \mathbb{Z}[\pi_1 X] & \longrightarrow  \mathbb{Z}[\pi_1 X]\\
					    \gamma & \longmapsto w_1^X(\gamma) \gamma^{-1}\;\text{ for }\;\gamma\in\pi_1 X\\
\end{align*}

and 
\begin{align*}
\bar{\cdot}\colon  (\mathbb{Z}_2\times \pi_2 X)[\pi_1 X] & \longrightarrow  (\mathbb{Z}_2\times \pi_2 X)[\pi_1 X]\\
					    (n,\sigma)\gamma & \longmapsto \left(n + w_2^X(\sigma), -w_1^X(\gamma)\sigma^{\gamma^{-1}}\right)\gamma^{-1}.\\
\end{align*}
For $\gamma\in\pi_1 X$, $\sigma\in\pi_2 X$ ,$n\in\mathbb{Z}_2$. Here $\sigma^{\gamma^{-1}}$ denotes the action of $\gamma^{-1}$ on $\sigma$. Clearly this defines these maps on the whole group by additivity.

We can define an involution on $GL\left(\left(\mathbb{Z}_2\times \pi_2 X\right)[\pi_1 X]\right)$
\[ I+A_{i,j}\longmapsto I+\overline{A_{j,i}}. \]
This in turn induces an involution on $\Wh_1(\pi_1 X;\mathbb{Z}_2\times \pi_2 X)$, which we denote $\theta\mapsto\overline{\theta}$. Considering the isomorphism in Corollary \ref{whtrace}, we can equivalently define the involution on $\Wh_1(\pi_1 X;\mathbb{Z}_2\times \pi_2 X)$ by considering it as a quotient of $(\mathbb{Z}_2\times \pi_2 X)[\pi_1 X]$, then the involution is induced directly by the involution on $(\mathbb{Z}_2\times \pi_2 X)[\pi_1 X]$.

\begin{remark}
Note that this involution depends not only on the groups $\pi_1 X$ and $\pi_2 X$ and the action of $\pi_1 X$ on $\pi_2 X$ (which is all that is needed to define the Whitehead groups) but also on the Stiefel-Whitney classes $w_1^X$ and $w_2^X$.
\end{remark}

We also wish to define an involution on the target of $\Theta$ when $k_1 X\neq 0$, namely the quotient~$\Wh_1(\pi_1 X;\mathbb{Z}_2\times\pi_2 X)/\chi(K_3 \mathbb{Z}[\pi_1 X])$, however in order to do so we would need the involution to fix $\chi(K_3 \mathbb{Z}[\pi_1 X])$; we conjecture that this is true.
\begin{conjecture}\label{overlinechi} $\overline{\chi(K_3 \mathbb{Z}[\pi_1 X])}=\chi(K_3 \mathbb{Z}[\pi_1 X])$.
\end{conjecture}
We suspect that one can define an involution on $K_3\mathbb{Z}[\pi_1 X])$ so the involution commutes with $\chi$, which would prove this conjecture, however we have not been able to resolve this. In order to avoid this problem, instead we define
\[\widehat{\chi}\vcentcolon = \chi(K_3 \mathbb{Z}[\pi_1 X])+\overline{\chi(K_3 \mathbb{Z}[\pi_1 X])}\]
and 
\begin{align*}
\widehat{\Theta}\colon\pi_0\mathcal{P}(X)&\longrightarrow \Wh_1(\pi_1 X;\mathbb{Z}_2\times\pi_2 X)/\widehat{\chi}\\
\end{align*}
by taking the composition of $\Theta$ and the quotient map
\[\Wh_1(\pi_1 X;\mathbb{Z}_2\times\pi_2 X)/\chi\rightarrow \Wh_1(\pi_1 X;\mathbb{Z}_2\times\pi_2 X)/\widehat{\chi}.\]
Now we can also define
\[\widehat{\Theta}\colon\pi_0\Diff_{PI}(X,\partial X)\longrightarrow \big(\Wh_1(\pi_1 X;\mathbb{Z}_2\times\pi_2 X)/\widehat{\chi}\big)/\widehat{\Theta}(\mathcal{J}\cap\ker\Sigma)\]
as we did for $\Theta$.

Clearly when $k_1 X=0$ or $\chi(K_3 \mathbb{Z}[\pi_1 X])=0$, $\Theta=\widehat{\Theta}$, and if Conjecture \ref{overlinechi} holds then $\Theta=\widehat{\Theta}$ regardless.

\subsubsection{Involution of $\Wh_2(\pi_1 X)$}\label{wh_2_involution}

We now recall the involution of $\Wh_2(\pi_1 X)$ defined by Hatcher-Wagoner; see \cite[ChapterVIII]{HatcherWagoner}. We use this to relate $\Sigma(\overline{F})$ and $\Sigma(F)$.

First we define an involution on the Steinberg group $\St(\mathbb{Z}[\pi_1 X])$ group denoted $a\mapsto \overline{a}$. We define the map on the generators by $s_{i,j}^\lambda\mapsto s_{j,i}^{\overline{\lambda}}$; note that $\lambda\in\mathbb{Z}[\pi_1]$ and that we use the involution defined in Section \ref{wh_1_involution}. Since this preserves the Steinberg relations it defines an involution of $\St(\mathbb{Z}[\pi_1 X])$.

One can easily define an involution on $E(\mathbb{Z}[\pi_1 X])$ sending $M_{i,j}$ to $\overline{M_{j,i}}$, and it is clear that this commutes with the map $\pi\colon\St(\mathbb{Z}[\pi_1 X])\rightarrow E(\mathbb{Z}[\pi_1 X])$, which means the involution is defined on $K_2(\mathbb{Z}[\pi_1 X])$.

Since $\overline{w_{i,j}^{\pm g}} = w_{j,i}^{\pm \overline{g}}$ the involution sends $W(\pm\pi_1 X)$ to itself; recall that $W(\pm\pi_1 X)$ is the subgroup of $\St(\mathbb{Z}[\pi_1 X])$ defined in Section \ref{thesteinberggroup}, and that $\Wh_2(\pi_1[X]) = K_2(\mathbb{Z}\pi_1[X])/W(\pm\pi_1 X)$. Hence it follows that the involution is defined in the quotient $\Wh_2(\pi_1[X])$ as required.

\subsubsection{Duality formulae}
We can now state the duality formulae of Hatcher and Wagoner.

\begin{proposition}[{\cite[Chapter VIII]{HatcherWagoner}, \cite[Duality Formula 4.4]{Hatcher}}]\label{duality-formulae}
Let $X$ be a manifold of dimension $n$. Then
\[\Sigma\left(\overline{F}\right)=(-1)^n \overline{\Sigma(F)}\]
and if $F\in\ker\Sigma$
\[\widehat{\Theta}\left(\overline{F}\right)=(-1)^n \overline{\widehat{\Theta}(F)},\]
where on the right hand side we use the involutions defined in Sections \ref{wh_1_involution} and \ref{wh_2_involution}.
\end{proposition}

Hatcher and Wagoner prove this for $\Theta$ which is defined when $k_1 X=0$, however taking the additional quotient by $\widehat{\chi}$ as in Section \ref{wh_1_involution} makes no difference to the proof. Indeed for $\Theta$ they consider the change of path of Morse functions from~$f_t$ to~$1-f_t$, and compare the change on the elements $\sigma_C$, $\gamma_C$, $s_C$ for each circle of intersection. They show that it corresponds to the involution we defined on $(\mathbb{Z}_2\times \pi_2 X)[\pi_1 X]$, then pass to the quotient $\Wh_1(\pi_1 X;\mathbb{Z}_2 \times \pi_2 X)$; passing to a further quotient by $\widehat{\chi}$ does not change the argument.

\subsection{Inertial pseudo-isotopies of \texorpdfstring{$M^3\times I$}{MxI}}\label{subsec:inertial-pis-in-product}

When $X=M\times I$ is the product of a 3-manifold $M$ and the interval, we can say more about the image of $\mathcal{J}$ under $\widehat{\Theta}$ and $\Sigma$.

We first note that there is a differential defined on $\Wh_2(\pi_1 X)\oplus \Wh_1(\pi_1 X;\mathbb{Z}_2\times\pi_1 X)/\widehat{\chi}$ given by $d_i (x) = x-(-1)^i\overline{x}$. Note that this is defined independently on the summands. We define $Z_i=\ker d_i$ and $B_i=\im d_{i+1}$, we will also split these out as 
\[B_i=B_i^2\oplus B_i^1\subset \Wh_2(\pi_1 X)\oplus \Wh_1(\pi_1 X;\mathbb{Z}_2\times\pi_1 X)/\widehat{\chi}\]
and 
\[Z_i=Z_i^2\oplus Z_i^1\subset \Wh_2(\pi_1 X)\oplus \Wh_1(\pi_1 X;\mathbb{Z}_2\times\pi_1 X)/\widehat{\chi}.\]

We recall the result of Hatcher that will allow us to bound the size of $\Sigma(\mathcal{J})$ and $\Theta(\mathcal{J}\cap\ker\Sigma)$ for $M\times I$.
\begin{proposition}{\cite[Lemma 5.3]{Hatcher}}\label{Bnprop}
	Let $M$ be an $(n-1)$-manifold, let $X=M\times I$, and let $F\in\mathcal{J}(X)$ be an inertial pseudo-isotopy. Then
	\begin{align*}
	\Sigma(J) &= (-1)^n\overline{\Sigma(J)}\\
	\end{align*}
	and if $J\in\ker\Sigma$,
	\begin{align*}
	\hTheta(J) &= (-1)^n\overline{\hTheta(J)}.
	\end{align*}
Hence we have
	\[\Sigma(\mathcal{J})\subset Z_n^2\]
	and
	\[\Theta(\mathcal{J}\cap\ker\Sigma)\subset Z_n^1.\]
\end{proposition}

Hatcher proves this for $\Theta$ but it also holds for $\hTheta$ when $k_1 X\neq 0$. We recall the proof below.
\begin{proof}[Proof of Proposition \ref{Bnprop}]
	Let $R$ denote the map on $X\times I$ sending $(x,s)$ to $(x,1-s)$. Since $X=M\times I$, there is also an involution on $X$ sending $(m,l)$ to $(m,1-l)$, which in turn induces an involution on $X\times I=M\times I\times I$, which we denote by $L$.
	
	We can define a further map on $M\times I\times I$ by rotating around the $I^2$ factor; we denote this rotation $R_\theta$ for $\theta\in[0,2\pi]$. Define
	\[\widetilde{J}= R_\pi\circ J\circ R_\pi.\]
	Noting that $R_\pi = R\circ L$ we have that $\widetilde{J}=L\circ\overline{J}\circ L$. Since conjugation by $L$ induces the identity on $\pi_* X$ and because $L$ is level preserving, conjugation by $L$ induces the identity on $\Wh_2(\pi_1 X)$ and $\Wh_1(\pi_1 X;\mathbb{Z}_2\times\pi_2 X)/\widehat{\chi}$, so $\Sigma\big(\widetilde{J}\big)=\Sigma\big(\overline{J}\big)$ and~$\Theta\big(\widetilde{J}\big)=\Theta\big(\overline{J}\big)$. 
	
	Further $\widetilde{J}$ is isotopic to $J$ in $\mathcal{P}(X)$ via the path
	\[J_\theta = R_\theta^{-1}\circ J\circ R_\theta\in\mathcal{P}(X)\]
where $\theta\in[0,\pi]$. Note that $J_0 = J$ and $J_\pi = \widetilde{J}$. Hence applying Proposition \ref{duality-formulae} we have $\Sigma(J)=\Sigma\big(\widetilde{J}\big)=\Sigma\big(\overline{J}\big)=(-1)^n \overline{\Sigma(J)}$ and when~$J\in\ker\Sigma$,~$\hTheta(J)=\hTheta\big(\widetilde{J}\big)=\hTheta\big(\overline{J}\big)=(-1)^n \overline{\Sigma(J)}$ as required. 
\end{proof}

\subsection{Diffeomorphisms of \texorpdfstring{$X^4\times I$}{X3xI}}

In this section we prove Theorem \ref{diffoffivemanifolds} which gives diffeomorphisms of the 5-manifold $X\times I$, for $X$ a 4-manifold, which are pseudo-isotopic but not isotopic for the identity. We do not use the results of this section elsewhere so a reader uninterested in 5-manifolds may skip this subsection entirely.

We recall the following result from Hatcher.
\begin{proposition}{\cite[Lemma 5.2]{Hatcher}}
Let $X$ be an n-manifold, $n\geq 5$. Then
\begin{gather*}
B_n^2\subset \Sigma(\mathcal{J})\text{ and}\\
B_n^1\subset \Theta(\mathcal{J}\cap \ker\Sigma).
\end{gather*}
\end{proposition}
Recall that we defined the subgroups $B_n^i$ at the beginning of Section \ref{subsec:inertial-pis-in-product}.

This uses the surjectivity of $\Sigma$ and $\Theta$. In 4-dimensions the following weaker statement still holds.
\begin{proposition}\label{Znprop}
Let $X$ be a 4-manifold. Then
\begin{gather*}
\{\theta +\overline{\theta}\;|\;\theta\in\Sigma(\mathcal{P})\}\subset \Sigma(\mathcal{J})\text{ and}\\
\{\theta +\overline{\theta}\;|\;\theta\in\hTheta(\ker\Sigma)\}\subset\widehat{\Theta}(\mathcal{J}\cap \ker\Sigma).
\end{gather*}
\end{proposition}

Combining this with Theorem \ref{Whoneimage} gives the following corollary.
\begin{corollary}
Let $X$ be a 4-manifold. Then
\begin{gather*}
\{\theta +\overline{\theta}\;|\;\theta\in\Xi\}\subset \widehat{\Theta}(\mathcal{J}\cap \ker\Sigma)
\end{gather*}
where we consider $\Xi$ in the quotient $\Wh_1 (\pi_1 X;\mathbb{Z}_2\times \pi_2 X)/\widehat{\chi}$.
\end{corollary}

The proof for 4-dimensions is the same as that in high dimensions. Let $X$ be an $n$-manifold, and let $F$ be a pseudo-isotopy of $X$. Let $p_1\colon X\times I\rightarrow X\times I$ be the map which sends $(x,s)$ to~$(x, s/2)$, and let $p_2\colon X\times I\rightarrow X\times I$ be the map which sends~$(x,s)$ to~$(x, 1- s/2)$. We form the \emph{double} of $F$, $2F\in\mathcal{J}(X)$ via
	\[(2F)(x,s) = 
	\begin{cases}
	p_1\circ F(x,2 s) & s\leq 1/2\\
	p_2\circ F(x, 2-2 s) & s>1/2\\
	\end{cases}
	\]
	Where $p_X$ is the projection of $X\times I$ onto $X$. That is we compress $F$ into the first half of the interval, and $\overline{F}$ into the second half. It is clear that $2F\in \Diff(X\times I, \partial(X\times I)) = \mathcal{J}(X)$. Now as in \cite[Corollary 4.5]{Hatcher} and \cite[Lemma 5.2]{Hatcher} we have
\[\Sigma (2F) = \Sigma(F)+(-1)^n\overline{\Sigma(F)}\]
and
\[\widehat{\Theta} (2F) = \widehat{\Theta}(F)+(-1)^n\overline{\widehat{\Theta}(F)}.\]
Using $2F$ we can see the above elements of $\Wh_2(\pi_1 X)$ and $\Wh_1(\pi_1 X;\mathbb{Z}_2\times\pi_1 X)$ in the image of $\mathcal{J}$ under $\Sigma$ and $\hTheta$ as required.

We can now prove the an analogue of \cite[Corollary 4.5]{Hatcher} for 5-manifolds.

\diffoffivemanifolds*

\begin{proof}
By Corollary \ref{whtrace} identify $\Wh_1(\pi_1 X; \mathbb{Z}_2\times \pi_2 X)$ with 
\[(\mathbb{Z}_2\times\pi_2 X )[\pi_1 X]/ \langle \alpha\sigma -\alpha^\tau \tau\sigma\tau^{-1}, \beta\cdot 1|\alpha,\beta\in\mathbb{Z}_2\times\pi_2 X,\; \tau,\sigma\in \pi_1 X\rangle \]
and $\Wh_1(\pi_1 X; \mathbb{Z}_2)$ with
\[\mathbb{Z}_2[\pi_1 X]/ \langle \sigma - \tau\sigma\tau^{-1}, 1| \tau,\sigma\in \pi_1 X\rangle= \bigoplus_{\Conj(\pi_1 X)^{\neq 1}}\mathbb{Z}_2  \]
where $\Conj(\pi_1 X)^{\neq 1}$ is the set of conjugacy classes of $\pi_1 X$ which are not the conjugacy class of $1$.

By Theorem \ref{Whoneimage} there exists a pseudo-isotopy $F$ of $X$, with $\Theta(F)=(1+\sigma)\gamma$. Considering only $\Wh_1(\pi_1 X; \mathbb{Z}_2)$, this means $F$ has $\Wh_1(\pi_1 X; \mathbb{Z}_2)$ invariant $\gamma$.

We proceed as in \cite[Corollary 4.5]{Hatcher}. Consider the double 
\[2F\in\Diff(X\times I,\partial(X\times I)).\]
If we consider this as a pseudo-isotopy of $X$, then it has $\Wh_1(\pi_1 X; \mathbb{Z}_2)$ invariant $\gamma+\gamma^{-1}$.

Suppose $2F\in\Diff(X\times I,\partial(X\times I))$ is isotopic to the identity. Then $2F$ would be isotopic as a pseudo-isotopy to the identity, so $\Theta(F)=0$. But the $\Wh_1(\pi_1 X; \mathbb{Z}_2)$ invariant of~$2F$ is~$\gamma+\gamma^{-1}$, and by the assumption that $\gamma$ and $\gamma^{-1}$ are not conjugate, $\gamma+\gamma^{-1}$ is not conjugate to $1$, so does not vanish in $\Wh_1(\pi_1 X; \mathbb{Z}_2)$, so this is a contradiction.

Hence it suffices to prove that $2F$, considered as a diffeomorphism of $X\times I$, is pseudo-isotopic to the identity. That is we must construct a pseudo-isotopy of $X\times I$ from the identity to $2F$, namely a diffeomorphism $D$ of $(X\times I)\times I$ with $D|_{(X\times I)\times 1}=F$ and $(X\times I)\times 0$ the identity. For this construction, the suspension $SF$ defined in \cite[Chapter 1, Section 5]{HatcherWagoner} suffices.
\end{proof}

\subsection{Diffeomorphisms of \texorpdfstring{$S^1\times S^2\times I$}{S1xS2xI}}

We begin with the example $ X = S^1\times S^2\times I$. We identify $\pi_2 X = \mathbb{Z}$ and $\pi_1 X$ with the multiplicative infinite cyclic group $\{t^n\;|\;n\in\mathbb{Z}\}$ so we may identify $(\pi_2 X)[\pi_1 X]$ with $\mathbb{Z}[t^\pm]$. We note that the action of $\pi_1 X$ on $\pi_2 X$ is trivial (one can see this in $S^1\times S^2$). 

Since by Section \ref{postnikovdef}
\[k_1 X\in H^3(\pi_1 X; \pi_2 X)=H^3(\mathbb{Z};\mathbb{Z})= H^3(S^1;\mathbb{Z})=0\] 
it follows that $k_1 X=0$. Hence we can consider
\[\Theta(\ker\Sigma)\subset \Wh_1(\pi_1 X;\mathbb{Z}_2\times\pi_2 X) = \Wh_1(\pi_1 X;\pi_2 X)\oplus\Wh_1(\pi_1 X;\mathbb{Z}_2).\]

By Proposition \ref{whtrace} we have
\begin{align*}
\Wh_1(\pi_1 X;\pi_2 X) & = \mathbb{Z}[t]/\langle n t^a -n^{t^b}t^b t^a t^{-b},\; n\cdot 1\;|\;a,b,n\in\mathbb{Z}\rangle \\
                       & = \mathbb{Z}[t]/\langle n\cdot 1\;|\;n\in\mathbb{Z}\rangle \\
					   & = \bigoplus_{i\in \mathbb{Z}^{\times}}\mathbb{Z}t^i.
\end{align*}

Similarly
\begin{align*}
\Wh_1(\pi_1 X;\mathbb{Z}_2) & = \mathbb{Z}_2[t]/\langle m t^a -n^{t^b}t^b t^a t^{-b},\; m\cdot 1\;|\;a,b\in\mathbb{Z},\;m\in\mathbb{Z}_2\rangle \\
                       & = \mathbb{Z}_2[t]/\langle m\cdot 1\;|\;n\in\mathbb{Z}\rangle \\
					   & = \bigoplus_{i\in \mathbb{Z}^{\times}}\mathbb{Z}_2 t^i.
\end{align*}

Hence 
\begin{align*}
\Wh_1(\pi_1 X;\mathbb{Z}_2\times\pi_2 X) & = \Wh_1(\pi_1 X;\mathbb{Z}_2)\oplus\Wh_1(\pi_1 X;\pi_2 X)\\
                                         & = \bigoplus_{i\in \mathbb{Z}^{\times}}\mathbb{Z}_2 t^i \oplus \bigoplus_{i\in \mathbb{Z}^{\times}}\mathbb{Z}_2 t^i\\
										 &= \bigoplus_{i\in \mathbb{Z}^{\times}}(\mathbb{Z}_2\times\mathbb{Z}) t^i.
\end{align*}

By Proposition \ref{duality-formulae} we also have that 
\[\Theta(\mathcal{J}((S^1\times S^2)\times I))\subset Z_4((S^1\times S^2)\times I) = \{\theta\in \Wh_1(\pi_1 X;\mathbb{Z}_2 X)\ |\; \theta = \overline{\theta}\}.\] 
For $(m,n)\cdot t^a \in \bigoplus_{i\in \mathbb{Z}^{\times}}(\mathbb{Z}_2\times\mathbb{Z}) t^i = \Wh_1(\pi_1 X;\mathbb{Z}_2\times\pi_2 X)$ note that 
\[\overline{(m,n)t^a}=(m+w_2^X(m),-w_1^X(m)n^{t^{-a}})t^{-a} = (m,-n)t^{-a}\]
since $w_1^X$ and $w_2^X$ are trivial for $S^1\times S^2\times I$. 

Since $a\neq 0$ we never have that $(m,n)t^a=(m,-n)t^{-a}$ so it is clear that 
\[Z_4((S^1\times S^2)\times I)=\{b\in \Wh_1(\pi_1 X;\mathbb{Z}_2\times\pi_2 X)\ |\ b=\overline{b}\} = \langle (m,n)t^a + (m,-n) t^{-a} \rangle.\]

Hence quotienting out by $Z_4(S^1\times S^2\times I)$ just identifies $\mathbb{Z}_2\times\mathbb{Z} t^a$ with $\mathbb{Z}_2\times\mathbb{Z} t^{-a}$, so we have a map

\[\Wh_1(\pi_1 X;\mathbb{Z}_2\times\pi_2 X)/\Theta(\mathcal{J}\cap\ker\Sigma)\xrightarrow{q}\Wh_1(\pi_1 X;\mathbb{Z}_2\times\pi_2 X)/Z_4(X) = \bigoplus_{i\in \mathbb{Z}_{>0}}(\mathbb{Z}_2\times\mathbb{Z}) t^i.\]

By Corollary \ref{Whoneimage} we have $F\in\ker\Sigma\subset\mathcal{P}(X)$ with $\Theta(F)=(0,n)t^a$, and so it follows we have  $ f\in\ker\Sigma \subset\Diff_{PI}(S^1\times S^2\times I,\partial(S^1\times S^2\times I))$ with $\Theta(f)=(0,n)t^a$. Let $p_2$ be the projection $p_2\colon\bigoplus_{i\in \mathbb{Z}_{>0}}(\mathbb{Z}_2\times\mathbb{Z}) t^i\rightarrow \bigoplus_{i\in \mathbb{N}}\mathbb{Z}$. The $S^1\times S^2\times I$ case of Theorem \ref{nontrivdiffeos} follows letting $K=\ker\Sigma$, and $\Theta'=p_2\circ q\circ\Theta$.

\nontrivdiffeos*

Note that as $w_2^X=0$ for $S^1\times S^2\times I$ we do not know how to realise the $\mathbb{Z}_2$ part of $\bigoplus_{i\in \mathbb{Z}_{>0}}(\mathbb{Z}_2\times\mathbb{Z}) t^i$ which arises from the framing.

\subsection{Diffeomorphisms of the connect sum of aspherical 3-manifolds times \texorpdfstring{$I$}{I}}

In this section we produce diffeomorphisms for $X=(M_1\# M_2)\times I$, where $M_i$ are closed, orientable, aspherical 3-manifolds. The condition of being aspherical is equivalent to being irreducible with an infinite fundamental group; this follows from the sphere theorem, see \cite[Theorem 4.3]{Hempel}. Many examples of such 3-manifolds exist, including $\Sigma_g\times S^1$ for $\Sigma_g$ a surface of genus $g>0$, as well as many hyperbolic 3-manifolds.

Note that aspherical 3-manifolds $M_i$ have torsion free fundamental group. To see this note that $M_i$ is $K(\pi_1 M_i,1)$ space. If $G\leqslant \pi_1 M_i$ is a cyclic subgroup, let $\widetilde{M_i}$ be the corresponding cover of $M_i$. Then $\widetilde{X}$ is a $K(G,1)$ space so $H_i(G,\mathbb{Z})=H_i(\widetilde{M_i},\mathbb{Z})=0$ for~$i>3$, which is only possible if $G$ is infinite; see \cite[Proposition 2.45]{HatcherAT}.

We first compute $\pi_1 X $ and $\pi_2 X $ along with the action. Let $M=M_1\#M_2$. It is clear that~$\pi_i X =\pi_i M $ for all $i$. It is also clear that $\pi_1 M=\pi_1 M_1\ast\pi_1 M_2$. To compute $\pi_2 M $ we consider the universal cover $p\colon\widetilde{M}\rightarrow M$. Writing $M=(M_1\setminus B^3)\cup_{S^2}(M_2\setminus B^3)$ we denote~$Y_i=p^{-1}(M_i\setminus B^3)$. Considering the action of $\pi_1 M $ on $\widetilde{M}$ we can make the following identifications
\begin{gather*}
Y_1 = \bigsqcup_{\pi_1 M/ \pi_1 M_1} \widetilde{(M_1\setminus B^3)},\\
Y_2 = \bigsqcup_{\pi_1 M/ \pi_1 M_2} \widetilde{(M_2\setminus B^3)},\\
Y_1\cap Y_2 =\bigsqcup_{\pi_1 M} S^2.
\end{gather*}


We write the Mayer-Vietoris sequence for $\widetilde{M}=Y_1\cup Y_2$ with coefficients in $\mathbb{Z}$

\[ 0 = H_3(\widetilde{M})\rightarrow H_2(Y_1\cap Y_2)\xrightarrow{(j_1,- j_2)} H_2(Y_1)\oplus H_2(Y_2)\xrightarrow{i_1 + i_2} H_2(\widetilde{M})\rightarrow H_1(Y_1\cap Y_2)=0.\]

Note that $H_3(\widetilde{M})=0$ as $\pi_1 M$ is infinite, so $\widetilde{M}$ is non compact. Using the identifications above we obtain

\begin{align*}
0 &\rightarrow \mathbb{Z}[\pi_1 M]\xrightarrow{(j_1, -j_2)} \left(\mathbb{Z}[\pi_1 M] \otimes_{\mathbb{Z}\pi_1 M_1}H_2(\widetilde{M_1\setminus B^3})\right)\oplus\left(\mathbb{Z}\pi_1 M \otimes_{\mathbb{Z}\pi_1 M_2}H_2(\widetilde{M_2\setminus B^3})\right)\\
&\rightarrow H_2(\widetilde{M})\rightarrow 0.
\end{align*}

To calculate $H_2(\widetilde{M_i\setminus B^3})$, note that since $M_i$ is aspherical, $H_2(\widetilde{M_i})=\pi_2 M_i = 0$. Note also that 
\[\widetilde{M_1\setminus B^3} = \widetilde{M_1}\setminus\bigcup_{g\in\pi_i M}g\widetilde{B^3}\]
for $\widetilde{B^3}$ some lift of the ball $B^3$ (which we note is also a ball). Hence $H_2(\widetilde{M_i\setminus B^3})$ is generated by the boundaries of these balls, and the only possible relation between these generators comes from taking the boundary of a 3-chain corresponding to the entire 3-manifold $\widetilde{M_i\setminus B^3}$. There is only such a 3-chain if $\widetilde{M_i\setminus B^3}$ is compact, which does not hold since $\pi_1 M_i$ is infinite. Hence $H_2(\widetilde{M_i\setminus B^3}) = \mathbb{Z}[\pi_1 M_i]$. Substituting this into the short exact sequence we obtain:
\[ 0 \rightarrow \mathbb{Z}[\pi_1 M]\xrightarrow{(\id,- \id)} \mathbb{Z}[\pi_1 M]\oplus \mathbb{Z}[\pi_1 M]\xrightarrow{i_1 + i_2} H_2(\widetilde{M})\rightarrow 0.\]
Hence,
\[\pi_2 (M_1\# M_2)= H_2(\widetilde{M_1\# M_2})=\mathbb{Z}[\pi_1 M] =\mathbb{Z}[\pi_1 M_1 \ast \pi_1 M_2] \]
with the action of $\pi_1 M$ on $\pi_2 M$ given by the obvious left multiplication. 

By Proposition \ref{whtrace} we have
\begin{align*}
\Wh_1(\pi_1 X;\mathbb{Z}_2\times\pi_2 X)
&= (\mathbb{Z}_2\times\mathbb{Z}[\pi_1 M]) [\pi_1 M]/\langle (m,n g) a -(m,(n g)^{b})b a b^{-1},\ (m,n g) 1\rangle \\
& = (\mathbb{Z}_2\times\mathbb{Z}[\pi_1 M]) [\pi_1 M]/\langle (m,n g) a -(m,n bg)b a b^{-1},\ (m,n g) 1)\rangle .\\
\end{align*}
Identifying $\Wh_1(\pi_1 X;\mathbb{Z}_2\times\pi_2 X)$ with this quotient of $(\mathbb{Z}_2\times\mathbb{Z}[\pi_1 M])[\pi_1 M]$ consider the surjective map,
\begin{align*}
q\colon\Wh_1(\pi_1 X;\mathbb{Z}_2\times \pi_2 X )&\longrightarrow \bigoplus_{S\in\Conj(\pi_1 X)^{\neq 1}} (\mathbb{Z}_2\times\mathbb{Z})S \\
q\colon(m,n g)a&\longmapsto \begin{cases}
(m,n)\Cl(a)&\text{ if }\Cl(a)\neq 1\\
0&\text{ if }\Cl(a)=1.
\end{cases}
\end{align*}
where $\Conj(\pi_1 X)^{\neq 1}$ denotes the set of conjugacy classes which are not the conjugacy class of $1$, and $\Cl(a)$ denotes the conjugacy class of $a$. To see this is well defined we note that it vanishes on both relations since
\begin{align*}
(m,n g) a -(m,n b g)b a b^{-1}\longmapsto &(m,n)\Cl(a)-(m,n)\Cl(bab^{-1})\\
&=(m,n)\Cl(a)-(m,n)\Cl(a)=0.
\end{align*}

Since $\pi_1 M=\pi_1 M_1\ast\pi_1 M_2$, and $\pi_1 M_i$ are infinite there are many conjugacy classes in $\pi_1 M$. 

Ultimately we wish to consider the quotient 
\[\big( \Wh_1(\pi_1 X;\mathbb{Z}_2\times\pi_2 X)/\widehat{\chi}\big) /\hTheta(\mathcal{J}(X)\cap\ker\Sigma).\]

Note by Proposition \ref{Bnprop} we have
\[\hTheta(\mathcal{J}(X)\cap\ker\Sigma)\leqslant Z_4^1(X)=\{\theta\in\Wh_1(\pi_1 X;\mathbb{Z}_2\times\pi_2 X)/\widehat{\chi} \;| \;\theta=\overline{\theta}\}.\]
Let 
\[\widecheck{Z}=\{\theta\in\Wh_1(\pi_1 X;\mathbb{Z}_2\times\pi_2 X) \;| \;\theta=\overline{\theta}\}\leqslant \Wh_1(\pi_1 X;\mathbb{Z}_2\times\pi_2 X)\]
and let $r$ be the quotient map $r\colon\Wh_1(\pi_1 X;\mathbb{Z}_2\times\pi_2 X)\rightarrow \Wh_1(\pi_1 X;\mathbb{Z}_2\times\pi_2 X)/\widecheck{Z}$. Then it is clear that 
\[\big(\Wh_1(\pi_1 X;\mathbb{Z}_2\times\pi_2 X)/\widehat{\chi}\big)/Z_4^1(X)=\big(\Wh_1(\pi_1 X;\mathbb{Z}_2\times\pi_2 X)/\widecheck{Z}\big)/r(\widehat{\chi}).\]

Our approach for the remainder of this section is to use the map
\[\big(\Wh_1(\pi_1 X;\mathbb{Z}_2\times\pi_2 X)/\widehat{\chi}\big) /\hTheta(\mathcal{J}(X)\cap\ker\Sigma)\rightarrow \big(\Wh_1(\pi_1 X;\mathbb{Z}_2\times\pi_2 X)/\widecheck{Z}\big)/r(\widehat{\chi})\]
to understand $\big(\Wh_1(\pi_1 X;\mathbb{Z}_2\times\pi_2 X)/\widehat{\chi} \big)/\hTheta(\mathcal{J}(X)\cap\ker\Sigma)$.

Again identifying $\Wh_1(\pi_1 X;\mathbb{Z}_2\times\pi_2 X)$ with a quotient of~$(\mathbb{Z}_2\times\mathbb{Z}[\pi_1 X])[\pi_1 X]$, first note that
\[\overline{(m,ng)a} = (m+w_2^X(ng),-w_1^X(a)(ng)^{a^{-1}})a^{-1}=(m,-n a^{-1}g)a^{-1}\]
since $w_1^X$ and $w_2^X$ are trivial in $M$ as it is an orientable 3-manifold, and so $w_1^X$ and $w_2^X$ are also trivial in $X=M\times I$. Noting that

\[q(\overline{(m,ng)a}) = (m, -n) \Cl(a^{-1}),\]
we define an involution on $\bigoplus_{S\in\Conj(\pi_1 X)^{\neq 1}} (\mathbb{Z}_2\times\mathbb{Z})S$ via
\begin{align*}
\bar{\cdot}\colon \bigoplus_{S\in\Conj(\pi_1 X)^{\neq 1}} (\mathbb{Z}_2\times\mathbb{Z})S &\rightarrow \bigoplus_{S\in\Conj(\pi_1 X)^{\neq 1}} (\mathbb{Z}_2\times\mathbb{Z})S \\
\bar{\cdot}\colon (m,n)\Cl(a)&\mapsto (m,-n)\Cl(a^{-1}).\\
\end{align*}

Now $q(\overline{a}) = \overline{q(a)}$, so $q(\widecheck{Z}) = \{ s\in \bigoplus_{S\in\Conj(\pi_1 X)^{\neq 1}} (\mathbb{Z}_2\times\mathbb{Z})S\ |\ s=\overline{s}\}$. Given a conjugacy class $S\in\Conj(\pi_1 X)$ we denote $\overline{S}=\Cl(a^{-1})$ where $S=\Cl(a)$. 

We claim that 
\[q(\widecheck{Z}) = \langle (m,n) S + (m, -n)\overline{S},\ (m, 0) P\ |\ m\in\mathbb{Z}_2,\ n\in\mathbb{Z}\ S,P\in\Conj(\pi_1 X),\ \overline{P}=P \rangle\]
Note that if $\Cl(a^{-1}) = \Cl(b^{-1})$ then $a^{-1} = r b^{-1} r^{-1}$, so $a = r b r^{-1}$ and $\Cl(a) = \Cl(b)$, so $\overline{S} =\overline{P}\implies S=P$.  To prove this claim, we can write any element in 
\[\bigoplus_{S\in\Conj(\pi_1 X)^{\neq 1}} (\mathbb{Z}_2\times\mathbb{Z})S\]
as $\sum_i (m_i, n_i) S_i,$ where the $S_i$ are distinct. If
\[\sum_i (m_i, n_i) S_i = \overline{\sum_i (m_i, n_i) S_i} = \sum_i (m_i, -n_i) \overline{S_i}\]
then since $\bar{\cdot}$ is injective, there is a permutation $\sigma$ which pairs $S_i$ with the unique $S_{\sigma(i)}$ such that $S_i=\overline{S_{\sigma(i)}}$. We hence also see that $m_i=m_{\sigma(i)}$ and $n_i=-n_{\sigma(i)}$; note that if~$i=\sigma(i)$ then $n_i=-n_{i}=0$. Hence we can rewrite the sum as

\begin{align*}
\sum_i (m_i, n_i)S &= \sum_{i,\ i=\sigma(i)}(m_i,0) S_i + \sum_{i,\ i<\sigma(i)} (m_k, n_k)S_k + (m_{\sigma(i)},n_{\sigma(i)})S_{\sigma(i)}\\
& = \sum_{i,\ i=\sigma(i)}(m_i,0) S_i + \sum_{i,\ i<\sigma(i)} (m_i, n_i)S_i + (m_i, - n_i)\overline{S_i} 
\end{align*}
which is the sum of generators of the required form, proving our claim.

Hence we can see that quotienting $\bigoplus_{S\in\Conj(\pi_1 X)^{\neq 1}} (\mathbb{Z}_2\times\mathbb{Z})S$ by $p(\widecheck{Z})$ identifies $(\mathbb{Z}_2\times \mathbb{Z}) S$ with $(\mathbb{Z}_2\times \mathbb{Z}) \overline{S}$ when $S\neq\overline{S}$, and kills the $\mathbb{Z}_2$ part when $S=\overline{S}$, that is
\[\Bigg(\bigoplus_{S\in\Conj(\pi_1 X)^{\neq 1}}(\mathbb{Z}_2\times\mathbb{Z})S\Bigg)/p(\widecheck{Z}) = \bigoplus_{\substack{S\in\Conj(\pi_1 X)^{\neq 1},\\ S=\overline S}}\mathbb{Z} S\oplus\bigoplus_{\substack{[S]\in \Conj(\pi_1 X)^{\neq 1}/\sim,\\S\neq\overline{S}}} (\mathbb{Z}_2\times \mathbb{Z})S\]
where $\sim$ is the equivalence relation on $\Conj(\pi_1 X)^{\neq 1}$ given by $S\sim\overline{S}$. Clearly $q$ induces a surjective map between the quotients
\[\tilde{q}\colon \Wh_1(\pi_1 X; \mathbb{Z}_2\times\pi_2 X)/\widecheck{Z}\rightarrow\bigoplus_{\substack{S\in\Conj(\pi_1 X)^{\neq 1},\\ S=\overline S}}\mathbb{Z} S\oplus\bigoplus_{\substack{[S]\in \Conj(\pi_1 X)^{\neq 1}/\sim,\\S\neq\overline{S}}} (\mathbb{Z}_2\times \mathbb{Z})S.\]

As previously, we will see that we are able to realise the $0\times\mathbb{Z}$ part of the second summand, as well as the $\mathbb{Z}$ part of the first summand, however, it will be easier to come up with conjugacy classes with $S\neq \overline{S}$). 

\subsubsection{Conjugacy classes in free products.}
In order to come up with suitable conjugacy classes, in this subsection we prove the following
\begin{proposition}\label{infiniteconjclasses}
For all $n\in\mathbb{N}$ there exists $a_n\in \pi_1 X = \pi_1 M_1\ast \pi_1 M_2$ such that $a_n$ and $a_m$ are not conjugate for $n\neq m$, and that $a_n$ and $a_m^{-1}$ are not conjugate~$\forall n,m\in\mathbb{N}$; in particular $a_n$ is not conjugate to its inverse. Hence there are infinitely many distinct equivalence classes of conjugacy classes $[S]\in\Conj(\pi_1 X)/\sim$ such that $S\neq \overline{S}$.
\end{proposition}
This gives us the following corollary.
\begin{corollary} The abelian group
\[\bigoplus_{S\in\Conj(\pi_1 X)^{\neq 1},\ S=\overline S}\mathbb{Z} S\oplus\bigoplus_{[S]\in \Conj(\pi_1 X)^{\neq 1}/\sim,\ S\neq\overline{S}} (\mathbb{Z}_2\times \mathbb{Z})S\]
has infinite rank, hence $\Wh_1(\pi_1 X; \mathbb{Z}_2\times\pi_2 X)/\widecheck{Z}$ also has infinite rank as $\tilde{q}$ is surjective.
\end{corollary}
Proposition \ref{infiniteconjclasses} follows easily from the following fact about free products of groups.
\begin{lemma}\label{freeprodconj}
Given groups $A$ and $B$, let $a\in A$, $b\in B$ with $a,b\neq 1$, then $(ab)^n\in A\ast B$ are distinct conjugacy classes in $A\ast B$ for all $n\in\mathbb{N}$. If additionally one of $a$ or $b$ is not of order two, then $(ab)^n\in A\ast B$ are distinct conjugacy classes in $A\ast B$ for all $n\in\mathbb{Z}$.
\end{lemma}
Taking $A=\pi_1 M_1$, $B\in\pi_1 M_2$, setting $a_n = (ab)^n$ proves Proposition \ref{infiniteconjclasses}; note that $\pi_1 M_i$ is torsion free so has no order two elements. We give a proof of Lemma \ref{freeprodconj} below.

\begin{proof}[Proof of Lemma \ref{freeprodconj}]
We define a length function $l:\Conj(A\ast B)^{\neq 1}\rightarrow\mathbb{N}$ by defining~$l(S)$ to be the minimum $n$ such that we can write $S = \Cl(c_1 c_2 \cdots c_n)$ such that $1\neq c_i\in A$ for $i$ odd and $1\neq c_i\in B$ for $i$ even, or $1\neq c_i\in A$ for $i$ even and $1\neq c_i\in B$ for $i$ odd. That is, it is the minimum $n$ such that we can give as an alternating product in elements of $A$ and $B$.

Note that in the free product $A\ast B$ it is clear that if $c_1 \ldots c_n = d_1 \ldots d_n$ for both $c_i$ and $d_i$ alternating in $A$ and $B$, then $c_i = d_i$ $\forall i$.

For $a_i\in A$, $b_i\in B$, we claim that $l(\Cl(a_1 b_1 a_2 b_2 \cdots a_n b_n))$ is $2n$. To prove this, suppose 
\[r a_1 b_1 a_2 b_2 \ldots a_n b_n r^{-1} = c_1 c_2 \ldots c_m\]
for $m< 2 n$ and $c_1 c_2 \ldots c_m$ an alternating sum in $A$ and $B$. We can also write $r=r_1 r_2\cdots r_k$ as an alternating sum in $A$ and $B$, so we have 
\[r_1 r_2 \ldots r_k a_1 b_1 a_2 b_2 \ldots a_n b_n r_k^{-1} r_{k-1}^{-1}\ldots r_1^{-1} = c_1 c_2 \ldots c_m.\]

Suppose first that $r_n\in B$. Then $r_1\ldots r_n a_1 b_1 a_2 b_2 \ldots a_n b_n$ is alternating. If $b_n r_k^{-1}\neq 1$ then 
\[r_1 r_2 \ldots r_k a_1 b_1 a_2 b_2 \ldots a_n (b_n r_k^{-1}) r_{k-1}^{-1}\ldots r_1^{-1}\]
is alternating of length $2n+2k-1$, so $m = n+ 2k-1 > n$ which is a contradiction. Hence $b_n r_k^{-1}= 1$. Cancelling $b_n r_k^{-1}= 1$ and repeating this argument on $r_k a_1 b_1 a_2 b_2 \ldots a_n r_{k-1}^{-1}\ldots r_1^{-1}$ we prove that $a_n r_{k-1}^{-1} = 1$.

When $k\leq 2n$, we can repeat this argument to prove that $r_k^{-1} r_{k-1}^{-1}\ldots r_1^{-1}$ cancels with the right $k$ terms of $a_1 b_1 \ldots a_n b_n$ at which point the remaining product $r_1 \ldots r_n a_1 b_2 \ldots$ is alternating and $m=2n$ which is a contradiction. 

When $k>2n$ we cancel every term in $a_1 b_1\ldots a_n b_n$ and are left with 
\[ r_1 r_2 \ldots r_k r_{k-2n}^{-1}\ldots r_1^{-1} .\]

Now if $r_k r_{k-2n}^{-1}\neq 1$, then $r_1 r_2 \ldots (r_k r_{k-2n}^{-1}) r_{k-2n-1}^{-1}\ldots r_1^{-1}$ is again alternating and has length~$k+k-2n-1\geq n$, so $m=k+k-2n-1$ which is a contradiction. Hence we must have $r_k r_{k-2n}^{-1}= 1$. 

Repeating this argument we see that all $k-2n$ terms of $r_{k-2n}^{-1}\ldots r_1^{-1}$ must cancel the right $k-2n$ terms of $r_1 r_2 \ldots r_k$ and are left with $r_1\ldots r_{2n}$ which is alternating of length $2n$ so again a contradiction.

If we suppose instead that $r_n\in A$, we see that $a_1 b_1 a_2 b_2 \ldots a_n b_n r_k^{-1} r_{k-1}^{-1}\ldots r_1^{-1}$ is alternating, and run the same argument proving that $r_k a_1 =1$ and so on, again arriving at a contradiction. 

Hence we cannot write $r a_1 b_1 a_2 b_2 \ldots a_n b_n r^-1 = c_1 c_2 \ldots c_m$ for $m\leq n$. This completes the proof of our claim that $l(\Cl(a_1 b_1 a_2 b_2 \cdots a_n b_n))$ is $2n$.

Now since $l(\Cl((ab)^n)) = 2n$, we have that $(ab)^n$ are in distinct conjugacy classes for all $n\in\mathbb{N}$ as required.

Suppose now that one of $a$ or $b$ is not order two. To prove that $(ab)^n$ are in distinct conjugacy classes for all $n\in\mathbb{Z}$ it is sufficient to prove that $(ab)^n$ is not conjugate to~$(ab)^{-n}$. 

Suppose that $r(ab)^n r^{-1}=(ab)^{-n} = (b^{-1} a^{-1})^n$ writing $r=r_1r_2\ldots r_k$ as an alternating product in $A$ and $B$, we have 
\[r_1r_2\ldots r_k (ab)^n r_k^{-1}\ldots r_2^{-1} r_1^{-1}=(b^{-1}a^{-1})^n\]

Suppose $r_k\in B$. Then we must have that $r_1\in B$ as the right hand term starts with an element of $B$. Hence $k$ is odd. Now the only way the alternating length of the left can agree with the alternating length of the right is if $r_k^{-1}\ldots r_2^{-1} r_1^{-1}$ cancels with the right hand $k$ terms of $(ab)^n$, hence $r_k^{-1} r_{k-1}^{-1}\ldots r_1 = b^{-1}a^{-1}b^{-1}\ldots a^{-1}b^{-1}$. Now we have
\[r_k (ab)^n r_k^{-1}\ldots r_2^{-1} r_1^{-1}= (ba)^n = (b^{-1}a^{-1})^n\]
which is only possible if $a=a^{-1}$ and $b=b^{-1}$ so $a$ and $b$ are order two which is a contradiction.

Suppose instead that $r_k\in A$. Then we see that $r_1\in A$ as the right hand term ends with a term of $A$. Hence again $k$ is odd. Similarly to the previous argument we see that $r$ must cancel with the left hand $k$ terms and again conclude that
\[r_k (ab)^n r_k^{-1}\ldots r_2^{-1} r_1^{-1}= (ba)^n\] 
again leading to a contradiction.
\end{proof}

\subsubsection{The rank of \texorpdfstring{$K_3\mathbb{Z}[\pi_1 M_1\# M_2]$}{K3Z[pi1X]}}
To complete our argument, we must quotient out by $r(\widehat{\chi})=r(\chi(K_3\mathbb{Z}[\pi_1 X])+\overline{\chi(K_3\mathbb{Z}[\pi_1 X])})$. We will prove that $K_3\mathbb{Z}[\pi_1 X]$ has finite rank, and so $r(\widehat{\chi})$ also has finite rank.

\begin{proposition}\label{KthreeM}
If $M=M_1\# M_2$ is a 3-manifold that is the connect sum of two aspherical 3-manifolds then $K_3\mathbb{Z}[\pi_1 M]$ has rank two.
\end{proposition}
\begin{proof}
In order to compute $K_3\mathbb{Z}[\pi_1 M]$, we first note that the Farrell Jones conjecture holds for 3-manifold groups; see \cite[Corollary 0.3]{FJfor3manifolds}. Since $\pi_1 M=\pi_1 M_1\ast\pi_1 M_2$ is torsion free there is an isomorphism
\[ H_n(B(\pi_1 M);\mathbf{K}(\mathbb{Z}))\xrightarrow{\cong} K_n\mathbb{Z}[\pi_1 M]\]
where $\mathbf{K}(\mathbb{Z})$ is the $K$-Theory spectrum of $\mathbb{Z}$, and $H_n(-;\mathbf{K}(\mathbb{Z}))$ is the generalised homology theory associated to this spectrum; see \cite{luckbook} for further details on this, and the Farrell Jones conjecture.

We proceed to calculate $H_3(B(\pi_1 M);\mathbf{K}(\mathbb{Z}))$. We first further simplify things by noting that since $M_i$ are aspherical we have that $B (\pi_1 M_i )= M_i$ and so
\[B( \pi_1 M )= B(\pi_1 M_1\ast\pi_1 M_2) = B (\pi_1 M_1)\bigvee B (\pi_1 M_2) = M_1\bigvee M_2.\]
Using the axioms for generalised homology theories we have that
\[H_n(B(\pi_1 M);\mathbf{K}(\mathbb{Z}))\cong H_n( M_1;\mathbf{K}(\mathbb{Z}))\oplus H_n( M_2;\mathbf{K}(\mathbb{Z}))\]
hence it will be sufficient to calculate $H_n( M_i;\mathbf{K}(\mathbb{Z}))$.

For any generalised homology theory there is an Atiyah-Hirzebruch spectral sequence with $E_2$ page given by
\[E_2^{p,q}=H_p(M_i;K_q(\mathbb{Z})) \]
where $H_p(M_i;K_p(\mathbb{Z}))$ is usual singular homology with coefficients, and $K_p(\mathbb{Z})$ are the algebraic K-theory groups of $\mathbb{Z}$; see for example \cite{davis-kirk} for details on the Atiyah-Hirzebruch spectral sequence. This spectral sequence converges to $H_n(B(\pi_1 X);\mathbf{K}(\mathbb{Z}))$ in the following sense; on each homology group there is a filtration
\[H_n(M_i;\mathbf{K}(\mathbb{Z})) = \mathcal{F}_0^n\supset \mathcal{F}_1^n\supset \ldots\supset \mathcal{F}_k^n = 0 \]
and the $E_\infty$ page of the spectral sequence gives
\[E_\infty^{n-j,j} = \mathcal{F}_j^n/\mathcal{F}_{j+1}^n.\]

We will use the first few terms of $K_q(\mathbb{Z})$, namely $K_0(\mathbb{Z})=\mathbb{Z}$, $K_1(\mathbb{Z})=\mathbb{Z}_2$, $K_2(\mathbb{Z})=\mathbb{Z}_2$ $K_3(\mathbb{Z})=\mathbb{Z}_{48}$; see for example \cite{weibel}.

Below we write out the second page of the spectral sequence; in red we highlight those terms with $p+q=3$, that will contribute to $H_3(M_i;\mathbf{K}(\mathbb{Z}))$, and we draw arrows where there is a non zero differential; note that for $p<0$ and $p>3$, $E_2^{p,q}= H_p(M_i; K_q(\mathbb{Z}))=0$. 

\begin{sseqdata}[ name = AHss-pg2,
xscale = 3,
homological Serre grading,
classes = { draw = none } ]

\class["\mathbb{Z}"](0,0)
\class["\mathbb{Z}_2"](0,1)
\class["\mathbb{Z}_2"](0,2)
\class["\mathbb{Z}_{48}", red](0,3)
\class["\vdots"](0,4)

\class["{H_1(M_i;\mathbb{Z})}"](1,0)
\class["{H_1(M_i;\mathbb{Z}_2)}"](1,1)
\class["{H_1(M_i;\mathbb{Z}_2)}", red](1,2)
\class["{H_1(M_i;\mathbb{Z}_{48})}"](1,3)
\class["\vdots"](1,4)

\class["{H_2(M_i;\mathbb{Z})}"](2,0)
\class["{H_2(M_i;\mathbb{Z}_2)}", red](2,1)
\class["{H_2(M_i;\mathbb{Z}_2)}"](2,2)
\class["{H_2(M_i;\mathbb{Z}_{48})}"](2,3)
\class["\vdots"](2,4)

\class["\mathbb{Z}", red](3,0)
\class["\mathbb{Z}_2"](3,1)
\class["\mathbb{Z}_2"](3,2)
\class["\mathbb{Z}_{48}"](3,3)
\class["\vdots"](3,4)


\d[green]2(3,0)
\d2(2,1)
\d2(3,1)
\d2(2,2)

\end{sseqdata}
\printpage[ name = AHss-pg2, page = 2 ] \quad

We note that above the first row all groups are torsion, and indeed finite! The only group which is non torsion and which contributes to $H_3(M_i;\mathbf{K}(\mathbb{Z}))$ is $E_2^{3,0}=\mathbb{Z}$. Since the differential $d_2^{3,0}\colon \mathbb{Z}\rightarrow E_2^{1,1}=H_1(M_i;\mathbb{Z}_2)$, drawn in green, is a map into a finite group, we see that we must have $\ker(d_2^{3,0})=\mathbb{Z}$, so $E_3^{3,0}=\mathbb{Z}$. 

Now consider $d_3^{3,0}\colon\mathbb{Z}\rightarrow E_3^{0,2}$. since $E_3^{0,2}$ is a quotient of $\mathbb{Z}_2$, it is either $\mathbb{Z}_2$ or $0$, so again we see that $\ker( d_3^{3,0})=\mathbb{Z}$ and that $E_4^{3,0}=\mathbb{Z}$. Since all differentials $d_n^{p,q}$ are zero for $n\geq 4$, $E_\infty^{p,q} = E_4^{p,q}$.

Hence $\mathcal{F}^3_0/\mathcal{F}^3_1 =\mathbb{Z}$, while $\mathcal{F}^3_1/\mathcal{F}^3_2=E_4^{2,1}$,  $\mathcal{F}^3_2/\mathcal{F}^3_3=E_4^{1,2}$,  $\mathcal{F}^3_3/\mathcal{F}^3_4=E_4^{0,3}$ are all torsion and finite. We also have  $\mathcal{F}^3_4/\mathcal{F}^3_5=E_4^{-1,4}=0$; since the filtration terminates with 0, it must be that $\mathcal{F}^3_4=0$. Hence $\mathcal{F}^3_3$ is finite, hence $\mathcal{F}^3_2$ is finite, and hence $\mathcal{F}^3_1$ is finite. 

We have a short exact sequence
\[ 0\rightarrow\mathcal{F}^3_1\rightarrow \mathcal{F}^3_0 = H_3(M_i;\mathbf{K}(\mathbb{Z}))\rightarrow \mathcal{F}^3_0/\mathcal{F}^3_1 = \mathbb{Z}\rightarrow 0.\]
as all the groups are abelian and $\mathcal{F}^3_0/\mathcal{F}^3_1=\mathbb{Z}$ is free abelian, the sequence splits and so
\[H_3(M_i;\mathbf{K}(\mathbb{Z}))=\mathbb{Z}\oplus \mathcal{F}^3_1.\]
Since $\mathcal{F}^3_1$ is finite, it follows that $H_3(M_i;\mathbf{K}(\mathbb{Z}))$ is rank one, and so $H_n(B\pi_1 M;\mathbf{K}(\mathbb{Z}))$ is rank two.
\end{proof}

\subsubsection{Diffeomorphisms of \texorpdfstring{$\Diff_{PI}((M_1\# M_2)\times I,\partial((M_1\# M_2)\times I))$}{DiffPIM1M2}} Recall the map
\[\tilde{q}\colon \Wh_1(\pi_1 X; \mathbb{Z}_2\times\pi_2 X)/\widecheck{Z}\rightarrow\bigoplus_{\substack{S\in\Conj(\pi_1 X)^{\neq 1},\\ S=\overline S}}\mathbb{Z} S\oplus\bigoplus_{\substack{[S]\in \Conj(\pi_1 X)^{\neq 1}/\sim,\\S\neq\overline{S}}} (\mathbb{Z}_2\times \mathbb{Z})S.\]
Denote the target of this map by $R$. We can induce a map on the quotients
\[\tilde{\tilde{q}}\colon \big(\Wh_1(\pi_1 X; \mathbb{Z}_2\times\pi_2 X)/\widecheck{Z}\big)/r(\widehat{\chi})\rightarrow R/\tilde{q}(r(\widehat{\chi}))\]
Recall that $\big(\Wh_1(\pi_1 X; \mathbb{Z}_2\times\pi_2 X)/\widecheck{Z}\big)/r(\widehat{\chi})=\big(\Wh_1(\pi_1 X; \mathbb{Z}_2\times\pi_2 X)/\widehat{\chi}\big)/Z_4^1(X)$ and that we have a map
\[\big(\Wh_1(\pi_1 X;\mathbb{Z}_2\times\pi_2 X)/\widehat{\chi}\big) /\hTheta(\mathcal{J}(X)\cap\ker\Sigma)\rightarrow \big(\Wh_1(\pi_1 X; \mathbb{Z}_2\times\pi_2 X)/\widehat{\chi}\big)/Z_4^1(X).\]
Composing this with $\tilde{\tilde{q}}$ we obtain a map
\[q'\colon\big(\Wh_1(\pi_1 X;\mathbb{Z}_2\times\pi_2 X)/\widehat{\chi} \big)/\hTheta(\mathcal{J}(X)\cap\ker\Sigma)\rightarrow R/\tilde{q}(r(\widehat{\chi})).\]
The left hand side is precisely the target of
\[\hTheta\colon\ker\Sigma\subset\Diff_{PI}(X,\partial X)\rightarrow \big(\Wh_1(\pi_1 X;\mathbb{Z}_2\times\pi_2 X)/\widehat{\chi} \big)/\hTheta(\mathcal{J}(X)\cap\ker\Sigma).\]
Taking the composition $q'\circ\hTheta$ gives $q'\circ\hTheta\colon \ker\Sigma\rightarrow R/\tilde{q}(r(\widehat{\chi}))$. By Theorem \ref{Whoneimage} the image of $q'\circ\hTheta$ contains 
\[\left(\rule{0cm}{0.9cm}\right. \bigoplus_{\substack{S\in\Conj(\pi_1 X)^{\neq 1},\\ S=\overline S}}\mathbb{Z} S\oplus\bigoplus_{\substack{[S]\in \Conj(\pi_1 X)^{\neq 1}/\sim,\\S\neq\overline{S}}} (0\times \mathbb{Z})S\left.\rule{0cm}{0.9cm}\right)/\tilde{q}(r(\widehat{\chi}))\leqslant R/\tilde{q}(r(\widehat{\chi})).\]
By Proposition \ref{infiniteconjclasses}
\[S = \bigoplus_{\substack{S\in\Conj(\pi_1 X)^{\neq 1},\\ S=\overline S}}\mathbb{Z} S\oplus\bigoplus_{\substack{[S]\in \Conj(\pi_1 X)^{\neq 1}/\sim,\\S\neq\overline{S}}} (0\times \mathbb{Z})S\]
has infinite rank. By Proposition \ref{KthreeM}, $\widehat{\chi}=\chi( K_3\mathbb{Z}[\pi_1 M]) + \overline{\chi( K_3\mathbb{Z}[\pi_1 M])}$ is at most rank four, and so $\tilde{q}(r(\widehat{\chi}))$ is at most rank four. Hence $S/\tilde{q}(r(\widehat{\chi}))$ has infinite rank, and so contains a subgroup isomorphic to $\bigoplus_\mathbb{N}\mathbb{Z}$. 
Letting $K=\ker\Sigma$ and letting $\Theta'$ be the composition of $q'\circ\Theta$ and projection onto the subgroup isomorphic to $\bigoplus_\mathbb{N}\mathbb{Z}$ yields the $(M_1\# M_2)\times I$ case of Theorem \ref{nontrivdiffeos}.

%

\end{document}